\documentclass[aos,preprint]{imsart}

\RequirePackage{amsthm,amsmath,amsfonts,amssymb,bbm}
\RequirePackage[numbers]{natbib}
\RequirePackage[colorlinks,citecolor=blue,urlcolor=blue]{hyperref}
\usepackage{graphicx,subfig}
\usepackage{mathtools,verbatim,enumitem,mathrsfs,wasysym}

\usepackage{color}

\startlocaldefs
\newcommand{\cB}{\mathcal{B}}
\newcommand{\cJ}{\mathcal{J}}
\newcommand{\cT}{\mathcal{T}}

\newcommand{\cF}{\mathcal{F}}
\newcommand{\cG}{\mathcal{G}}
\newcommand{\cX}{\mathcal{X}}

\newcommand{\cC}{\mathcal{C}}
\newcommand{\cH}{\mathcal{H}}

\newcommand{\cV}{\mathcal{V}}
\newcommand{\cU}{\mathcal{U}}

\newcommand{\Exp}{{\sf E}}

\newcommand{\Dro}{{\sf D}}

\newcommand{\Pro}{{\sf P}}
\newcommand{\Rro}{{\sf R}}

\newcommand{\bN}{\mathbb{N}}
\newcommand{\bS}{\mathbb{S}}

\newcommand{\bG}{{\mathbb{G}}}
\newcommand{\bQ}{{\mathbb{Q}}}

\newcommand{\bR}{{\mathbb{R}}}

\newcommand{\bY}{{\mathbb{Y}}}

\newcommand{\ind}[1]{\mathbbm{1}_{#1}}  
\newcommand{\mwd}[1]{\widetilde{#1}}  

\newtheorem{theorem}{Theorem}[section]
\newtheorem{lemma}{Lemma}[section]
\newtheorem{corollary}{Corollary}[section]

\newtheorem{remark}{Remark}[section]

\endlocaldefs

\begin{document}

\begin{frontmatter}
\title{Change Acceleration and Detection}
\runtitle{Change Acceleration and Detection}


\begin{aug}
\author[A]{\fnms{Yanglei}~\snm{Song}\ead[label=e1]{yanglei.song@queensu.ca}}
\and
\author[B]{\fnms{Georgios}~\snm{Fellouris}\ead[label=e2]{fellouri@illinois.edu}} 

\address[A]{Department of Mathematics and Statistics, Queen's University \printead[presep={,\ }]{e1}}

\address[B]{Department of Statistics, University of Illinois Urbana-Champaign \printead[presep={,\ }]{e2}}


\end{aug}

%

\begin{abstract}
A novel sequential change detection problem is proposed, in which the goal is to not only detect but also accelerate the change. Specifically, it is assumed that the sequentially collected observations are responses to treatments selected in real time. The assigned treatments determine the pre-change and post-change distributions of the responses and also influence when the change happens. The goal is to find a  treatment assignment rule and a stopping rule that minimize the expected total number of observations subject to a user-specified bound on the false alarm probability. The optimal solution is obtained under a general Markovian change-point model.  Moreover, an alternative procedure is proposed, whose applicability is not restricted to  Markovian change-point models and whose design requires minimal computation. For a large class of change-point models,  the proposed procedure is shown to achieve the optimal performance in an asymptotic sense. Finally,   its performance is found in simulation studies to be comparable to the optimal,  uniformly with respect to the error probability. 
\end{abstract}



\end{frontmatter}


\section{Introduction} 

The field of sequential change detection, which dates back to the works of Shewhart \cite{shewhart1931economic} and Page \cite{page1954continuous},   deals with the problem of detecting quickly a change in a system that is monitored in real time (see, e.g.,    \cite{tartakovsky2014sequential,veeravalli2014quickest}). 
The goal is to minimize some metric of the \textit{detection delay}, i.e., the number of observations between the change-point and the time of stopping, while controlling the frequency of  \textit{false alarms}, i.e., stopping before the change has occurred. 
 There are two main approaches in the literature regarding the mechanism that triggers the change. In the first one,  the change mechanism  is  treated as  unknown  \citep{lorden1971procedures,pollak1985,moustakides2008}, whereas in the second  the  change-point is assumed to be   a random variable with a known  (prior) distribution  \citep{shiryaev1963optimum,tartakovsky2005general,tartakovsky2017asymptotic}.   Despite their wide range of applications, neither of the two  approaches  is   suitable   when  the change is a latent event that 
 we can \textit{influence}.  
 This is the case, for example,  in an educational setup (see, e.g., \citep{baker2014educational,zhang2016smart,zhang2023statistical}), 
 where a change occurs when a student masters a certain skill, the mastery status is inferred by the student's responses to various tasks, and instructors attempt with an appropriate selection of tasks to not only detect the time of mastery but also to accelerate it.

\subsection{Proposed problem} 
 Motivated by the above considerations,  we propose a novel sequential change detection problem with an experimental design component. Specifically, we assume that the sequentially collected observations are responses to experimental design choices, to which we refer as  \textit{treatments}.  At any given time, we need to decide whether to stop and declare that the change has already occurred or to continue the process and select the treatment to be assigned next. These decisions can be made based on the current as well as previous responses.    Our main assumption is that the assigned treatments influence not only the speed of detection but also the change-point itself.  The problem  is determined by 
 how the responses depend on the assigned treatments before and after the change, i.e., a  \textit{response model}, and how the treatments influence the change, i.e., a \textit{change-point model}. Given such models, we aim to find a rule for sequentially assigning treatments, i.e.,  a \textit{treatment assignment rule},    and a rule for determining when to stop the process, i.e., a \textit{stopping rule},  in order to minimize the expected \textit{total} number of responses subject to a constraint on the probability of false alarm.   Since in the absence of a false alarm  the total number of observations up to stopping is equal to the time of change plus the detection delay, we refer to this problem as \textit{sequential change acceleration and detection}.  
 
 In the context of  educational applications, the response model corresponds to a {cognitive diagnosis model} 
\cite{Templin2006,templin2010diagnostic,wang2020using},  the change-point model to a {latent transition model} \cite{collins2009latent,li2016latent,kaya2017assessing,wang2018tracking,chen2018hidden,liang2023latent},   and the treatments to educational items, such as formative assessment, practice, intervention.
These references consider multiple latent attributes (i.e., skills) and estimate the model parameters based on data from many users.  In the present work, we assume that both models have been calibrated offline,  and we focus on the real-time instruction of a single attribute to a single user. Note that the online instruction is individualized since the developed models consider subject-specific variables that capture the characteristics of users. This separation of offline model estimation and online decision-making is realistic and practical, as 
there is a large amount of historical data and the duration of the learning process is relatively short on average.

\subsection{Literature review}
The problem in this work reduces to the classical  Bayesian sequential change detection problem 
\cite{shiryaev1963optimum,lai1998information,shiryaev2007optimal,tartakovsky2005general,tartakovsky2017asymptotic}  
when there is only a single available treatment. Indeed, in this case,  there is no treatment assignment rule and the goal is to find a stopping rule that minimizes the average detection delay subject to a control on the probability of false alarm.  More recently, there have been certain works that consider the (non-Bayesian) sequential change detection problem with adaptive treatments selection \cite{HeyTaj16,LiMeiShi15,XuMeiMous21,chaudhuri2024round,XuMeiPostUncert_SeqAn,fel_veer_moust}. 
However, in these works, the treatments can only affect the speed of detection, but not the change-point; as a result, there is no acceleration task.

Another related research area is {sequential design of experiments}, also known as {active hypothesis testing} or {controlled sensing}  \citep{chernoff1959,kiefer1963asymptotically,keener1984,naghshvar2013active,nitinawarat2013controlled},
which refers to sequential hypothesis testing with experimental design.  In this literature, however, the latent state of the system,   i.e., the true hypothesis,  does not change over time and cannot be influenced by design choices.  

We should also refer to the  {stochastic shortest path}  problem \citep{bertsekas1991analysis,Patek01,patek2007partially}, where the goal is to perform a series of actions to drive a Markov chain to a certain absorbing state with the minimum possible cost. In this problem, the target state is assumed to be observable, therefore there is no detection task. 

Finally,  we discuss the offline reinforcement learning literature \cite{levine2020offline,uehara2022review,prudencio2023survey},
where online policies, which induce treatment assignment rules, are learned directly from historical sequential data. This approach, however,  faces several challenges when applied to the current problem: 
(i) change-points are not directly observable in the historical data, (ii) the problem has an infinite horizon with no discounting factor (see Section \ref{subsec:auxilliary}), and (iii) it is not clear how the false alarm rate can be controlled.

\subsection{Main results}
First, 
we assume that the response and change-point models are completely specified. 
We refer to a change-point model as \textit{Markovian} if the conditional probability that the change occurs at a certain time depends on past treatments only via a {sufficient statistic of fixed dimension}.  For general Markovian change-point models, we obtain the optimal solution to the proposed problem by embedding it into the framework of {partially observable Markov decision processes} (PO-MDP) \citep{bertsekas1995dynamic,hernandez2012discrete}, and
we establish a structural result that the optimal stopping rule is to stop the first time the posterior odds (that the change has occurred) exceeds a deterministic function of the sufficient statistic.  This function, as well as the optimal assignment rule, are obtained numerically, thus, they  provide limited insight into how treatments are selected. Moreover,  the computation can be challenging when the dimension of the sufficient statistic is large or the response model is complicated.     These limitations of the PO-MDP approach motivate us to propose an alternative procedure that is applicable beyond Markovian change-point models and has an explicit and intuitive assignment rule whose design requires minimal computation.  

 The proposed procedure starts by repeatedly assigning a certain  \textit{block} of treatments, $\Xi_1$, until the posterior odds (that the change has happened) exceeds some pre-specified threshold $b_1$. When this happens, we apply repeatedly a different block of treatments, $\Xi_2$, until either the posterior odds exceeds a larger threshold $b_2\geq b_1$, or the likelihood ratio {statistic} \textit{against} the change based on the responses to block $\Xi_2$ crosses another threshold  $d$.  In the former case, we stop and declare that the change has happened, whereas in the latter we switch back to block $\Xi_1$ and repeat the process until termination.

We show that an appropriate selection of $b_2$ \textit{alone} guarantees the false alarm constraint, and we choose the remaining parameters based on a non-asymptotic upper bound for the expected sample size.
Specifically,  for any given blocks $\Xi_1$ and $\Xi_2$, we obtain a closed-form approximation for the values of   $b_1$ and $d$ that minimize the upper bound. The blocks $\Xi_1$ and  $\Xi_2$ can be supplied by practitioners, or they can be chosen ``optimally''  to approximately minimize the resulting upper bound for the selected values of $b_1$ and $d$.  For change-point models with limited memory, we show that the optimal choices can be obtained by solving deterministic MDPs with finite state spaces. The upper bound analysis is one of the main technical challenges, since the number of switches between the two blocks until termination is random, and two stopping rules are active simultaneously when $\Xi_2$ is assigned.


In order to  show the tightness of the non-asymptotic upper bound and justify the above design, we consider an asymptotic framework where the model specifications are parametrized by the tolerance level  $\alpha$ on the false alarm rate, and both the expected detection delay and the expected time of  change  
are allowed to diverge as $\alpha \to 0$,  possibly at different rates.   Under this framework, we establish a universal asymptotic lower bound on the expected sample size as $\alpha \to 0$, whose proof combines techniques from  Bayesian sequential change detection \cite{tartakovsky2005general}, sequential experimental design \cite{chernoff1959}, as well as a coupling argument to upper bound the time between stopping and the change-point on the event of a false alarm.  By comparing the lower and upper bounds, we show that the proposed procedure, with the optimal selection of blocks and thresholds, achieves the optimal performance asymptotically as $\alpha \to 0$,  in the sense that the ratio of its expected sample size over the optimal 
goes to one.

Further, we consider the setup where the response model and the change-point model have unknown parameters. These parameters are assumed to be random with a given joint prior distribution,
which represents the uncertainty quantification in their offline estimation based on historical data. In this context, we extend the previous procedure and show that the proposed rule enjoys a valid false alarm control and fast computation due to Monte Carlo approximation for the posteriors. 
Finally, in simulation studies, we show that the performance of the proposed procedure is close to the optimal, uniformly with respect to the false alarm probability, under completely specified Markovian models. In the case of unknown parameters, it suffers a mild loss compared to an oracle with knowledge of the parameters, when the variance of the prior distribution is small.

\subsection{Outline}  
The remainder of the paper is organized as follows.  We formulate the problem in  Section \ref{sec:ProbFormulation}
and obtain the optimal solution under a general  Markovian change-point model in Section \ref{sec:dp}.  We introduce the proposed scheme  in Section \ref{sec:mastery}, discuss its design in Section \ref{sec:design}, and
  establish the asymptotic optimality in  Section \ref{sec:asymptotic_framework}.
 We consider models with unknown parameters in Section \ref{sec:distr_uncertainty},  
  present simulation studies in Section \ref{sec:simulation}, and conclude in Section \ref{sec:discussion}. Proofs are presented in the appendix.

\section{Problem formulation}\label{sec:ProbFormulation}
In this section, we introduce the setup and main assumptions (Subsection \ref{subsec:setup}),   formulate the problem of interest (Subsection \ref{subsec:formulation}), and discuss  some statistics that are used throughout the paper (Subsection \ref{subsec:posterior_odds}). 

\subsection{Setup} \label{subsec:setup}
Let  $(\Omega, \cF,\Pro)$  be a probability space that hosts all random variables we define in this work.  Let $\{L_t, t=0,1 \ldots\}$ be a latent stochastic  process  such that $L_t=0$ for every $t < \Theta$ and $L_t=1$ for every $t \geq  \Theta$. That is,  $\Theta$ denotes the latent time at which there is a change in the state of a system of interest. 
At each time $t \in \bN \equiv \{1,2,\ldots\}$,  we select a treatment, $X_t$,  and observe a  response, $Y_t$. For each $t \in \bN$, we denote by  $\cF_t$ the $\sigma$-algebra generated by the first $t$  responses and treatments, i.e.,  
\begin{align*} 
\mathcal{F}_t &\equiv \sigma(X_s, Y_{s},\; 1 \leq s \leq t),
\end{align*} 
and we set $\cF_0 \equiv \{\emptyset, \Omega\}$. We denote by $\pi_0$ the probability that the change has occurred before observing any response, and by  $\Pi_{t}$   the conditional probability that the change happens at time $t$ given the current treatment and all past treatments and responses, i.e., 
\begin{align*}
\pi_0  &\equiv \Pro(L_{0} = 1);\quad
\Pi_{t} \equiv\Pro(L_{t} = 1 \, \vert \,  L_{t-1} = 0, \cF_{t-1}, X_t ),\quad t \in \bN.
\end{align*}  

We assume that treatments are selected sequentially based on the already acquired responses and that the latent change-point, $\Theta$, can be inferred from the observed responses and influenced by the assigned treatments. Figure \ref{fig:model} provides a graphical illustration of this setup, which we define formally 
next by specifying the treatment assignment model, the response model, and the change-point model.  

\begin{figure}[tbp!]
\centering
\includegraphics[width=0.65\textwidth]{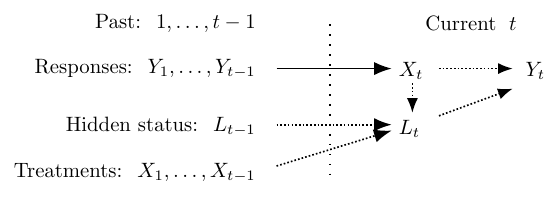}
\caption{An illustration of the proposed setup. The solid arrow is for the treatment assignment model, while the dotted arrows indicate the response model and the change-point model.} 
\label{fig:model}
\end{figure}

 \smallskip
 \noindent \textit{Treatment assignment model.} Consider $2 \leq K < \infty$ treatments. 
For each $t \in \bN$,   $X_t$ takes values in $[K] \equiv \{1,\ldots, K\}$, and $X_t = x$ for some $x \in [K]$ if and only if treatment $x$ is assigned at time $t$. For $t \geq 2$, we assume that 
 $X_t$  is a deterministic  function of  the first $t-1$ assigned treatments and observed responses,  $X_1, \ldots, X_{t-1}$ and $Y_{1}, \ldots, Y_{t-1}$ respectively,  whereas  $X_1$  is  deterministic.     That is,  $X_t$ is a $[K]$-valued, $\cF_{t-1}$-measurable  random variable for every $t \in \bN$.   We focus on deterministic assignment rules for simplicity, as the extension to randomized rules is straightforward and not necessary for our results. 

\smallskip
\noindent \textit{Response model.}  Without loss of generality, we assume that all responses take values in a common Polish space, denoted by $\bY$.  
For each treatment $x \in [K]$,  let $f_x$ and $g_x$ denote densities with respect to some $\sigma$-finite measure, $\mu$, on the Borel $\sigma$-algebra, $\cB(\bY)$, on $\bY$.  For each $t \in \bN$, we assume that the current response, $Y_t$,  is conditionally independent of the past given the current state of the system, $L_t$,  and the current treatment, $X_t$,  with density  $f_{X_t}$ before the change and  $g_{X_t}$  after the change. Formally,  for each $t \in \bN$ and  $B \in \cB(\bY)$ we have
\begin{equation*}
\Pro \left( Y_t \in B \, | \; \, X_t = x, \,L_t = i,\, \cF_{t-1},\, \{L_s\}_{0 \leq s \leq t-1} \right)  =
\begin{cases}
\int_B f_x \, d\mu, \quad i  = 0 \\
\int_B g_x \,d\mu, \quad i = 1
\end{cases}.
\end{equation*}

\noindent \textit{Change-point model.}  
For each  $ t \in \bN$, we assume that  the conditional probability of the change happening at time $t$ given the past observations and the event $\{L_{t-1} = 0\}$   is a function of only  the first $t$ assigned treatments, i.e., there exists a function $\pi_{t}: [K]^{t} \to [0,1)$ such that 
\begin{equation*}
\Pi_{t} =  \pi_{t}(X_1,\ldots, X_{t}).
\end{equation*}
Thus, the change-point model is determined by the prior probability $\pi_0\in [0,1)$   and the sequence of functions    $\{\pi_t, t \in \bN\}$. To emphasize that the distribution of the change-point,  $\Theta$,  depends on the treatment assignment rule,   $\cX \equiv \{X_t, t \in \bN \}$,  we   denote it by $\Theta_\cX$.

\smallskip
\noindent \textit{Examples.} We refer to a change-point model as  \textit{memoryless} if, for each $t \in \bN$,   $\Pi_t$ depends only on the  current treatment, $X_t$, i.e., if  there is a  link function $\Psi: [K] \to  [0,1)$ such that
\begin{align*}
\Pi_t  &= \Psi(X_t). 
\end{align*}

 We refer to a change-point model as  \textit{Markovian} if, for each $t\in \bN$,   
  $\Pi_t$ depends on the previous treatments, $X_1, \ldots, X_{t-1}$,  only via  a    {sufficient  statistic  of fixed dimension} $\kappa \geq 0$, i.e., if  there exist
 $\kappa$-dimensional random vectors,   $ \{\bS_{t}\}$, 
 such that  
\begin{align}\label{Markovian_change_point}
\begin{split}
\Pi_t  &= \Psi(X_t, \bS_{t-1}), \quad
 \bS_t = \Phi(X_t, \bS_{t-1} ) \;  \text{ for }\; t \in \bN, \quad \bS_0 = \zeta_0 \in \bR^\kappa,
 \end{split}
\end{align}
for some   link function $\Psi: [K] \times \bR^\kappa \to  [0,1)$  and transition function $\Phi : [K] \times  \bR^\kappa \to \bR^{\kappa}$. Note that $\Psi$ links $\Pi_t$ to the current treatment $X_t$ and the sufficient statistic $\bS_{t-1}$, while $\Phi$ updates the sufficient statistic, given $X_t$, from $\bS_{t-1}$ to $\bS_{t}$. 

Further, we say that a Markovian change-point model has \textit{finite memory}  if $\Pi_t$ is a function of the current and the  past $\kappa$ treatments, that is, if   the transition  function $\Phi$  is 
\begin{equation} \label{geometric_change_point}
\Phi(x, (x_\kappa,\ldots,x_1))
\; = \;(x,x_\kappa,\ldots,x_{2})   \;\;    \text{ for any } x,x_\kappa,\ldots x_1 \in [K],
\end{equation}
and the sufficient statistic  is  $\bS_{t} = (X_{t},\ldots,X_{t-\kappa+1})$ for $t \in \bN$.   Note that we may view the memoryless model as a finite memory model with $\kappa = 0$.

In the main text, we use the Markovian model and the finite memory model as running examples. In Appendices \ref{app:ex_exp} and \ref{app:ex_poly}, 
we present and study additional examples of change-point models, including a Markovian model with ``long memory'' and a non-Markovian one.

\subsection{Problem formulation}  \label{subsec:formulation} 
Our goal is to find a \textit{treatment assignment} rule and a \textit{stopping} rule so that the change occurs fast and we can detect its occurrence quickly and reliably.
Specifically, 
an admissible procedure is a pair 
$(\cX, T)$, where $\cX\equiv \{X_t, t \in \bN\}$ is  the sequence of assigned treatments  and   $T$ is an $\{\cF_t\}$-stopping time  at which we stop and declare that the change has occurred, i.e.,   $\{T=t\} \in \cF_t$ for every $t \geq 0$,
 where the $\sigma$-algebra $\cF_t$ is defined in the previous subsection. We  denote by $\cC$ the class of all such procedures and 
by  $\cC_{\alpha}$  the subclass of procedures  whose  probability of false alarm does not exceed $\alpha$, 
 i.e.,  
\begin{equation*}
\cC_{\alpha} \equiv 
\left\{
(\cX,T) \in \cC:\; \Pro(T < \Theta_{\cX}) \leq \alpha 
\right\},
\end{equation*}
where $\alpha \in (0,1)$ is a user-specified tolerance level. The objective  is to find  a procedure in $\cC_\alpha$ that minimizes the expected total number of observations, i.e., the one that attains
 \begin{equation} \label{infi}
\inf_{(\cX,T) \in \cC_{\alpha}} \;\;
\Exp \left[ T \right].
\end{equation}

\begin{remark}
For any procedure  $(\cX, T)$, the expected sample size  $\Exp[T]$ can be decomposed as follows:
 \begin{equation} \label{ess_decompose} 
\Exp[\Theta_{\cX}]  \,+\,   \Exp[\max\{T - \Theta_{\cX},0\}]  \,-\, \Exp[\max\{\Theta_{\cX} - T, 0\}].
\end{equation} 
The first term is the expected number of observations until the change,  and the second is the average detection delay. Note that $\max\{\Theta_{\cX} - T, 0\}$ is  non-zero only when a false alarm occurs, and thus the third term is
practically negligible for $(\cX, T) \in \cC_\alpha$ if $\alpha$ is small. Hence,  we refer to the problem in \eqref{infi} as  \textit{sequential change acceleration and detection}.
\end{remark}

\begin{remark}  It is tempting to consider the minimization of the sum of only the first two terms in  \eqref{ess_decompose}, i.e., $\Exp[\Theta_{\cX}]  +   \Exp[(T - \Theta_{\cX})^+]$, instead of  $\Exp[T]$. However, the latter is not only a simpler but also more rigorous problem formulation.  Indeed, while we define the treatment assignment rule  $\cX$ as an infinite sequence for mathematical convenience,  no treatment is actually assigned in practice after time $T$.  Therefore, even though the change-point $\Theta_{\cX}$ is well defined,  it does not have a physical interpretation on the whole sample space, as it is not actually ``realized'' on the event of a false alarm, i.e., when  $T < \Theta_{\cX}$. 
\end{remark}

Until Section \ref{sec:distr_uncertainty},  we assume that both the response and the change-point models are completely specified. This may be considered a realistic assumption when we have access to a large amount of historical data and the offline estimation has  negligible variance.    In Section \ref{sec:distr_uncertainty}, we consider the case that both models have unknown parameters, which are random with  a given joint prior distribution.


\subsection{Posterior odds}
\label{subsec:posterior_odds}
Given an assignment rule $\cX$,  
we denote by   $\Gamma_t$  the \textit{posterior odds} that the change has already occurred at time $t \geq 0$, i.e., 
\begin{equation}
\label{posterior_odds_prob}
\Gamma_t \; \equiv \frac{\Pro(L_{t} = 1 \, \vert \,  \cF_t)}{\Pro(L_{t} = 0 \, \vert \,  \cF_t)},
\end{equation}
with the understanding that   $\Gamma_0 \equiv    \pi_0 / (1-\pi_0)$.  Although $\{\Gamma_t\}$  depends on $\cX$, we do not emphasize this dependence to lighten the notation.

We conclude this section with two important lemmas,  whose proofs can be found in Appendix \ref{app:poster_odds}. Lemma \ref{lemma:recursive} shows 
that the posterior odds process admits a recursive form, which facilitates its computation and the analysis.    Lemma \ref{lemma:err_control} shows that the false alarm probability does not exceed $\alpha$ if the posterior odds process is at least $(1-\alpha)/\alpha$ at the time of stopping. 

\begin{lemma}\label{lemma:recursive}
For any  assignment rule $\cX$, for $t \in \bN$
\begin{equation*}  
\Gamma_t  \; = \;  (\Gamma_{t-1} + \Pi_t) \, 
  \frac{\Lambda_t}{1-\Pi_t}, \;\;\text{ where }\;\;
  \Lambda_t \equiv   \frac{g_{X_t}(Y_t)}{f_{X_t}(Y_t)}.
\end{equation*}
\end{lemma}

\begin{lemma} \label{lemma:err_control}
For any  assignment rule $\cX$ and  any  finite,  $\{\cF_t\}$-stopping time $S$,
\begin{equation*}
\Pro(S < \Theta_\cX \vert \cF_S)\;=\; {1}/{(1+ \Gamma_{S})}.
\end{equation*}
As a result, if $\Pro(\Gamma_{S} \geq b) = 1$ for some $b > 0$, then
$\Pro(S < \Theta_\cX) \leq  {1}/{(1 + b)}$.
\end{lemma}

\section{Optimal scheme in the Markovian case}\label{sec:dp}
In this section, we focus on the   \textit{Markovian} change-point model in \eqref{Markovian_change_point}.
We study an unconstrained sequential optimization problem
and discuss how its solution leads to the optimal procedure for \eqref{infi}.

\subsection{An auxiliary optimal stopping and control problem}
\label{subsec:auxilliary} We start by defining an auxiliary unconstrained sequential optimization problem,  in which the cost of each new observation is  $c >0$ and the cost of a  false alarm is $1$.  To be specific, when the prior belief $\Gamma_0$ and the initialization $\bS_{0}$ in \eqref{Markovian_change_point} take values $\gamma \in {\bR}_{+} \equiv [0,\infty)$ and $\zeta \in \bR^{\kappa}$, respectively,  we write $\Pro_{\gamma,\zeta}$ and $\Exp_{\gamma,\zeta}$ instead of $\Pro$ and $\Exp$.  We  define the {integrated cost} of a  procedure $(\cX,T) \in \cC$ to be $$\tilde{J}_c(\gamma,\zeta; \cX, T) \equiv c\, \Exp_{\gamma,\zeta}[T] + \Pro_{\gamma,\zeta}(T < \Theta_{\cX}),$$ 
and we  study the following problem:
\begin{equation}\label{J_star_cost}
\inf_{(\cX, T) \in \cC} \tilde{J}_c(\gamma,\zeta; \cX, T).
\end{equation}
 
We formulate this optimization problem in the language of {partially observable Markov decision processes} (PO-MDP) \citep{bertsekas1995dynamic,hernandez2012discrete}. 
Fix a procedure $(\cX,T) \in \cC$. For each $t \geq 0$, let $\Delta_t = \mathbbm{1}\{t > T\}$, where $\mathbbm{1}\{\cdot\}$ is the indicator function; that is, $\Delta_t$ indicates whether the process has terminated at time $t$. Thus, the triplet $(\Gamma_t,\bS_t, \Delta_t)$ of the posterior odds, the sufficient statistic in model \eqref{Markovian_change_point}, and the stopping indicator takes values in the \textit{state} space $\mathcal{S}  \equiv {\bR}_{+}\times \bR^{\kappa} \times \{0,1\}$. Further, we define the \textit{action} space $\mathcal{U} \equiv [K] \times \{0,1\}$, and we set  $${U_t \equiv  (U_t^{(1)}, U_t^{(2)})  \equiv (X_{t+1}, \mathbbm{1}\{T = t\}) \in \mathcal{U}, \quad 
 t \geq 0}.$$
That is, the first component {of $U_t$} identifies with the treatment assignment $X_{t+1}$, while the second  indicates whether the process is stopped at time $t$.

The transition dynamics of the states $\{({\Gamma}_t,\bS_t, \Delta_t), t \geq 0\}$ under actions $\{U_t, t \geq 0\}$ are as follows. Fix $t \geq 1$. 
Due to the definitions of the model \eqref{Markovian_change_point} and the stopping indicator, we have
\begin{align}\label{dp:transition_S_D}
\bS_{t} = \Phi(U_{t-1}^{(1)}, \bS_{t-1}), \quad
\Delta_t = 1  \quad \text{ if and only if }\quad 
 \Delta_{t-1} = 1 \text{ or } U_{t-1}^{(2)} = 1.
\end{align}
Further, in Appendix \ref{app:density_of_Y}, we show that the conditional $\mu$-density of $Y_t$ given  
$\cF_{t-1}$
is  $\phi(y;\Gamma_{t-1},\Pi_{t}, X_t)$, where  $\Pi_t = \Psi(X_t, \bS_{t-1})$ (see model \eqref{Markovian_change_point}) and the function $\phi$ is defined in Equation \eqref{app:Y_post_density} 
in Appendix \ref{app:density_of_Y}. 
Then,  by Lemma \ref{lemma:recursive},  for any Borel subset $B \subset  {\bR}_{+}$, 
\begin{align}\label{dp:transition_Gamma}
    \Pro\left(\Gamma_t \in B \, \vert \, \cF_{t-1} \right) = \int_{\bY} \mathbbm{1}\left\{  
  \frac{\Gamma_{t-1} + \Pi_t}{1-\Pi_t} \frac{g_{X_t}(y)}{f_{X_t}(y)}
  \in B \right\} \phi(y;{\Gamma}_{t-1},\Pi_{t},X_t) \; \mu(dy).
\end{align}
Since $X_t = U_{t-1}^{(1)}$, the right-hand side above only depends on the state and action at time $t-1$.

As a result, $\{({\Gamma}_t,\bS_t, \Delta_t), t \geq 0\}$ under actions $\{U_t, t \geq 0\}$ is a controlled  Markov decision process  \citep{bertsekas1995dynamic,hernandez2012discrete} with dynamics given by \eqref{dp:transition_S_D} and \eqref{dp:transition_Gamma}. Finally, we define the one-stage cost function $\tilde c: \mathcal{S} \times \mathcal{U} \to [0,\infty)$ as follows:
\begin{align*}
\tilde c({\Gamma}_t, \bS_t, \Delta_t; U_t) = \mathbbm{1}\{\Delta_t = 0\} \left( c \,  \mathbbm{1}\{ U_t^{(2)} = 0\} + 
(1+\Gamma_t)^{-1}  \, \mathbbm{1}\{U_t^{(2)} = 1\} \right).
\end{align*}
By Lemma \ref{lemma:err_control}, the second term inside the parentheses is the $\cF_t$-conditional  false alarm rate if the process is stopped at time $t$. Then,  by definition,
for each $(\gamma, \zeta, \delta) \in \mathcal{S}$, the integral cost 
$\tilde{J}_c(\gamma,\zeta; \cX, T) $ is given by $J_c(\gamma,\zeta,0;\cX,T)$, where
$$
J_c(\gamma,\zeta,\delta; \cX, T) \equiv \sum_{t=0}^{\infty}\Exp\left[  \tilde c({\Gamma}_t, \bS_t, \Delta_t; U_t) \; \vert \; ({\Gamma}_0, \bS_0,  \Delta_0) = (\gamma,\zeta,\delta) \right].
$$ 
Thus,  to solve \eqref{J_star_cost}, it suffices to find  $(\cX,T)$, or equivalently actions $\{U_t:t \geq 0\}$, that achieves
$$
J_c^*(\gamma,\zeta,\delta) \equiv \inf_{(\cX,T) \in \cC} J_c(\gamma,\zeta,\delta; \cX, T).
$$

Next, we apply the value iteration algorithm from dynamic programming literature \cite{bertsekas2022abstract} to solve the above problem. 
We denote by $\cJ$ the space of all non-negative measurable functions $ J:\mathcal{S} \to [0,\infty]$ such that $J(\gamma,\zeta,1) = 0$ for any $\gamma,\zeta$. For each   action $u \in \mathcal{U}$,  {we} define {the operator}  $\cT_{c,u}: \cJ \to \cJ$ as follows
\begin{align*}
    \cT_{c,u}(J)(\gamma,\zeta,0) \equiv 
\tilde{c}(\gamma,\zeta,0; u) +
\Exp\left[ J(\Gamma_1, \bS_1, \Delta_1)\
\vert\ ({\Gamma}_0, \bS_0,  \Delta_0) = (\gamma,\zeta,0),\ U_0 = u \right].
\end{align*}
That is, given the current state $(\gamma,\zeta,0)$ and a cost function $J$, $T_{c,u}(J)(\gamma,\zeta,0)$ is the sum of the current cost and the {expected} one-step cost under the action $u$. Note that once the process has stopped, no further cost is incurred, that is, $\cT_{c,u}(J)(\gamma,\zeta,1) = 0$ for any $\gamma,\zeta$.
Further, {we} define the Bellman operator, which corresponds to the cost under the optimal decision over $\cU$, as follows:
$\cT_{c}: \cJ \to \cJ$, 
$$\cT_{c}(J) \equiv \min_{u \in \cU} \cT_{c,u}(J), \quad J \in \cJ.
$$


Since the cost function $\tilde{c}(\cdot)$ is non-negative, this Markov decision process enjoys the ``Monotone Increase'' structure \cite[Assumption I in Chapter 4.3]{bertsekas2022abstract} (see Appendix \ref{app:mi} for verification). As a result, by 
\cite[Proposition 4.3.3]{bertsekas2022abstract}, the optimal cost function  $J_c^*$  satisfies the  {Bellman equation}: $\cT_c(J_c^*) = J_c^*$. Further,   since $\cU$ is finite, the compactness assumption in \cite[Proposition 4.3.14]{bertsekas2022abstract} clearly holds, and thus  $J_c^*$ can be computed by repeatedly applying the operator  $\cT_c$ with the initial function being the constant zero function, i.e., 
\begin{equation*}
\lim_{t \to \infty} \cT_{c}^{\bigotimes t}(0)(\gamma,\zeta,\delta) = J_c^*(\gamma,\zeta,\delta)\;\; \text{ for all }  (\gamma,\zeta,\delta) \in \mathcal{S},
\end{equation*}
where $0$ is  the zero function in $\cJ$, and $\cT_{c}^{\bigotimes t}$  the operator on $\cJ$ obtained by composing $\cT_c$  with itself for $t$ times. 

Once $J_c^*$ is solved, a  procedure $(\cX^*_c, T_c^*)$ that  achieves \eqref{J_star_cost}  can be obtained as follows. For $t \geq 0$, if the process has not been terminated, the cost of stopping is $1/(1+ \Gamma_t)$, while the optimal cost with  assigning treatment $k \in [K]$ is 
$\cT_{c,u}(J_c^*)(\Gamma_t,\bS_t, 0)$ with $u = (k,0)$. Thus, we stop  the first time that the  cost due to stopping does not exceed  the  {optimal cost}, and otherwise 
select the treatment that minimizes the expected future cost  \citep{bertsekas1995dynamic,hernandez2012discrete}, i.e.,  
\begin{align}\label{dp_optimal}
\begin{split}
    &T_c^* = \inf  \left\{t \geq 0: 1/(1 + {\Gamma}_t) \leq J_c^*({\Gamma}_t,\bS_t,0) \right\},\\
&X_{t+1,c}^{*} = \arg \min_{k \in [K]}\; {\cT}_{c,(k,0)}(J_c^*)({\Gamma}_{t}, \bS_{t},0).
\end{split}
\end{align}

 The next theorem, whose proof can be found in Appendix \ref{app:proof_structural},  establishes the following structural result: $T_c^*$ is the first time $t$ that the posterior odds ${\Gamma}_t$ exceeds a  deterministic function of the sufficient statistic $\bS_t$.

\begin{theorem}\label{dp_structural}
Assume the change-point model is Markovian as defined in \eqref{Markovian_change_point}, and consider the procedure 
$(\cX^*_c, T_c^*)$  defined in    \eqref{dp_optimal}, that solves the auxiliary optimization problem     \eqref{J_star_cost} for some $c > 0$. Then there exists  a function $b_c:\bR^{\kappa} \to \bR_{+}$ such that
$$
T_c^* \;= \; \inf\{\,t \geq 0: {\Gamma}_t \geq b_c(\bS_t)\,\}.
$$
In the special case of the memoryless model (i.e., $\kappa=0$), the function $b_c$ is a constant. 
\end{theorem}

 The computation of  $(\cX_c^*, T_c^*)$ requires  repeated application of the Bellman operator $\cT_c$, for which  we need to discretize the state space ${\bR}_{+}\times \bR^{\kappa}$ and  use interpolation in evaluating the expectation in $\cT_{c,u}$ for $u \in \cU$. This can become   challenging when the dimension of the sufficient statistic, $\kappa$, is large, or when the density $\phi$ in \eqref{dp:transition_Gamma}    has a complicated form.

\subsection{Optimality for the original problem}\label{subsec:limitations}

By definition, the optimal pair $(\cX_c^*,T_c^*)$ for the auxiliary optimization problem \eqref{J_star_cost} also solves the original problem
\eqref{infi} for a given tolerance level $\alpha \in (0,1)$ if  we can set $c=c({\alpha})$, where $c({\alpha})$ is such that the false alarm constraint is satisfied with equality, i.e.,
{$\Pro ( T_{c(\alpha)}^{*} < \Theta_{\cX^*_{c(\alpha)}})  = \alpha$.}
To determine $c(\alpha)$ we need to {compute}  $(\cX_c^*,T_c^*)$  for a {wide range} of values of $c$,  and then to evaluate \textit{for each of them}  the associated probability of false alarm via {simulation}. Clearly, this task can be   computationally demanding, especially when this is the case for the computation of 
$(\cX_c^*,T_c^*)$  for a single $c$.

\section{Proposed procedure} \label{sec:mastery}

In addition to the computational burden, the PO-MDP approach in Section \ref{sec:dp} has two inherent limitations. First, it is restricted to Markovian change-point models, thus it does not apply to, for example, the one in Appendix \ref{app:ex_poly}, whose transition probabilities do not admit sufficient statistics of  fixed dimension.   Second, even for the simple 
memoryless change-point model, there is no explicit form for the optimal assignment rule, thus,   it does not offer intuition about the selection of treatments. 

In this section,    we propose an alternative procedure with an explicit assignment rule, which is applicable beyond  Markovian change-point models and is based on a simple idea:  while there is not sufficient evidence that the change has occurred, the treatments should be selected to {facilitate}  the change, whereas when there is strong evidence that the change has occurred, the treatments should be selected to enhance its detection.

Specifically,  for $\ell \in \bN$,  we refer to an element of $[K]^\ell$,   i.e., an ordered list of  $\ell$ treatments, as a \textit{block} (of length $\ell$). While  the posterior odds process, $\{\Gamma_t\}$ in \eqref{posterior_odds_prob},   is below some threshold $b_1\geq 1$, we apply  repeatedly a certain block of  treatments,  $\Xi_1$.  Once $\{\Gamma_t\}$ exceeds  $b_1$, we apply repeatedly a different  block of treatments,  $\Xi_2$, and  if   $\{\Gamma_t\}$  exceeds a higher threshold $b_2 \geq b_1$,  we stop and declare that the change has occurred. However, since the unobservable change may not have occurred when we switch to block $\Xi_2$, and may not happen while applying this block, we need a mechanism 
for switching back to  $\Xi_1$.  For this purpose, we use the likelihood ratio of the responses to block $\Xi_2$ for testing the hypothesis that the change has occurred. If this statistic becomes smaller than a threshold $1/d \leq 1$ \textit{before} $\{\Gamma_t\}$ exceeds $b_2$,  we switch back to block $\Xi_1$ and repeat the previous process.  Thus, the proposed procedure alternates between \textit{acceleration stages}, i.e., time periods in which block $\Xi_1$ is applied repeatedly, and \textit{detection stages}, i.e.,  time periods in which block $\Xi_2$ is applied repeatedly.  We refer to an acceleration stage together with its subsequent detection stage as a   \textit{cycle} and denote by $N$ the random number of {cycles} until stopping. 

We provide a  graphical illustration of the proposed algorithm in Figure \ref{fig:proc}. For the formal definition, we introduce additional notations. For an arbitrary block  $\Xi \in [K]^\ell$,  we denote by $|\Xi|$ its size, i.e., $|\Xi|=\ell$,  by   $\Xi(i)$ its $i^{th}$ component,  and  adopt the following cyclic convention:
\begin{equation*}
\Xi(j\ell + i) \;\equiv\; \Xi(i)\; \text{ for each integer } j \;\text{ and }\; 1 \leq i \leq \ell.
\end{equation*}
We set
$\Xi(i:j) \equiv (\Xi(i),\ldots,\Xi(j))$   for  arbitrary  integers $i,j$.   Moreover,    we define  the following random  times 
\begin{align} \label{def:detection_testing_rules}
\sigma(t; b) &\equiv \inf\{s \geq 0: \Gamma_{t + s} \geq b \},  
\quad \;\;
\tau(t;d) \equiv \inf \{s \geq 1:\; \prod_{j=1}^{s} 
 \Lambda_{t+j}  \; \leq 1/d \; \},  
\end{align}
where $\{\Gamma_t\}$ and $\{\Lambda_t\}$  are defined in \eqref{posterior_odds_prob} and Lemma \ref{lemma:recursive}
 respectively.  Thus, $\sigma(t; b)$ 
represents  the number of observations after time $t$  required by the posterior odds process   to cross some threshold $b$, and $\tau(t;d)$ the number of observations  {after time $t$} until  the
likelihood ratio statistics of the observations go under   $1/d$.

\begin{figure}[!tbp]
\centering
\includegraphics[width=0.8\textwidth]{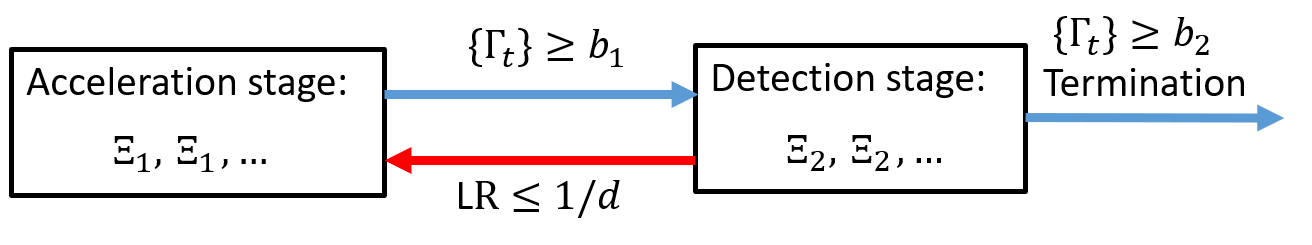}
\caption{An illustration of the proposed procedure. ``LR'' is short for the Likelihood Ratio statistic of the pre-change hypothesis against the post-change hypothesis.}
\label{fig:proc}
\end{figure}

\subsection{Definition}\label{subsec:proc_def}
We set $S_0 \equiv 0$, and {for each  $n \in \bN$ we   denote by  ${S}_n$   the   time of} the \textit{end} of the $n^{th}$ stage.    Then,    the assignment rule  $\mwd{\cX} \equiv \{\mwd{X}_t, t \in \bN\}$ is as follows:
 \begin{align*} 
\mwd{X}_t \equiv&
\begin{cases} 
\Xi_1\left( t-S_{2m-2} \right),  &\text{ if } \quad  t \in (S_{2m-2}, S_{2m-1}] \quad \; \text{ for some }  \; m \in \bN\\
\Xi_2\left( t-S_{2m-1} \right),  &\text{ if } \quad   t \in (S_{2m-1}, S_{2m}] \; \qquad  \text{ for some } \;  m \in \bN
\end{cases}.
\end{align*}
The times $\{S_n: n \in \bN\}$ are defined recursively. 
We stop the   acceleration stage of the $m^{th}$ cycle as soon as the posterior odds process  exceeds   $b_1$, that is,
$$
{S}_{2m-1} = {S}_{2m-2} + \sigma({S}_{2m-2};b_1), \text{ for } m \geq 1,
$$
and stop the  detection stage  of the $m^{th}$ cycle  
as soon as either the posterior odds process exceeds  $b_2 \geq b_1$, or the likelihood ratio of the observations \textit{during this stage} becomes smaller than or equal to   $1/d$, that is, 
$$
S_{2m} = {S}_{2m-1} + \min\{ \sigma({S}_{2m-1}; b_2),\  \tau({S}_{2m-1};d)\}, \text{ for } m \geq 1. 
$$
More compactly,  for each stage $n \in \bN$ we have 
\begin{equation*}
{S}_n = {S}_{n-1} + 
\begin{cases}
 \qquad \quad \sigma({S}_{n-1};b_1),  \qquad \qquad \qquad  n \text{ is odd}\\
\ \min\{\sigma({S}_{n-1}; b_2),\ \tau({S}_{n-1};d)\}, \quad n \text{ is even}
\end{cases}.
\end{equation*}
Finally,    we  terminate the whole process  at the end of the first cycle in which the detection is positive, that is,
 $\mwd{T} \equiv S_{2 N}$, where the number of cycles until stopping is
\begin{align*}
\begin{split}
N  &\equiv   \inf\left\{m \geq 1: \sigma({S}_{2m-1}; b_2) \leq \tau({S}_{2m-1};d)\right\} 
= \inf\{m \geq 1: \Gamma_{S_{2m}}  \geq b_2\}.
\end{split}
\end{align*} 
The  proposed  procedure $(\mwd{\cX}, \mwd{T})$  is illustrated   in Figure \ref{fig:one_run}.

\begin{figure}[!tbp]
\centering
\includegraphics[width=0.93\textwidth]{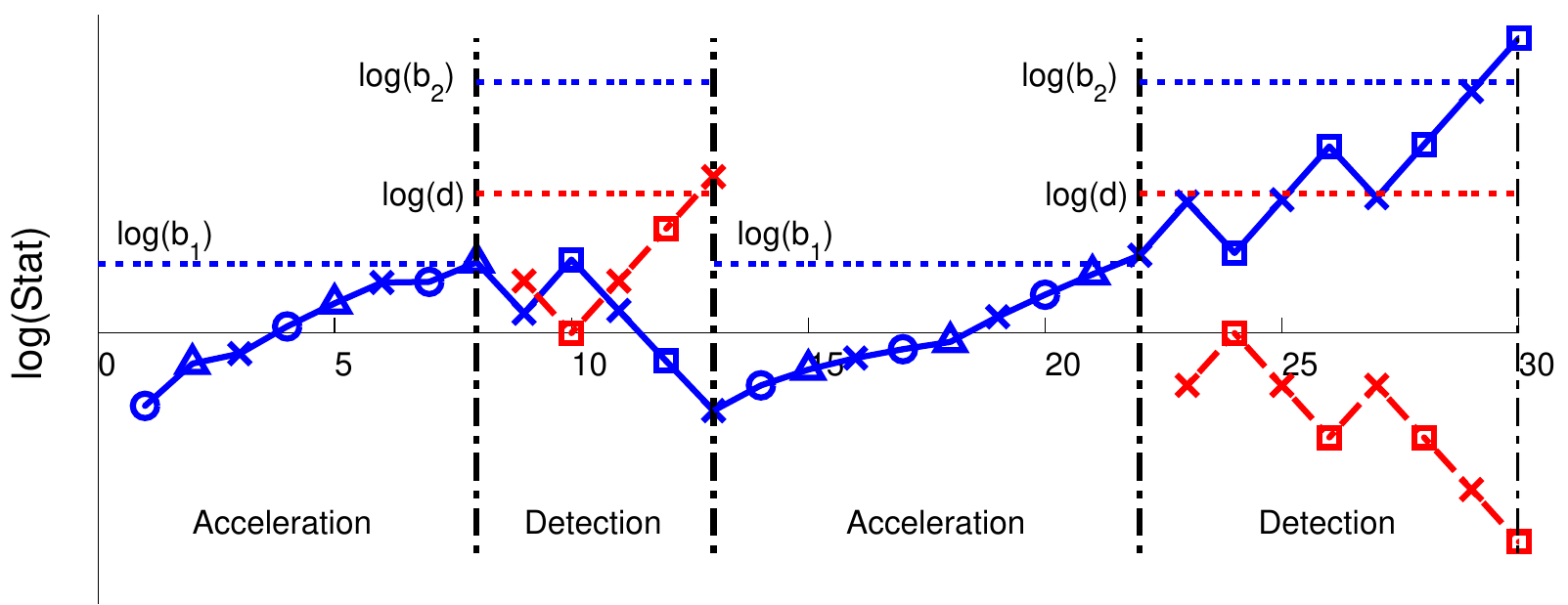}
\caption{A  simulation run of the proposed procedure. The treatments are represented by different shapes, 
$\Circle, \Delta,\times, \square$.  The acceleration block is 
 $\Xi_1 = [\Circle, \Delta, \times]$ and the detection block is 
$\Xi_2 = [\times, \square]$.  The solid line is the logarithm of the posterior odds process and the dashed line is the \textit{minus} of the logarithm of the likelihood ratio statistic in \eqref{def:detection_testing_rules}.
In an acceleration stage, we use block $\Xi_{1}$, wait until the posterior odds crosses $b_1$, and then switch to a detection stage. In a detection stage, we use block $\Xi_{2}$ and run both the detection rule with parameter $b_2$ and the testing rule with parameter $d$.
If the testing rule stops earlier, as in the first detection stage of this figure, we switch back to an acceleration stage; otherwise, we terminate the process, as in the second detection stage above.
}
\label{fig:one_run}
\end{figure}

\subsection{An extension}\label{subsection:minor_etension}
We assume above that the same block, $\Xi_1$, is applied in each acceleration stage and the same block, $\Xi_2$, in each detection  stage. With an appropriate design of this scheme, the change typically happens and is detected during the first cycle.  Thus, it can be useful to start in the \textit{first acceleration stage} with some alternative, finite, deterministic sequence of treatments before applying block $\Xi_1$.  
That is,   the assignment rule 
$\mwd{\cX}$ {can be} modified only from time $1$ to $S_1$ as follows: 
\begin{align}\label{mod_rule}
\mwd{X}_{t} = 
\begin{cases}
z_0(t) \qquad\quad\;  \text{ if } \quad 1 \leq t \leq \min\{t_0, S_1\}\\
\Xi_1(t-t_0) \quad \text{ if } \quad  t_0 < t \leq  S_1
\end{cases},
\end{align}
for some fixed integer $t_0 \geq 0$ and $z_0 \in [K]^{t_0}$. 
This modification is important for the finite memory model \eqref{geometric_change_point}; see Subsections \ref{subsec:example_finite_memory} and \ref{subsec:simulation_geo} for more details.



\section{Design of the proposed procedure}  \label{sec:design}
To implement the  procedure $(\mwd{\cX}, \mwd{T})$ introduced in  Section \ref{sec:mastery},  we need to  specify its thresholds, $b_1,b_2,d$, and   blocks, $\Xi_1, \Xi_2$. Since the posterior odds process at the time of stopping, $\mwd{T}$, is by definition larger than or equal to $b_2$,  it follows from Lemma \ref{lemma:err_control}  that  if we set 
\begin{align}\label{eq:bK_sel}
b_2= {(1-\alpha)}/{\alpha},
\end{align} 
then  $(\mwd{\cX},\mwd{T}) \in \cC_\alpha$.  That is, this choice of $b_2$ alone guarantees that the false alarm constraint is satisfied,  for any given $\alpha \in (0,1)$,  \textit{irrespectively of how  $\Xi_{1},\Xi_{2}$ and  $b_1, d$ are chosen}.   We select these ``free'' parameters based on a non-asymptotic upper bound on the expected sample size of this procedure, which is established in Subsection \ref{subsec:upper_bound}.  We obtain a closed-form expression for values of  $b_1, d$ that minimize approximately this upper bound for any \textit{given} blocks $\Xi_{1}, \Xi_{2}$ in Subsection \ref{subsec:thres}, and then select the blocks to minimize the resulting upper bound in Subsection \ref{subsec:block}.  This analysis is conducted under certain assumptions on the response and change-point models, which are described in Subsection \ref{subsec:assump}.

\subsection{Assumptions} \label{subsec:assump}
For each $x \in [K]$, we denote the  Kullback-Leibler divergences between the response densities $g_x$ and $f_x$ as:
\begin{align}\label{def:KL}
I_x \equiv \int_{\bY} \left( \log \frac{g_x}{f_x} \right)\, g_x\, d\mu,\qquad
J_x \equiv \int_{\bY} \left( \log \frac{f_x}{g_x} \right)\, f_x\, d\mu.
\end{align}
They are positive and finite under the following    assumption: for each $x \in [K]$,
\begin{align}
0 < V_x^{I}  \equiv \int_{\bY} \left( \log \frac{g_x}{f_x} - I_x\right)^2  g_x\, d\mu < \infty, \quad 
V_x^{J} \equiv \int_{\bY} \left( \log \frac{f_x}{g_x} - J_x\right)^2 f_x\, d\mu < \infty. \label{assumptions:response_KL}
\end{align}

Moreover,  we impose a stability assumption on the transition probabilities so that if a  block is assigned repeatedly after some time $t$, starting from any of its components, the transition probability eventually stabilizes in an average sense.
Specifically, we  assume that for each $\Xi  \in [K]^{\ell}$, there exists  a non-negative number  $d(\Xi)$  such that 
\begin{equation} \label{assumption_stability_pointwise}
\lim_{M \to \infty} \frac{1}{M} \sum_{s = 1}^{M} \omega 
\left( \pi_{t+s} \left(z, \Xi (j:
s+j-1) \right) \right)= d(\Xi), \quad \text{ with } \;\omega(x) \equiv |\log(1-x)|,
\end{equation}
for every $t \geq 0$, $z \in [K]^t$ and $j \in [\ell]$.
In fact, we require this convergence to be uniform, in the sense that for any $\epsilon > 0$ and    block $\Xi  \in [K]^{\ell}$, there exists  a positive integer $M(\Xi,\epsilon)$ such that for any $M \geq M(\Xi,\epsilon)$,
\begin{equation}\label{assumption_stability}
\sup_{t \geq 0,  z \in [K]^{t},  j \in [\ell]} \left\vert \frac{1}{M} \sum_{s = 1}^{M} \omega \left(
\pi_{t+s} \left(z, \Xi(j:
s+j-1) \right) \right) - d(\Xi)  \right\vert  \leq \epsilon.
\end{equation}
Note that condition \eqref{assumption_stability_pointwise} is  assumed in \cite[see equation (2.6)]{tartakovsky2005general}, where the  Bayesian sequential change detection problem is considered in the absence of design.

\subsection{An  upper bound} \label{subsec:upper_bound}
For any  block  $\Xi \in [K]^{\ell}$, we define  the average information numbers and second moments over the treatments in block $\Xi$ as follows:
$$I(\Xi) \equiv {\frac{1}{\ell} } \sum_{i=1}^{\ell} I_{\Xi(i)}, \;\; \;
J(\Xi) \equiv {\frac{1}{\ell} } \sum_{i=1}^{\ell} J_{\Xi(i)},
\;\;\;
V^{I}(\Xi) \equiv {\frac{1}{\ell} } \sum_{i=1}^{\ell} V^{I}_{\Xi(i)}, \;\;\;
V^{J}(\Xi) \equiv {\frac{1}{\ell} } \sum_{i=1}^{\ell} V^{J}_{\Xi(i)}.
$$

Further, we define the \textit{adjusted information number} of a block $\Xi$, which combines  information from  both the response and change-point models,  as follows:
\begin{equation}\label{adjust_info_num}
\Dro(\Xi) \equiv I(\Xi) + d(\Xi).
\end{equation}

We denote by  $\lambda(\Xi)$ the expected time of change
when   block $\Xi$ is  repeatedly assigned \textit{from the beginning}. It is given by the following formula: 
\begin{align}\label{block_expected_change}
\lambda(\Xi) &\equiv (1 - \pi_0)\left(1 +  \sum_{t=1}^{\infty}
\prod_{s=1}^{t}\left(1-\pi_s(\Xi(1:s))\right) \right).
\end{align}

Further, we denote by  $\widetilde{\lambda}(\Xi)$ an \textit{upper bound} on the expected time of change if block $\Xi$ is repeatedly assigned after some initial period. Specifically, for each integer $m \geq 0$ and a sequence of treatments $z \in [K]^{m}$ of length $m$, define
$$
\widetilde{\lambda}(\Xi; m, z) \equiv 1+ \sum_{t=1}^{\infty}
\prod_{s=1}^{t}\left(1-\pi_{m+s}(z, \Xi(1 : s))\right),
$$
which is the expected duration between the time $m$ and the change-point  if the initial $m$ treatments are given by $z$ and the block $\Xi$ is assigned afterwards. Then we define
\begin{align}\label{block_upper}
\widetilde{\lambda}(\Xi)  &\equiv \sup_{m \geq 0, z \in [K]^{m}} \widetilde{\lambda}(\Xi; m, z),
\end{align}
which is the worst case over all possible initial periods.

Finally,  we denote by $\zeta(\Xi)$ a lower bound on the maximal transition probability within two blocks of $\Xi$, uniformly over the time $t \geq 0$, previous assignments $z \in [K]^t$, and the  starting position $j \in [|\Xi|]$, i.e.,
\begin{equation}
\label{def:zeta}
\zeta(\Xi) \equiv \inf_{t\geq 0,\;z\in [K]^{t}, \;1 \leq j \leq |\Xi|}\;\; \left[\max_{1 \leq s \leq  2|\Xi|} \, \pi_{t+s}(z,\Xi(j:j+s-1))\right].
\end{equation}

\begin{remark} 
The proposed procedure may terminate in  more than one cycle, although the probability of this event is small for large values of $b_1,d$; see  Appendix \ref{subsec:upp_cycles}.  In the analysis, we use $\widetilde{\lambda}(\Xi_1)$  and $\zeta(\Xi_1)$ \emph{only} for acceleration stages, \emph{excluding the first one}.  Specifically,  $\widetilde{\lambda}(\Xi_1)$ is used to upper bound the expected time of change. Further,  recall the recursive formula for the posterior odds in Lemma \ref{lemma:recursive}. The role of $\zeta(\Xi_1)$  is to lower bound the \emph{maximum} value of $(\Gamma_{t-1} + \Pi_t)$ over the initial $2|\Xi_1|$ time points of an acceleration stage.
\end{remark}

The following theorem establishes an upper bound on the expected total number of observations required by the proposed procedure. Its proof is presented in 
   Appendix \ref{app:proofs_regarding_upper_bound}.

\begin{theorem}\label{thm:upper_bound_T}
Suppose that the response model satisfies condition 
\eqref{assumptions:response_KL} and that the change-point model satisfies condition \eqref{assumption_stability}. 
Let  $\Xi_1$ and $\Xi_2$ be  blocks of length $\ell_1$ and $\ell_2$, respectively. Then, for any $\epsilon > 0$, $b_2 \geq b_1 > 1$, and $d > 1$, we have
$$
\Exp[\mwd{T}] \leq {\cU}(b_1,b_2,d)\,(1+\epsilon) + \frac{(1+\epsilon)^2}{1-\eta(b_1, d)} \mathcal{R},
$$
where \,  $\eta(b_1, d)  \equiv (b_1 + d)/ (d(1+b_1))$,  
\begin{align*}
{\cU}  (b_1,b_2, & d) 
\equiv \left( \lambda(\Xi_1) + \frac{\log(b_2)}{\Dro(\Xi_2)}\right) +  \frac{\eta(b_1,d)}{1-\eta(b_1,d)} \left(\widetilde{\lambda}(\Xi_1) + \frac{\log(b_2)}{\Dro(\Xi_2)} \right)   \\
&+ \frac{1}{1-\eta(b_1,d)} \left[
 \log(b_1)\left(\frac{1}{\Dro(\Xi_1)} - \frac{1}{\Dro(\Xi_2)} \right) +
  \frac{\log(d)}{1+b_1}\left( \frac{1}{\Dro(\Xi_2)}+\frac{1}{J(\Xi_2)} \right)  \right],
  \end{align*}
  and 
  \begin{align*}
\mathcal{R} &\equiv M\left(\Xi_{1},\frac{\epsilon \Dro(\Xi_1)}{1+\epsilon} \right) + M \left(\Xi_{2}, \frac{\epsilon\Dro(\Xi_2)}{1+\epsilon} \right) + 3\ell_1 +2\ell_2 \\
&+\left(\frac{V^{I}(\Xi_1)}{(\Dro(\Xi_{1}))^2} + \frac{V^{I}(\Xi_{2})}{(\Dro(\Xi_{2}))^2} + \frac{V^{J}(\Xi_{2})}{(J(\Xi_{2}))^2} \right) + \frac{|\log(\zeta(\Xi_{1}))|}{\Dro(\Xi_{1})}.
\end{align*}
Further, if  the treatment assignment rule is modified as in Subsection \ref{subsection:minor_etension}, the
upper bound   holds with $\lambda(\Xi_1)$  replaced by
$\lambda(z_0, \Xi_1)$, which is  defined as follows:
\begin{align}\label{modified_lambda}
\begin{split}
&(1 - \pi _{0})  +(1 - \pi _{0})  \sum
_{t=1}^{t_0} \prod_{s=1}^{t}
\bigl(1-\pi _{s}\bigl(z_{0}(1:s)\bigr) \bigr) \\
& + (1 - \pi _{0}) \sum
_{t=t_0+1}^{\infty} \prod_{s=1}^{t_{0}}
\bigl(1-\pi _{s}\bigl(z_{0}(1:s)\bigr) \bigr) \prod
_{s=t_{0}+1}^{t} \bigl(1- \pi
_{s}\bigl(z_{0}, \Xi _{1}(1:s-t_{0})
\bigr) \bigr). 
\end{split}
\end{align}
\end{theorem}

The above upper bound is non-asymptotic. In Section \ref{sec:asymptotic_framework}, we consider an asymptotic framework, where we let the tolerance level $\alpha$ vanish and allow the change-point model to depend on  $\alpha$. Under this framework,  we show that ${\cU}(b_1,b_2,d)$ is the dominant term in the upper bound as $\alpha \to 0$, and thus we select the blocks $\Xi_1, \Xi_2$ and the thresholds $b_1,d$ to minimize ${\cU}(b_1,b_2,d)$, which is discussed in Subsection \ref{subsec:thres} and \ref{subsec:block}.

We do not attempt to optimize the upper bound over $\epsilon > 0$, since the speed of uniform convergence $M(\Xi, \epsilon)$ in \eqref{assumption_stability} diverges as $\epsilon \to 0$. In the asymptotic analysis, we show that, for each fixed $\epsilon > 0$, the ratio  $\Exp[\tilde{T}]/{\cU}(b_1,b_2,d)$ is asymptotically upper bounded by $1+\epsilon$ as $\alpha \to 0$. Since this ratio does not depend on $\epsilon$, we then let $\epsilon \to 0$.

\subsection{{Thresholds} selection} \label{subsec:thres}
First, we {select}  $b_2$ as in \eqref{eq:bK_sel} to  control the false alarm rate below $\alpha$. Then,  for  \textit{given} blocks $\Xi_1, \Xi_2$, we  select  $b_1, d$  to  minimize $\cU(b_1,b_2,d)$, as discussed above. The minimization can be done numerically, but it is possible to obtain a closed-form expression for approximate minimizers.  Indeed,  as long as the block used in the detection stages  has a better detection power than the block  used in the acceleration  stages, in the sense that $\Dro(\Xi_{2}) \geq \Dro(\Xi_{1})$,   then, as we explain in Appendix \ref{discussion_b1_d_sel},
\begin{align}\label{eq:b1_d_sel}
\begin{split}
b_1 &= \min\left\{b_2, \;\; 
\frac{\widetilde{\lambda}(\Xi_{1}) + {\log(b_2)}/{\Dro(\Xi_2)}}{ {1/}{\Dro(\Xi_1)} -{1}/{\Dro(\Xi_2)} }-1
\right\}, \quad
d = b_1 \,  \frac{\widetilde{\lambda}(\Xi_{1}) + \log(b_2)/{\Dro(\Xi_{2})}}{1/\Dro(\Xi_{2}) + 1/J(\Xi_2)}
\end{split}
\end{align}
are approximate minimizers of  $\cU(b_1,b_2,d)$ for small values of $\alpha$. 

For simplicity of presentation, we assume that $\Dro(\Xi_{2}) \geq \Dro(\Xi_{1})$ holds in what follows. 
Indeed, if $\Dro(\Xi_{2}) < \Dro(\Xi_{1})$, the upper bound in Theorem \ref{thm:upper_bound_T} suggests choosing $b_1 = b_2$, which implies that the duration of detection stages is zero and block $\Xi_2$ is actually never assigned. That is, when block $\Xi_1$ has a better detection power than $\Xi_2$, we would not switch to block $\Xi_2$ and instead have a single stage with block $\Xi_1$.

\subsection{Blocks selection} \label{subsec:block}
It remains to select the blocks $\Xi_1$ and $\Xi_2$. For this, we rely on the fact, which is  formally justified under the asymptotic framework in Section \ref{sec:asymptotic_framework},  that  when $b_2$ is  selected according to \eqref{eq:bK_sel} and 
 $b_1, d$ according to  \eqref{eq:b1_d_sel},  the dominant term in the upper bound ${\cU}(b_1,b_2,d)$ for small  values of $\alpha$    is given by {
${\lambda}(\Xi_1) + |\log(\alpha)|/\Dro(\Xi_2)$},  
 where  $\Dro(\Xi_2)$ is the adjusted information number of $\Xi_2$,  defined in  \eqref{adjust_info_num},
  and ${\lambda}(\Xi_1)$ the expected time of change under $\Xi_1$, defined in 
  \eqref{block_expected_change}; in the case of the  modified assignment rule in \eqref{mod_rule}, ${\lambda}(\Xi_1)$  is replaced by $\lambda(z_0,\Xi_1)$, defined in \eqref{modified_lambda}. Thus, we propose to select the blocks $\Xi_1$ and $\Xi_2$ as the solutions to the following optimization problems:
\begin{equation}
\label{prop_opt_blocks_sel}
\;\;
\Xi_1 \;=\; \underset{\Xi \in [K]^{\ell},\, \ell \in \bN}{\arg\min}\, \lambda(\Xi),\qquad
\;\;\Xi_2 \;=\; \underset{\Xi \in [K]^{\ell}, \, \ell \in \bN}{\arg\max}\, \Dro(\Xi).
\end{equation}
In  Section \ref{sec:asymptotic_framework}, we show that when the  blocks are  selected in this way, the proposed procedure
 achieves the optimal performance in an asymptotic sense.

As we show in the next subsection,   both optimization problems in  \eqref{prop_opt_blocks_sel}  can be solved efficiently for the finite memory change-point model in \eqref{geometric_change_point}. The solutions under two other change-point models are discussed in  Appendices \ref{app:ex_exp} and \ref{app:ex_poly}.  
In general, if it is difficult to solve  \eqref{prop_opt_blocks_sel} exactly, one may instead find approximate solutions by exhaustive search over a finite collection of candidates. For example, we may restrict the lengths of the two blocks below some  integer $\ell_{\max}$ and  select $\Xi_1, \Xi_2$ as follows:
\begin{equation*}
\Xi_1 \;=\; \underset{\Xi \in [K]^{\ell},  \, \ell \leq \ell_{\max}}{\arg\min}\, \lambda(\Xi),\qquad
\Xi_2 \;=\; \underset{\Xi \in [K]^{\ell},\, \ell \leq \ell_{\max}}{\arg\max}\, \Dro(\Xi).
\end{equation*}

Finally, we stress that the two blocks can always be supplied by practitioners. As mentioned earlier,  the validity of the method is guaranteed by selecting $b_2$ according to \eqref{eq:bK_sel}, and the flexibility in block specifications may be desired in practice.

\subsection{Example: the finite memory model}\label{subsec:example_finite_memory}
In this subsection, we illustrate the design of the proposed procedure for 
 the finite memory change-point model in \eqref{geometric_change_point},
 according to which the current transition probability depends only on the current and previous $\kappa \geq 0$ treatments through a link function $\Psi$.

We first discuss the computation of $b_1,d$ using the formulae in      \eqref{eq:b1_d_sel} for given blocks $\Xi_1, \Xi_2$.
Fix a block $\Xi$. Note that $I(\Xi), J(\Xi)$ and  $\mwd{\lambda}(\Xi)$  are obtained by   \eqref{def:KL} and \eqref{block_upper} respectively, {and}  recall that $I(\Xi)$ (resp. $J(\Xi)$) is the average of the Kullback-Leibler divergences $\{I_x: x \in \Xi\}$ (resp. $\{J_x: x \in \Xi\}$). 
Since $\Dro(\Xi) = I(\Xi) + d(\Xi)$, it remains to determine    the quantity $d(\Xi)$   defined in \eqref{assumption_stability_pointwise}. Due to the periodic structure induced by the proposed assignment rule, $\mwd{\cX}$,  condition \eqref{assumption_stability_pointwise} holds for every block $\Xi \in [K]^{\ell}$  with  
$$
d(\Xi) = 
\frac{1}{\ell} \sum_{j=1}^{\ell} \omega\left( \Psi(\Xi(j:(j-\kappa))) \right).
$$

 Next, we  show how to obtain the blocks  $\Xi_1$ and $\Xi_2$ that are solutions to  the optimization problems in \eqref{prop_opt_blocks_sel}
 by solving two auxiliary, \textit{deterministic} dynamic programming problems with finite state spaces.

\smallskip
\noindent \textbf{\textit{Selection of acceleration block.}} 
 We start with the first optimization problem in \eqref{prop_opt_blocks_sel} and  consider  the modified assignment rule in \eqref{mod_rule}, for which $\lambda(z_0,\Xi)$ is given by  \eqref{modified_lambda}. 
Under general  Markovian change-point models in \eqref{Markovian_change_point},  
for any sequence of treatments $\{x_t \in [K]: t \in \bN\}$,  the expected time of change is given by $(1-\pi_0)(1+V(\bS_0; \{x_t\}))$, where
\begin{align*}
    V(\bS_0; \{x_t\}) \equiv \sum_{t =1}^{\infty} \prod_{s = 1}^{t} \left(1 - \Psi(x_s, \bS_{s-1})\right),
\end{align*}
and  $\bS_t = \Phi\left(x_t, \bS_{t-1} \right)$ for $t \in \bN$; for the finite memory model, the transition function $\Phi$ takes the special form  in \eqref{geometric_change_point}. Then,  minimizing  $\lambda(z_0, \Xi)$ over the initial segment $z_0$ and block $\Xi$ as  in \eqref{prop_opt_blocks_sel}
is equivalent to finding a sequence of treatments $\{x_t\}$ that minimizes $V(\bS_0; \{x_t\})$ subject to the \textit{``periodicity'' constraint} that there exist $\ell, u \in \bN$ with $\ell < u$ and $x_{\ell+s} = x_{u+s} \text{ for each } s \geq 0$.

To solve this constrained problem, 
our strategy is first to \textit{drop the constraint} and find the minimizing sequence $\{x_t\}$ of $V(\bS_0; \{x_t\})$  without requiring it to be periodic after an initial period, and then to show that the solution to this unconstrained problem, in fact, satisfies this constraint.  Specifically, we define the optimal cost \textit{without} the constraint 
as follows:
\begin{equation*}
V^*(\bS_0) \equiv \inf_{(x_1,x_2,\ldots) \in [K]^{\infty}} 
  V(\bS_0; \{x_t: t \in \bN\}).
\end{equation*}
As shown in Appendix 
\ref{app:proof_acc_block_fmm}, by 
 abstract dynamic programming theory \citep{bertsekas2022abstract},   the minimizing sequence of treatments  that attains $V^*(\bS_0)$ is given by 
$$
x^*_t \equiv \upsilon^*(\bS_{t-1}^*)  \qquad \text{and} \qquad \bS_t^* = \Phi\left(x^*_t, \bS_{t-1}^* \right), \quad  t \in \bN
$$
where  $\bS_0^{*} = \bS_0$
and   the  policy $\upsilon^*: [K]^\kappa \to [K]$  can be computed numerically. 
Since the state space $[K]^\kappa$ is finite, after some burn-in period, the sequence is periodic.  That is, if  
$$u^* \equiv \inf\{t \geq 1: \bS^*_{t} = \bS^*_\ell \text{ for some }  0 \leq \ell <  t\},$$  we have $u^* \leq K^{\kappa}$. Finally,   denote by  
$\ell^*$  the  integer  in $[0, u^*)$   such that $\bS^*_{\ell^*} = \bS^*_{u^* }$; then,   
  the solution to the first problem in \eqref{prop_opt_blocks_sel} (more precisely,  its modified version) is given by
\begin{align}\label{geometric_acc_sol}
z_0^* \equiv (x^*_1,\ldots, x^*_{\ell^*}) \qquad \text{and} \qquad
\Xi_1^* \equiv \left(x^*_{\ell^*+1}\,\ldots, x^*_{u^*} \right).
\end{align}

\smallskip
\noindent \textbf{\textit{Selection of detection block.}}  
We next discuss the second optimization problem in \eqref{prop_opt_blocks_sel}. 
As before, the strategy is first to find the solution to the following unconstrained problem
\begin{align*}
\sup_{(x_1,x_2,\ldots) \in [K]^{\infty}} 
\liminf_{M \to \infty}
\frac{1}{M}
\sum_{t =1}^{M}\left(
I_{x_t} + \omega(\Psi(x_t, \bS_{t-1}))
\right), \;\;
\text{ with }\; \bS_{t} = \Phi(x_t, \bS_{t-1}) \text{ for } t \geq 1,
\end{align*} 
and then to show that the solution is periodic after an initial period. 

The unconstrained minimization problem is an instance of average-cost dynamic programming problems \cite{bertsekas1995dynamic}. In Appendix \ref{app:proof_det_block_fmm}, we show  that the solution is given by {
$$x'_t \equiv \upsilon'(\bS'_{t-1}) \quad 
\text{ and } \quad \bS'_t = \Phi\left(x'_t, \bS'_{t-1} \right), 
 \quad   t\in \bN$$}
 where
$  \bS'_0 = \bS_0$ and  the policy 
 $\upsilon':[K]^{\kappa} \to [K]$ can be obtained numerically.  Since the state space $[K]^{\kappa}$ is finite, we have that 
$$u' \equiv \inf\{t \geq 1: \bS'_{t} = \bS'_\ell \text{ for some }  0 \leq \ell <  t\} 
$$ is at most $K^{\kappa}$. Further, let   $\ell'$ be  the   integer such that  $\bS'_{\ell'} = \bS'_{u'}$ and $0 \leq \ell' < u'$. Then, {a}
 solution to  the second optimization problem in \eqref{prop_opt_blocks_sel} is given by
\begin{align} \label{xi2_finite_choice}
\Xi_2^* \equiv (x'_{\ell'+1},\ldots,x'_{u'}).
\end{align}

\begin{remark}
In Appendix \ref{app:proof_det_block_fmm}, we show that
the optimal value of the unconstrained problem does not depend on the initial state, $\bS_0$. Further, since the objective is {to minimize the} average cost, we can select $\Xi_2$ to be any recurring pattern appearing in $\{x_t'\}$, i.e., $(x'_{\ell'+1+s},\ldots,x'_{u'+s})$ for any $s \geq 0$. 
\end{remark}

To sum up, for the finite memory change-point model in \eqref{geometric_change_point}, we suggest the modified procedure \eqref{mod_rule} with the initial treatments $z_0$, acceleration block $\Xi_1$ and detection block $\Xi_2$ selected according to \eqref{geometric_acc_sol} and \eqref{xi2_finite_choice}. Once $\Xi_1$ and $\Xi_2$ are chosen, we set  $b_2$ according to \eqref{eq:bK_sel} and $b_1,d$ to \eqref{eq:b1_d_sel}, as discussed at the beginning of this subsection. Note that the computation of $\Xi_1$ and $\Xi_2$ as outlined above requires solving dynamic programming problems, which however are significantly simpler than the one in Section \ref{subsec:auxilliary} since they are deterministic and have finite state spaces.

\section{Asymptotic optimality}\label{sec:asymptotic_framework}
The results so far are non-asymptotic. In this section,  we introduce an asymptotic regime (Subsection \ref{subsec:alpha}) in which both the expected time of change and the detection delay may diverge at possibly different rates as the tolerance level $\alpha \to 0$. Under this framework,   we obtain an asymptotic approximation to the upper bound in Theorem \ref{thm:upper_bound_T}  (Subsection \ref{subsec:asym_upper})  and establish a \textit{universal}  asymptotic lower bound on the expected total number of observations (Subsection \ref{subsec:asy_lower}). Based on these results, we establish the asymptotic optimality of the proposed procedure under various change-point models, including the finite memory model in \eqref{geometric_change_point} (Subsection \ref{subsec:asy_opt}).




\subsection{Asymptotic framework}\label{subsec:alpha}
Recall the decomposition  of  $\Exp[T]$ in  \eqref{ess_decompose} for  some procedure  $(\cX,T) \in \cC_\alpha$.  Due to the false alarm constraint, as $\alpha \to 0$, the third term becomes negligible, whereas the second, corresponding to the average detection delay,  goes to infinity. For the asymptotic analysis to be relevant, the first term, which represents the expected time of change, should also be allowed to diverge, maybe at a rate even faster than the second term.  Thus, in this section, we parametrize the functions $\{\pi_{t}\}$ by $\alpha$  and allow them to vanish as $\alpha \to 0$. When we want to emphasize this parametrization, we write $\pi_{t}(\,\cdot\, ; \alpha)$ instead of  $\pi_{t}(\,\cdot\,)$.   As a result, all quantities related to the change-point model, such as 
 $\lambda(\Xi)$ and $\mwd{\lambda}(\Xi)$ in \eqref{block_expected_change} and \eqref{block_upper},  $d(\Xi)$ and  $M(\Xi,\epsilon)$ in \eqref{assumption_stability_pointwise} and  \eqref{assumption_stability},  $\Dro(\Xi)$ in \eqref{adjust_info_num}, $\zeta(\Xi)$ in \eqref{def:zeta}, 
  as well as the blocks $\Xi_1, \Xi_2$ of the proposed procedure and their lengths,  $\ell_1, \ell_2$, are allowed  to vary with $\alpha$. For simplicity, we assume that the densities $\{f_x,g_x: x \in [K]\}$ do not depend on $\alpha$, although such an extension is straightforward. 

For quantities $x$ and $y$ that are parameterized by $\alpha$,  we  write   $x = o(y)$ if $\lim (x/y) = 0$,  $x = O(y)$ if $\limsup (x/y) < \infty$,\; 
$x \geq y (1+o(1))$ if $\liminf (x/y) \geq 1$, 
$x \leq y (1+o(1))$ if $\limsup (x/y) \leq 1$,  and $x \sim y$ if $\lim (x/y) = 1$,  where all limits are taken as $\alpha \to 0$.

\subsection{Asymptotic upper bound} \label{subsec:asym_upper}
We first explore the asymptotic {behavior} of the upper bound in  Theorem \ref{thm:upper_bound_T} on the expected {total number of observations} 
of the proposed procedure, with the thresholds $b_2$ and $b_1, d$ selected according to \eqref{eq:bK_sel} and \eqref{eq:b1_d_sel}, respectively.  We observe that, as $\alpha \to 0$,  $b_1,b_2,d$ diverge and $\eta(b_1,d)$ vanishes.  The term $\mathcal{R}$  does not depend on the thresholds, and it {will} be negligible compared to ${\cU}(b_1,b_2,d)$ if     the blocks $\Xi_1$ and $\Xi_2$ used in the proposed procedure satisfy the following conditions
as $\alpha \to 0$:
\begin{align}\label{assumption:asymp_cond_cp1}
\begin{split}
&  \max_{i = 1,2}\; d(\Xi_i) = O(1), \qquad \max_{i = 1,2}\; M(\Xi_i, \epsilon) = o\left(|\log(\alpha)| \right) \text{ for any } \epsilon \in (0,1), \\
&  \max_{i = 1,2}\; \ell_i = O(|\log(\alpha)|),\quad
|\log(\zeta(\Xi_1))|  +  \log\left( \mwd{\lambda}(\Xi_1) \right) =  o\left( \lambda(\Xi_1) + |\log(\alpha)| \right),
\end{split}
\end{align}
where  $\lambda(\Xi_1)$ is replaced by $\lambda(z_0, \Xi_1)$, defined in \eqref{modified_lambda}, 
 for the modified assignment rule in  \eqref{mod_rule}.  Under this technical condition, ${\cU}(b_1,b_2,d)$ is the dominant term in the upper bound in  Theorem \ref{thm:upper_bound_T}, which can be further simplified as shown next. The proofs for {the}  results in this subsection can be found in Appendix \ref{app:aymp_upper_bound}.

\begin{corollary}\label{lemma:aymp_upper_bound}
Assume that the response model satisfies condition
\eqref{assumptions:response_KL}  and that the blocks $\Xi_1, \Xi_2$ of the proposed procedure in Section \ref{sec:mastery}  satisfy \eqref{assumption:asymp_cond_cp1}. If the
 thresholds are selected according to \eqref{eq:bK_sel} and \eqref{eq:b1_d_sel}, then
\begin{align*}
\Exp [\mwd{T} ] \leq \left(
\lambda(\Xi_1) + \frac{|\log(\alpha)|}{\Dro(\Xi_2)}
\right)(1+o(1)).
\end{align*}

Further, if the treatment assignment rule is modified as in \eqref{mod_rule}, the conclusion continues to hold with $\lambda(\Xi_1)$  replaced by $\lambda(z_0, \Xi_1)$ defined in \eqref{modified_lambda}.
\end{corollary}


Some remarks are in order about the conditions in  \eqref{assumption:asymp_cond_cp1}. The first condition essentially requires that the change cannot be forced by applying the blocks of the proposed procedures. It holds for any block if the transition probabilities are bounded away from $1$, i.e.,   there exists a constant $c \in (0,1)$ such that 
\begin{equation}
\label{general_bounded_away_from_one}
\sup_{t \geq 0, z \in [K]^t} \pi_t(z; \alpha) \leq 1-c,\;   \text{ for any small enough } \alpha.
\end{equation} 
The other conditions impose requirements on the rates at which various quantities are allowed to diverge, and their verification depends on the change-point model. For the finite memory model in \eqref{geometric_change_point}, if the transition cannot be forced in the sense of \eqref{general_bounded_away_from_one}, they hold as shown below.


\begin{lemma}\label{lemma:upper_finite_meet_asym_cond}
Consider the finite memory model in \eqref{geometric_change_point} with a fixed $\kappa \geq 0$  and the modified assignment rule in \eqref{mod_rule}. Assume  that \eqref{general_bounded_away_from_one} holds and $\min\{\ell_1,\ell_2\} \geq \kappa$. If the lengths of $z_0, \Xi_1$ and $\Xi_2$ are bounded as $\alpha \to 0$, i.e., $\max\{\ell_1, \ell_2, |z_0|\}= O(1)$, then condition \eqref{assumption:asymp_cond_cp1} holds.

In particular, if  $z_0, \Xi_1$ and $\Xi_2$ are selected according to \eqref{geometric_acc_sol} and \eqref{xi2_finite_choice} as suggested in Subsection \ref{subsec:example_finite_memory}, then the lengths of $z_0, \Xi_1$ and $\Xi_2$ are bounded by $K^{\kappa}$.
\end{lemma}

\begin{remark}
The condition  $\min\{\ell_1,\ell_2\}  \geq \kappa$ in the above lemma can be made  without loss of generality, since for $\Xi \in \{\Xi_1,\Xi_2\}$, we can always define $\tilde{\Xi} \equiv (\Xi,\ldots,\Xi)$ 
by repeating the block $\Xi$ so that the length of $\tilde{\Xi}$ is at least $\kappa$. Note that $\tilde{\Xi}$ and $\Xi$ induce the same assignment rule, and since $\kappa$ is fixed, if Corollary \ref{lemma:aymp_upper_bound} holds for block $\tilde{\Xi}$, it also holds for $\Xi$.
\end{remark}

\subsection{Universal lower bound} \label{subsec:asy_lower}
Next, we establish an asymptotic lower bound on the expected total number of observations for all procedures with the desired error control. {For this,}
we recall the definition of  $\mwd{\lambda}(\Xi)$ in \eqref{block_upper} and
 assume  that there exists a block $\Xi$ such that
\begin{align}\label{existence of a block}
\alpha\, \mwd{\lambda}(\Xi) \;=\; o(\; \lambda^*  + |\log(\alpha)| \;), \quad \text{where} \quad
\lambda^* \equiv \inf_{\cX} \Exp\left[ \Theta_{\cX} \right].
\end{align}
Note that $\lambda^*$ is the smallest possible expected time of change. Further, {we assume that there is}  a positive quantity, $\Dro^*$,   which  depends on $\alpha$  such that $\Dro^* = O(1)$,    {so that} for every   assignment rule   $\cX$  
 and   $\epsilon > 0$ {we have}
\begin{align}
\label{detection_delay_optimal}
\limsup_{\alpha \to 0}\;\;
\sup_{t \geq 0, \cX} \; \; 
\frac{1}{ \Dro^* N_{\epsilon,\alpha}}\sum_{s=1}^{N_{\epsilon,\alpha}}\left(
I_{X_{t+s}} + \omega(\Pi_{t+s}(\cX))
\right) \leq  1 + \epsilon/2,
\end{align}
where $ N_{\epsilon,\alpha} \equiv \lfloor (1-\epsilon) \, |\log(\alpha)|/ \, \Dro^* \rfloor$ and   $\lfloor z \rfloor$  is the largest integer that does not exceed $z$. 

\begin{theorem}\label{thm:asymptotic_lower_bound}
Suppose that the response model satisfies condition 
\eqref{assumptions:response_KL} and that  the change-point model satisfies conditions \eqref{general_bounded_away_from_one} and  \eqref{existence of a block}. If   condition \eqref{detection_delay_optimal} also holds, then 
\begin{equation*}  
  \inf_{(\cX, T) \in \cC_{\alpha}} \Exp[T]
\geq  \left( \lambda^* + \frac{|\log(\alpha)|}{\Dro^*}  \right) (1+ o(1)).
\end{equation*}
\end{theorem}
\begin{proof}
    See Appendix \ref{app:lower_bd}.
\end{proof}

We discuss conditions \eqref{existence of a block} and   \eqref{detection_delay_optimal}. Fix an arbitrary procedure  $(\cX,T) \in \cC_\alpha$. To lower bound $\Exp[T]$,  we lower bound the first two terms in the decomposition  \eqref{ess_decompose} and upper bound the third, i.e.,  $\Exp[(\Theta_{\cX} - T)^{+}]$, which is, on the event of a false alarm,  the expected length of the interval between the time of the false alarm and the actual time of the change. 
The latter task is not possible for {an arbitrary}
assignment rule,  since one could choose any inferior treatment assignment rule {\textit{after}} the stopping time $T$.  However, we can modify $\cX$ after the time of stopping $T$ by assigning repeatedly the block 
$\Xi$ in \eqref{existence of a block} (see Appendix \ref{app:lower_bd}). If we denote the modified assignment rule by ${\cX}'$, then $({\cX}', T)$ has the same stopping rule and false alarm rate as $(\cX, T)$, since the modification is after $T$. Thus, in the analysis, {without loss of generality we can} focus on this modified rule instead, for which $\Exp[(\Theta_{{\cX}'} - T)^{+}]$ can be controlled.

Further, recall the recursive formula for the posterior odds in Lemma \ref{lemma:recursive}. The quantity $\Dro^*$ in \eqref{detection_delay_optimal} can be interpreted as the fastest speed, under any assignment rule, at which the logarithm of the posterior odds can grow.


For the finite memory change-point model in  \eqref{geometric_change_point},  $\lambda^*$ and $\Dro^*$ are closely related to the optimal design in Subsection \ref{subsec:example_finite_memory} for the proposed procedure, as shown in the following lemma.  Recall that $(z_0^*, \Xi_1^*)$ and $\Xi_2^*$ are defined in \eqref{geometric_acc_sol} and \eqref{xi2_finite_choice}, which  depend on $\alpha$.

 \begin{lemma}\label{lemma:finite_lb_condition} 
Consider the finite memory change-point model \eqref{geometric_change_point} with $\kappa \geq 0$ fixed, and 
assume condition \eqref{general_bounded_away_from_one} holds. Then,
condition \eqref{existence of a block} holds with $\Xi = \Xi_1^*$ and 
$\lambda^* = \lambda(z_0^*, \Xi_1^*)$, and condition \eqref{detection_delay_optimal}  holds with $\Dro^* = \Dro(\Xi_2^*)$.  
\end{lemma}
\begin{proof}
    See Appendix \ref{proof:finite_lb_condition}.
\end{proof}

\subsection{Asymptotic optimality} \label{subsec:asy_opt}

To establish the asymptotic optimality of the proposed procedure in Section \ref{sec:mastery}, we compare the asymptotic upper bound in Corollary \ref{lemma:aymp_upper_bound} with the universal asymptotic lower bound in Theorem \ref{thm:asymptotic_lower_bound}.

For the proposed procedure $(\widetilde{\cX}, \widetilde{T})$,  we denote by $\Xi_1, \Xi_2$ the acceleration and detection blocks,  respectively, which depend on $\alpha$. Given the two blocks $\Xi_1,\Xi_2$, we select the thresholds according to  \eqref{eq:bK_sel} and \eqref{eq:b1_d_sel}.

\begin{corollary}\label{cor:main}
Assume that the response model satisfies condition \eqref{assumptions:response_KL},
and that conditions
\eqref{assumption:asymp_cond_cp1},    \eqref{general_bounded_away_from_one},  \eqref{existence of a block} and \eqref{detection_delay_optimal} hold.
Assume further 
\begin{equation}\label{suff_for_asym_opt}
(i) \;\; \lambda(\Xi_1) \; \sim \;  \lambda^*  \quad \qquad \text{and} \quad \qquad 
(ii)  \;\;
 \Dro(\Xi_2) \sim \Dro^*.
\end{equation}
Then,   
$(\widetilde{\cX}, \widetilde{T}) \in \cC_{\alpha}$
for every $\alpha \in (0,1)$,   and 
$$\Exp[\mwd{T}] \; \sim \;   \inf_{(\cX, T) \in \cC_{\alpha}} \Exp[T]  \; \sim \;   \lambda^* + \frac{|\log(\alpha)|}{\Dro^*}.
$$ 
If the modification is made as in \eqref{mod_rule}, the conclusion continues to hold with $\lambda(\Xi_1)$  replaced by $\lambda(z_0, \Xi_1)$ defined in \eqref{modified_lambda}.
\end{corollary}

We have discussed before all conditions in Corollary  \ref{cor:main} except for \eqref{suff_for_asym_opt}, which requires that blocks $\Xi_1$ and $\Xi_2$ are selected optimally for the acceleration and the detection task, respectively.  Specifically, condition (i) in \eqref{suff_for_asym_opt} requires that the expected time of change, if block $\Xi_1$ was assigned repeatedly,  should be of the same order as the smallest possible over all sequences of treatments, while condition (ii) in \eqref{suff_for_asym_opt} requires that the detection speed under block $\Xi_2$ should be of the same order as the largest possible. 

For the finite memory model in  \eqref{geometric_change_point}, if we choose $(z_0^*, \Xi_1^*)$ and $\Xi_2^*$ as in \eqref{geometric_acc_sol} and \eqref{xi2_finite_choice},  then by Lemma \ref{lemma:finite_lb_condition}, 
condition \eqref{suff_for_asym_opt} holds. Together with Lemma \ref{lemma:upper_finite_meet_asym_cond} and \ref{lemma:finite_lb_condition}, we have that the proposed procedure is asymptotically optimal under this model, as summarized below. Denote by $(\widetilde{\cX}^*, \widetilde{T}^*)$ the proposed procedure with $(z_0^*, \Xi_1^*)$  in \eqref{geometric_acc_sol},  $\Xi_2^*$ in \eqref{xi2_finite_choice}, and the  
thresholds in \eqref{eq:bK_sel} and \eqref{eq:b1_d_sel}.

\begin{corollary}\label{cor:geo_opt} 
Consider the finite memory change-point model \eqref{geometric_change_point}  with $\kappa \geq 0$ fixed.
Assume that the response model satisfies condition \eqref{assumptions:response_KL} and 
that the change cannot be forced in the sense of 
\eqref{general_bounded_away_from_one}.  Then $(\widetilde{\cX}^*, \widetilde{T}^*) \in \cC_{\alpha}$  for every $\alpha \in (0,1)$, and 
$\Exp[\mwd{T}^*] \; \sim \;   \inf_{(\cX, T) \in \cC_{\alpha}} \Exp[T]
$.
\end{corollary}

In Appendices \ref{app:ex_exp} and \ref{app:ex_poly}, we establish the asymptotic optimality of the proposed procedure under proper selections of blocks $\Xi_1,\Xi_2$ and thresholds $b_1,b_2,d$, for a Markovian change-point model with long memory and a non-Markovian model, respectively.

\section{Models with unknown parameters} \label{sec:distr_uncertainty}
So far, we assume perfect knowledge of the sequence of functions $\{\pi_t, t \geq 0\}$, as well as the pre-and post-change response densities $\{f_x, g_x: x \in [K]\}$. In practice, however, they are estimated from historical data, and attached with distributional uncertainty, either in the form of asymptotic limiting distribution under the frequentist framework or posterior distribution under the Bayesian framework.

In this section, we assume that these quantities depend on some parameter $\eta$ in some Polish space $\cH$,  and write instead $\{\pi_{t,\eta}, t \geq 0\}$,   $\{f_{x,\eta}, g_{x,\eta}: x \in [K]\}$. Further, we assume there is a given prior distribution $\tilde{\mu}$ for $\eta$, which represents distributional uncertainty.  
For each $\eta \in \cH$ and $t \in \bN$, define $\Gamma_{0,\eta} = {\pi_{0,\eta}}/{(1-\pi_{0,\eta})}$ and
\begin{align*}
\Gamma_{t,\eta} = (\Gamma_{t-1,\eta} + \Pi_{t,\eta} ) \frac{\Lambda_{t,\eta}}{1-\Pi_{t,\eta}},  \quad
\Pi_{t,\eta} = \pi_{t,\eta}(X_1,\ldots,X_t), \quad
\Lambda_{t,\eta} = \frac{g_{X_t,\eta}(Y_t)}{f_{X_t,\eta}(Y_t)}.
\end{align*}
That is, if  the parameter $\eta$ was known, then $\Pi_{t,\eta}$, $\Lambda_{t,\eta}$, and $\Gamma_{t,\eta}$ are, respectively, the transition probability, the likelihood ratio of the response, 
and the posterior odds at time $t$ (see Section \ref{sec:ProbFormulation}); they depend on the assignment rule $\cX$, which  is omitted for simplicity. Further, define  
$$
D_{0,\eta} \equiv 1 - \pi_{0,\eta}, \quad \text{ and }\quad D_{t,\eta} \equiv D_{0,\eta} \prod_{s=1}^{t} (1-\Pi_{s,\eta}) f_{X_s,\eta}(Y_s), \;\;t \geq 1.
$$

The next lemma provides a formula for the posterior odds that the change has already occurred at time $t \geq 0$, i.e., for $\Gamma_t$ in \eqref{posterior_odds_prob}, and also shows that Lemma \ref{lemma:err_control} continues to hold. In particular, we can control the false alarm rate below $\alpha$, if, at the time of stopping, the posterior odds process is above $(1-\alpha)/\alpha$.  The proof is in Appendix \ref{proof:post_odds_para}.

\begin{lemma}\label{distr:lemma_postodds}
Let $\cX$ be an assignment rule, and $S$ a finite, $\{\cF_t\}$-stopping time.  For each $t \geq 0$, the posterior odds that $L_t=1$ given $\cF_t$ is given by the following:
\begin{align*}
    \Gamma_t = \left(\int_{\cH} \Gamma_{t,\eta} D_{t,\eta} \tilde{\mu}(d\eta)\right)/\left(\int_{\cH} D_{t,\eta} \tilde{\mu}(d\eta)\right).
\end{align*}
Further, if
$\Pro(\Gamma_{S} \geq b) = 1$ for some $b > 0$, then
$\Pro(S < \Theta_\cX) \leq  {1}/{(1 + b)}$.
\end{lemma}

In practice, it may be computationally expensive to evaluate the integrals in the above lemma. This motivates the following Monte-Carlo approximation, $\tilde{\Gamma}_t$, for $\Gamma_t$. Specifically, let $\eta_1,\ldots,\eta_M$ be a random sample of size $M$ from the prior distribution $\tilde{\mu}$ that is generated before any treatment assignment, and define for each $t \geq 0$,
\begin{align}\label{distr:poster_odds_approx}
    \tilde{\Gamma}_t = \left(\sum_{i=1}^{M} \Gamma_{t,\eta_i} D_{t,\eta_i} \right)/\left(\sum_{i=1}^{M}  D_{t,\eta_i}\right).
\end{align}
Note that for each $i \in [M]$, $\{\Gamma_{t,\eta_i}: t \geq 0\}$ and $\{D_{t,\eta_i}: t \geq 0\}$ admits recursive computation.


\begin{remark}
If $\eta$ is estimated by MCMC from historical data, $\{\eta_1,\ldots,\eta_M\}$ are samples from an ``approximate'' posterior distribution, given the historical data, which serves as the prior for our online acceleration and detection tasks. 
\end{remark}

We propose the following procedure when both the response model and the change-point model have unknown parameters with a given prior distribution. 

\begin{itemize}
    \item[Step 1.] Let $\eta^*$ be a ``typical'' value of the distribution $\tilde{\mu}$, such as the mean or mode.    Choose blocks $\Xi_1, \Xi_2$, and thresholds $b_1,b_2,d$ according to the discussions in Section
    \ref{sec:design} for   functions $\{\pi_{t,\eta^*}: t \geq 0\}$ and densities $\{f_{x,\eta^*}, g_{x,\eta^*}: x \in[K]\}$.
    \item[Step 2.] Modify the ``detection'' and ``testing'' rules in \eqref{def:detection_testing_rules} as follows:
    \begin{align*}\label{def:modi_det_test}
    \begin{split}
                \tilde{\sigma}(t,b) \equiv \inf\{s \geq 0: \tilde{\Gamma}_{t+s} \geq b\},   \quad 
        \tilde{\tau}(t,d) \equiv \inf\left\{
      s \geq 1:\;        
        \sum_{i=1}^{M} \omega_{t,\eta_i} \prod_{j=1}^{s}\Lambda_{t+j,\eta_i} 
        \leq 1/d \right\},
    \end{split}
    \end{align*}
    where $\omega_{t,\eta_i} =  D_{t,\eta_i}(1+\Gamma_{t,\eta_i})/(\sum_{j =1}^{M} D_{t,\eta_{j}}(1+\Gamma_{t,\eta_{j}}))
    $ is the weight for $\eta_i$ at time $t$; see discussions below.
    \item[Step 3.] Given the blocks $\Xi_1$ and $\Xi_2$, and thresholds $b_1,b_2,d$ in Step 1, we define the procedure as in Section \ref{subsec:proc_def} with $\sigma(\cdot)$ and $\tau(\cdot)$ in \eqref{def:detection_testing_rules} replaced by $\tilde{\sigma}(\cdot)$ and $\tilde{\tau}(\cdot)$ above. 
\end{itemize}

A few remarks are in order. First, by Lemma \ref{distr:lemma_postodds}, if we use the ``detection'' rule ${\sigma}(\cdot)$ in \eqref{def:detection_testing_rules},  the false alarm rate is controlled below $\alpha$, thanks to the choice of $b_2$ in \eqref{eq:bK_sel}. We replace the true posterior odds $\Gamma_t$ by its approximation $\tilde{\Gamma}_t$ due to computational considerations, {where}  the approximation error {can} be reduced by increasing the sample size,  $M$, {from the prior distribution}.
Second, by Lemma \ref{app:lemma_post_eta} in Appendix \ref{proof:post_odds_para},
the posterior distribution of $\eta$ given $\cF_t$ has a $\tilde{\mu}$-density that is proportional to  $D_{t,\eta}(1+\Gamma_{t,\eta})$; thus, using the weights  $\{\omega_{t,\eta_i}\}$    in Step 2 may be viewed again as a Monte-Carlo approximation. Finally, we note that if the prior $\tilde\mu$ is supported at a single point, i.e., it has no uncertainty, this procedure reduces to the one in Section \ref{sec:mastery}.

\section{Simulation studies} \label{sec:simulation}

In this section, we present simulation studies under the finite memory model in  \eqref{geometric_change_point} and with a binary response space, i.e., $\bY=\{0,1\}$. We assume that
the change cannot happen before collecting observations, i.e., 
$\pi_0 = 0$, and that for each $x \in [K]$   the corresponding  guessing and  slipping probabilities are equal, 
 that is,
\begin{align*} 
\Pro(Y_t = 0 \vert X_t = x, L_t = 1) = 
\Pro(Y_t = 1 \vert X_t = x, L_t = 0) = f_x \in (0,1), \; \;\text{ for all } t \in \bN.
\end{align*}

In Subsection \ref{subsec:simulation_geo} and \ref{sec:exp_momeryless}, we assume that both the response model 
and the   change-point model 
are completely specified, and we compare the procedure proposed in Section \ref{sec:mastery}  against the optimal in Section \ref{sec:dp}  in terms of the expected sample size ({\sf ESS}) and the actual error probabilities ({\sf Err})  for different target levels of $\alpha$. For the optimal procedure, we simulate the false alarm probability of $( \cX_c^*, T_c^*)$ in \eqref{dp_optimal}   for all values of $c$ in $\{1.2^{-a}: a = 10,11,\ldots,100\}$.  Then,   for any given $\alpha$ of interest,   we select $c(\alpha)$ to be the largest number in the above set for which the corresponding error probability does not exceed $\alpha$. This task requires extensive simulations as discussed in Section \ref{sec:dp}, even in the case of the binary response space.  For the
proposed procedure $(\mwd{\cX}, \mwd{T})$, 
as discussed  in Subsection \ref{subsec:example_finite_memory}, 
we consider the modification in \eqref{mod_rule}  and select  $z_0,\Xi_1$ and $\Xi_2$
 according to \eqref{geometric_acc_sol} and \eqref{xi2_finite_choice}  and  $b_1,b_2,d$ according to  \eqref{eq:bK_sel} and \eqref{eq:b1_d_sel}. By  Corollary  \ref{cor:geo_opt}, $(\mwd{\cX}, \mwd{T})$ is asymptotically optimal.


In Subsection \ref{subsec:simu_uncertainty}, we assume that the response and the change-point model have unknown parameters with a given joint prior {distribution},  and we compare the proposed procedure in Section \ref{sec:distr_uncertainty} against an \textit{oracle} who knows the exact {values} of the parameters.


\subsection{Finite memory models}  \label{subsec:simulation_geo}
{In the first simulation study}, 
we consider two cases for the {change-point} model in \eqref{geometric_change_point}: $K = 4, \kappa = 2$ and $K =6, \kappa = 3$.  
For  each $(x,\bS) \in [K] \times [K]^{\kappa}$ we randomly generate a number in $(0,0.08)$ for $\Psi(x, \bS)$, and these numbers are listed in  Appendix \ref{app:geo_values}.  The initial state $\bS_0$ in \eqref{Markovian_change_point}  is  $(1,\ldots,1)$. Moreover, we consider two treatments to be the most informative, so that no single treatment dominates all others in terms of detection. Specifically, we set  $f_1 > \ldots > f_{K-1} = f_K$, with values in  $(0.3,0.5)$ that  are listed  in  Appendix \ref{app:geo_values}.

The optimal procedure is implemented as described in Section \ref{sec:dp} with state space $\bR \times [K]^\kappa$.  For the proposed procedure $(\mwd{\cX}, \mwd{T})$, as discussed in Subsection \ref{subsec:example_finite_memory}, the block $\Xi_1$ is induced by a policy $\upsilon^*$, which is illustrated in the left two panels of Figure \ref{fig:geo_K4} for the two cases. Specifically, consider the upper left panel ($K=4,\kappa=2$). 
Each rectangle represents a state $(i,j) \in [4]^2$, and above its outgoing arrow is the treatment selected by $\upsilon^*$ given the current state, which leads to the next state. Since the initial state is $(1,1)$, we have $t_0 = 1, z_0 = (3)$ and $\Xi_1 = (2,1,3)$; that is, in the first acceleration stage,  treatment $3$ is used at time $1$, after which we assign treatments according to $\Xi_1$. Similarly, from the lower left panel (only a partial solution is drawn), we chose $\Xi_1 = [5,3,6,1,4]$ for the second case. The block $\Xi_2$ is induced by the policy $\upsilon'$ in  Subsection \ref{subsec:example_finite_memory}, which is illustrated in the right two panels of Figure \ref{fig:geo_K4}. As discussed there, we can use any cycle induced by $\upsilon'$. In this study,  we use  $\Xi_2 = [3,4,4,3]$ for the first case, and $[5,5, 5, 6 ,6,4]$ for the second.


 We report  in Table \ref{tab:bern} the {\sf ESS} and  {\sf Err}  
of each procedure  for different target values of $\alpha$, whereas in Figure \ref{fig:line}  we  plot {\sf ESS} against $-\log_{10}  ({\sf Err})$.  From  Table \ref{tab:bern}, for any given $\alpha$,   the {\sf ESS}  of the proposed procedure is at most 3 to 5  observations larger than the optimal. If the proposed procedure is designed so that the false alarm constraint is satisfied with (approximate) equality,  as the optimal approach does, the gap becomes even smaller; see Figure \ref{fig:line}. 



\begin{figure}[tbp!]
\subfloat[$K =4 ,\kappa =2$]{\label{fig:K4_acc}
\includegraphics[width=0.45\linewidth]{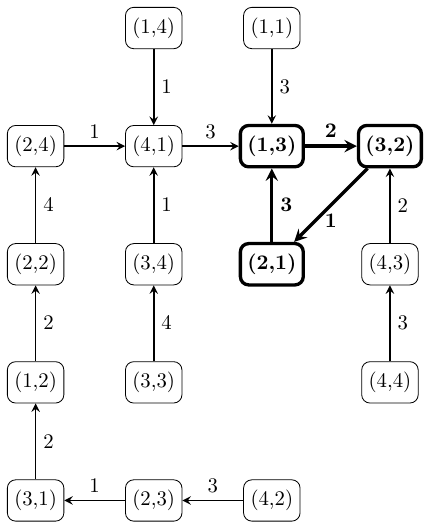}  
}
\hspace{0.15cm}
\subfloat[$K =4 ,\kappa =2$]{\label{fig:K4_det}
\includegraphics[width=0.45\linewidth]{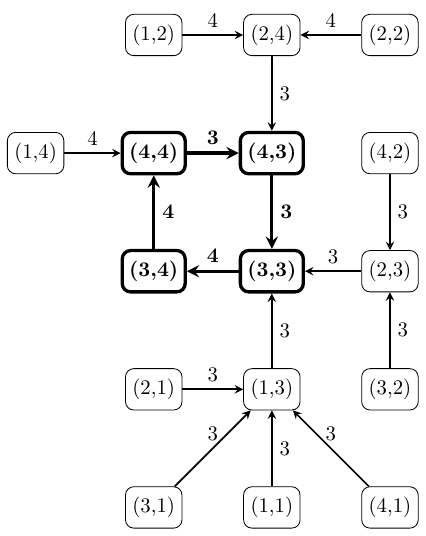}  
}\\
\subfloat[$K =6 ,\kappa =3$]{\label{fig:K6_acc}
\includegraphics[width=0.4\linewidth]{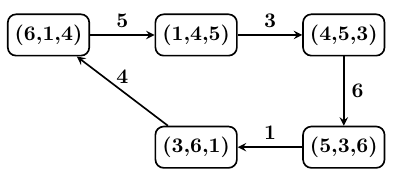}  
}
\hspace{0.3cm}
\subfloat[$K =6 ,\kappa =3$]{\label{fig:K6_det}
\includegraphics[width=0.4\linewidth]{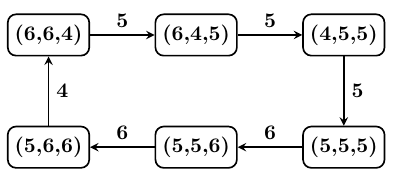}  
}

\caption{The left two panels show  solutions to Equation \eqref{geometric_accleration_opt} in Appendix \ref{app:proof_acc_block_fmm}, while the right two to Equation \eqref{geometric_det_opt} in Appendix \ref{app:proof_det_block_fmm}. Due to space constraints, only a subset of states are shown for the case $K=6,\kappa=3$.
}
\label{fig:geo_K4}
\end{figure}

\begin{table}[tbp!]
\centering
\caption{Given target level $\alpha$, we first determine the thresholds for each procedure, and then simulate the actual error probability (Err), and the expected sample size (ESS)}
\label{tab:bern}
\begin{tabular}{ccccccccc}
\hline
$\alpha$   & \multicolumn{2}{c}{0.05} & \multicolumn{2}{c}{1E-2} & \multicolumn{2}{c}{1E-3} & \multicolumn{2}{c}{1E-4} \\ \hline
 $K =4, \kappa = 2$       & Err          & ESS        & Err           & ESS       & Err           & ESS       & Err        & ESS       \\ \hline
Optimal &0.041		&29.2		& 8.6E-3		&37.6		&9.4E-4		&48.6		&8.8E-5		&60.2		 \\ 
Proposed   &0.038		&33.0		&7.8E-3     &41.9       &7.7E-4     &53.9       &7.7E-5     &65.7      \\~\\ 
\hline
 $K =6, \kappa = 3$       & Err          & ESS        & Err           & ESS       & Err           & ESS       & Err        & ESS       \\ \hline
Optimal &0.047		& 23.6		&9.3E-3		&30.3		&8.4E-4		&39.3		&9.4E-5		&47.4		 \\ 
Proposed   &0.038		& 25.9		&7.3E-3     &33.1      &7.3E-4     &41.9      &7.3E-5     &50.6     \\ \hline
\end{tabular}
\end{table}

\begin{figure}[tbp!]
\subfloat[$K=4,\kappa=2$]{\label{fig:line_K4}
\includegraphics[width=0.4\linewidth]{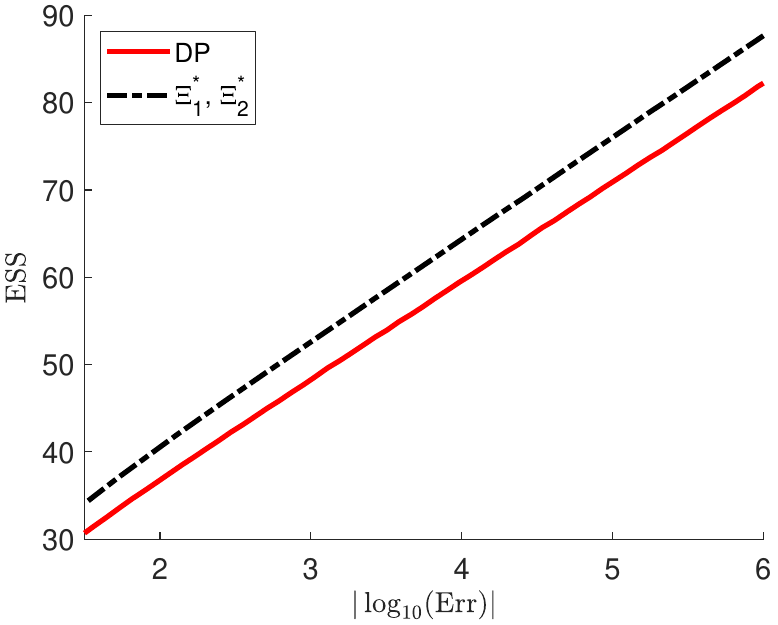}  
}
\hspace{0.1cm}
\subfloat[$K=6,\kappa=3$]{\label{fig:line_K6}
\includegraphics[width=0.4\linewidth]{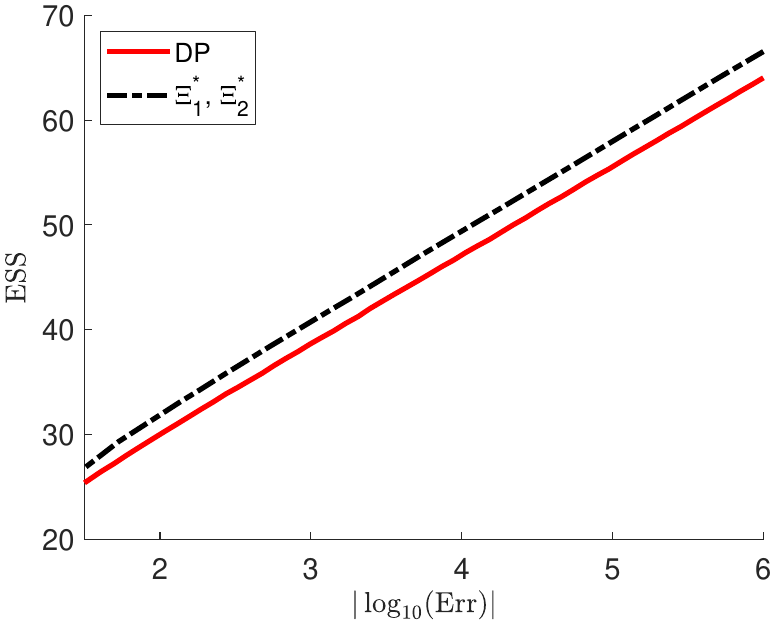}  
}
\caption{For both cases, we vary the thresholds of each procedure and plot $|\log_{10}(\text{Err})|$ vs ESS. 
}
\label{fig:line}
\end{figure}

\subsection{Parametrization by the target  level}\label{sec:exp_momeryless}
{In the second simulation study},  we consider the memoryless model, i.e., \eqref{geometric_change_point} with $\kappa = 0$, where the transition probability $\Pi_t$ only depends on $X_t$ for {each} $t \geq 1$. 
We recall the asymptotic framework in Section \ref{sec:asymptotic_framework} and  parametrize the transition probabilities by the target level $\alpha$ so that 
the asymptotically optimal expected time of change, i.e., $\lambda^*$, and the optimal detection delay, i.e., $|\log(\alpha)|/\Dro^*$, are equal, and thus both diverge as $\alpha \to 0$.

Specifically, we let $f_1 =0.4$, $f_2 = 0.35$ and $f_3 = 0.3$, and as a result, the Kullback–Leibler divergences are $I_1 = 0.0811$, $I_2 =  0.1857 $, $I_3 = 0.3389$. For each $\alpha \in (0,0.7)$, we let
$$
\Psi(1) = I_3/|\log(\alpha)|,\quad
\Psi(2) = I_3/|2\log(\alpha)|, \quad
\Psi(3) = 0.
$$
For the proposed procedure, the optimal acceleration block is $\Xi_1^* = (1)$, i.e., a block with a single treatment $1$, and the optimal detection block is $\Xi_2^* = (3)$. By Lemma \ref{lemma:finite_lb_condition}, we have
$\lambda^* = \lambda(\Xi_1^*)  = |\log(\alpha)|/I_3$ and $\Dro^* = \Dro(\Xi_2^*) = I_3$ for any $\alpha \leq 0.22$.


In Figure \ref{fig:memoryless}, as the target level $\alpha$ varies, we plot {\sf ESS} against $-\log_{10}  ({\sf Err})$    for the proposed procedure and for the optimal solution. 
Note that for the latter, 
due to the parametrization,
we need to run the dynamic programming for a collection of $c$ values for \textit{each} $\alpha$ and select $c(\alpha)$ by simulation. From Figure \ref{fig:memoryless}, we observe that the ratio of the {\sf ESS} of the proposed procedure to that of the optimal goes to $1$ as $\alpha \to 0$, which corroborates the asymptotic optimality of the proposed procedure in setups where neither the expected time of change nor the detection delay could be ignored asymptotically.

\begin{figure}[tbp!]
\includegraphics[width=0.4\linewidth]{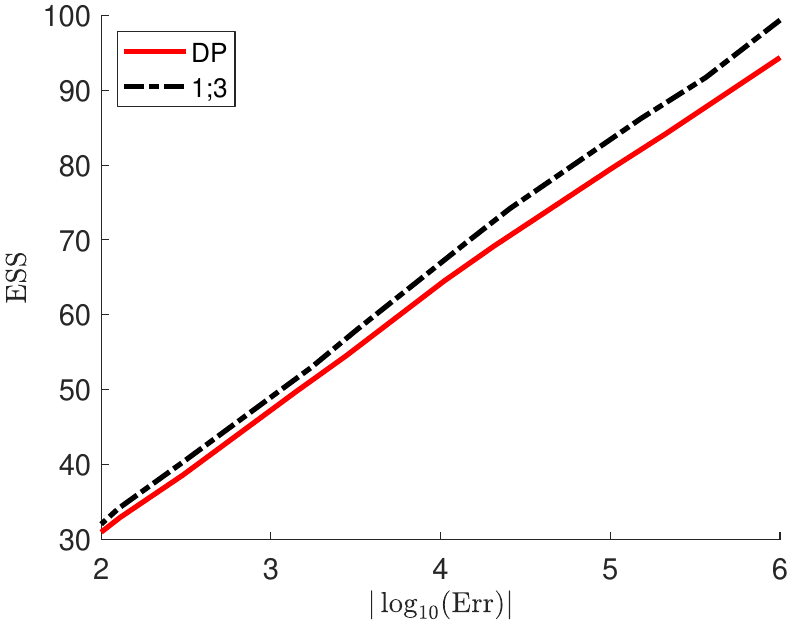}
\caption{Memoryless change-point model.}
\label{fig:memoryless}
\end{figure}

\subsection{Models with  unknown parameters}\label{subsec:simu_uncertainty}
{In the third simulation study},
we consider the memoryless model  with $K=3$  {treatments},  as in the previous subsection, but {now} both the response model and the change-point models have unknown parameters with a given prior distribution. 
Specifically, we  let $\omega_1, \omega_2 > 0$ be user-specified parameters that control the level of uncertainty, and denote by $a \mp b$  the interval $(a-b,a+b)$. We assume that the  slipping/guessing parameters  $f_1,f_2,f_3$  are uniformly distributed on  
$0.42 \mp \omega_1$, $0.39 \mp \omega_1$,  $0.36 \mp \omega_1$, respectively, and that 
the \textit{reciprocals} of $\Psi(1), \Psi(2), \Psi(3)$  are  uniformly distributed on $30 \mp \omega_2$, $40 \mp \omega_2$, $50 \mp \omega_2$, respectively.
Recall that the reciprocal of $\Psi(x)$ indicates the expected time of change under treatment $x \in [K]$.
Under the prior, all parameters are independent. The target level $\alpha$ is set to be $1\%$.

If these parameters {are} fixed to the {values} of their \textit{prior means}, then treatment $3$ has the largest detection power, while treatment $1$ the largest acceleration. Thus, for the procedure proposed in Section \ref{sec:distr_uncertainty}, we select $\Xi_1 = (1)$, $\Xi_2  = (3)$, and the thresholds based on the prior means. Further, we consider three values ($200,500,1000$) for  the number of samples $M$ from the prior distribution that are used to approximate the posteriors by Monte Carlo.

Under this setup, the dynamic programming approach in Section \ref{sec:dp} needs to be extended to include the posteriors for the unknown parameters in the sufficient statistics, which however would be computationally infeasible. Thus, we compare the proposal in Section \ref{sec:distr_uncertainty} with an \textit{oracle},  which knows the values of $f_x, \Psi(x), x \in [3]$. Specifically, in each repetition, we generate these parameters and feed them to the procedure in Section \ref{sec:mastery}.

We present the results in Table \ref{tab:parameter_uncertainty} and note that the proposed procedure in Section \ref{sec:distr_uncertainty} is not sensitive to the choice of $M$ and controls well the false alarm rate. Further, when the uncertainty regarding the parameters is small, the proposed procedure achieves similar {\sf ESS} to the oracle. However, as expected, the performance loss increases with the uncertainty.

\begin{table}[tbp!]
\caption{$M$ is the number of particles, and $\omega_1,\omega_2$ specify the prior distribution. The target level $\alpha = 1\%$.}
\begin{tabular}{ccccccccccc}
\hline
($\omega_1,\omega_2$) & \multicolumn{2}{c}{($0.01,5$)}    & \multicolumn{2}{c}{($0.02,10$)}   & \multicolumn{2}{c}{($0.03,15$)}   & \multicolumn{2}{c}{($0.04,20$)}   & \multicolumn{2}{c}{($0.05,20$)}    \\ \hline
                      & \multicolumn{1}{c}{Err}    & ESS  & \multicolumn{1}{c}{Err}    & ESS  & \multicolumn{1}{c}{Err}    & ESS  & \multicolumn{1}{c}{Err}    & ESS  & \multicolumn{1}{c}{Err}    & ESS  \\ \hline
$M = 200$             & \multicolumn{1}{c}{8.1E-3} & 68.2 & \multicolumn{1}{c}{7.7E-3} & 68.7 & \multicolumn{1}{c}{7.9E-3} & 68.9 & \multicolumn{1}{c}{8.3E-3} & 69.1 & \multicolumn{1}{c}{7.5E-3} & 71.2 \\ 
$M = 500$             & \multicolumn{1}{c}{7.8E-3} & 68.4 & \multicolumn{1}{c}{8.1E-3} & 68.3 & \multicolumn{1}{c}{7.6E-3} & 69.1 & \multicolumn{1}{c}{7.7E-3} & 70.0 & \multicolumn{1}{c}{7.8E-3} & 70.6 \\ 
$M = 1000$            & \multicolumn{1}{c}{7.8E-3} & 68.4 & \multicolumn{1}{c}{7.8E-3} & 68.6 & \multicolumn{1}{c}{8.1E-3} & 68.8 & \multicolumn{1}{c}{8.1E-3} & 69.3 & \multicolumn{1}{c}{7.9E-3} & 70.4 \\ 
Oracle                & \multicolumn{1}{c}{7.8E-3} & 68.2 & \multicolumn{1}{c}{7.7E-3} & 66.8 & \multicolumn{1}{c}{7.5E-3} & 63.8 & \multicolumn{1}{c}{7.1E-3} & 60.1 & \multicolumn{1}{c}{7.1E-3} & 59.6 \\ \hline
\end{tabular}
\label{tab:parameter_uncertainty}
\end{table}

\section{Conclusion}\label{sec:discussion}
In this work, 
we propose a generalization of the Bayesian sequential change detection problem, where the goal is not only to detect the change quickly but also to \textit{accelerate} it via adaptive experimental design.   
We first assume that the change-point and response models are completely known. We obtain the optimal solution under Markovian change-point models via dynamic programming approach, which however does not admit an explicit assignment rule and whose computation can be demanding. Thus, we propose an alternative, flexible procedure, and we establish its asymptotic optimality for a large class of (not necessarily Markovian) change-point models. Further, we extend the proposed procedure to the setup where the models have unknown parameters with a given joint prior distribution.  For future work, it would be interesting to study the acceleration of multiple changes, which  in  the educational setup   corresponds to the mastery of multiple skills \citep{baker2014educational,zhang2016smart}, 
and more complex response models in which the current response depends explicitly on previous responses.

\begin{acks}[Acknowledgments]
The authors  thank the anonymous referees, an Associate
Editor and the Editor for their constructive comments that improved the
quality of this paper.
\end{acks}

\begin{funding}
The first author was supported in part by the Natural Sciences and Engineering Research Council of Canada (NSERC) and Digital Research Alliance of Canada. The second author was supported in part by the U.S. National Science Foundation under Grants MMS 1632023 and CIF 1514245.
\end{funding}








\bibliographystyle{imsart-nameyear}
\bibliography{learning}

\begin{appendix}
\numberwithin{equation}{section}


\section{Proofs regarding the posterior odds process} \label{app:poster_odds}

Fix an assignment rule $\cX$. We  recall the definition of $\{\Lambda_t\}$ in Lemma \ref{lemma:recursive} and   set
\begin{equation} \label{LR}
\Rro^{s}_{t} \equiv \Pi_s \prod_{j=s}^{t} \frac{\Lambda_j}
{1- \Pi_j}, \quad \text{for }\; t \geq s \geq 0; \quad
\Lambda_0 \equiv 1.
\end{equation}
The following lemma justifies the interpretation of   $\Rro^{s}_{t}$ as the  ``likelihood ratio'' of the hypothesis $\Theta_{\cX}=s$ versus $\Theta_{\cX}>t$, and it is used in the  proof of Lemma \ref{lemma:recursive}, as well as in  the proof of Theorem 
\ref{thm:asymptotic_lower_bound}.

\begin{lemma} \label{lemma:change_of_distribution}
Fix integers $t \geq s \geq 0$ and an assignment rule $\cX$. For any  non-negative, measurable function $u: \bY^{t} \to [0,\infty]$  we have
\begin{equation*}
\Exp\left[u(Y_1, \ldots, Y_t);\; \Theta_\cX = s\right] = 
\Exp\left[u(Y_1,\ldots, Y_t) \; \Rro^{s}_{t};\; \Theta_\cX> t\right].
\end{equation*}
\end{lemma}

\begin{proof}
For $t \in \bN$ we set $y_{1:t} \equiv  (y_1,\ldots, y_t)$. Since $\cX$ is an assignment rule, there exists a sequence of measurable functions $\{x_j: j \geq 1\}$, such that $X_j = x_j(Y_{1:{j-1}})$.
For any non-negative measurable functions $u: \bY^{t} \to \bR$, by an iterated conditioning argument  we have
\begin{align*}
\Exp\left[u(Y_{1:t}); \Theta = s\right] &= \int u(y_{1:t}) \,  \pi_s \prod_{i=0}^{s-1} (1-\pi_i) \prod_{i=1}^{s-1}f_{x_i}(y_i) \prod_{j=s}^{t} g_{x_j}(y_j) \; d\mu^{t}(y_{1:t}), \\
\Exp\left[u(Y_{1:t}); \Theta > t\right] &= \int u(y_{1:t})  \,  \prod_{i=0}^{t} (1-\pi_i) \prod_{i=1}^{t} f_{x_i}(y_i)\; d\mu^{t}(y_{1:t}),
\end{align*}
where we drop the arguments of $\{\pi_t\}$ and $\{x_t\}$ to simplify the notation.                      
Since $u(\cdot)$ is arbitrary, in view of the definition \eqref{LR} of $\Rro^{s}_{t}$   we have
\begin{equation*}
\Exp\left[u(Y_{1:t});\; \Theta= s\right] = 
\Exp\left[u(Y_{1:t})\, \Rro^{s}_{t};\; \Theta > t\right],
\end{equation*}
which completes the proof.
\end{proof}

\begin{proof}[Proof of Lemma \ref{lemma:recursive}]
In view of Lemma \ref{lemma:change_of_distribution} and  the definition  of $\Rro^{s}_{t}$ in \eqref{LR},
for any $t \in \bN$ and $B \in \cF_t$ we have
\begin{align*}
\Exp\left[L_{t} = 1; B \right] &= 
\sum_{s=0}^{t} \Pro(B, \Theta= s) = 
\sum_{s=0}^{t} \Exp\left[\Rro^{s}_{t}; \,B, \Theta > t\right] \\
&= \Exp\left[\; \sum_{s=0}^{t}  \Rro^{s}_{t};\, B, L_{t} = 0\;\right]
= \Exp\left[\; \Pro(L_{t} = 0 \vert \cF_t) \sum_{s=0}^{t}  \Rro^{s}_{t} ;\, B  \; \right].
\end{align*}
By the definition of conditional expectation we have
$$
\Pro(L_{t} = 1 \, \vert \, \cF_t) = \Pro(L_{t} = 0  \, \vert \, \cF_t)  \, \sum_{s=0}^{t} \Rro^{s}_{t},
$$
and, in view of the definition \eqref{posterior_odds_prob} of the posterior odds,   we obtain
$$
\Gamma_t \;=\; 
\sum_{s=0}^{t} \Rro^{s}_{t} \;=\;
\sum_{s=0}^{t}   \Pi_s  \, \prod_{j=s}^{t}  \frac{\Lambda_j}
{1- \Pi_j}.
$$
Then, simple algebra shows that  $\{\Gamma_t, t \geq 0\}$ admits the recursive form given in Lemma \ref{lemma:recursive}.
\end{proof}


\begin{proof}[Proof of Lemma \ref{lemma:err_control}] \label{app:err_control}
For any integer $t \geq 0$, by the definition of $\Gamma_t$ in 
\eqref{posterior_odds_prob}  we have
$$
\Pro(L_t = 0 \,  \vert \,  \cF_t) = \frac{1}{1 + \Gamma_t}.
$$
Since $S$ is an $\{\cF_t\}$-stopping time, for any $B \in \cF_S$ we have $B \cap \{S = t\} \in \cF_t$, thus
\begin{align*}
\Pro(L_{S} = 0, B) &= \sum_{t=0}^{\infty} \Pro(L_t = 0, S = t, B) =\sum_{t=0}^{\infty} \Exp\left[ 
\Pro(L_t = 0 \,  \vert  \, \cF_t); S = t, B
\right] \\
&= \sum_{t=1}^{\infty} \Exp\left[ 
\frac{1}{1+\Gamma_t}; S = t, B
\right] = \Exp\left[\frac{1}{1+\Gamma_S}; B \right],
\end{align*}
which completes the proof by the definition of conditional expectation.
\end{proof}

\section{Proofs regarding the dynamic programming approach}\label{app:dp}

\subsection{The conditional density of responses}\label{app:density_of_Y}
In this subsection, we compute the conditional distribution of $Y_t$ given the past $\cF_{t-1}$ for $t \geq 1$.  For $t \geq 0$, denote by $\widehat{\Gamma}_t$ the posterior \textit{probability}  that the change has already occurred at time $t$, i.e.,
$$  \widehat{\Gamma}_t \; \equiv \;\Pro(L_t = 1 \, \vert \, \cF_t).
$$
Clearly $\widehat{\Gamma}_t = {\Gamma_t}/{(1+\Gamma_t)}$ for $t \geq 0$. 
For any $B \in \cB(\bY)$,  we have
\begin{align*}
\Pro(Y_{t} \in B \, \vert \,  \cF_{t-1})
 &= \Pro(Y_{t} \in B,L_{t}=1, L_{t-1} = 1 \, \vert \,  \cF_{t-1}) \\
&+\Pro(Y_{t} \in B, L_{t}=1, L_{t-1}=0 \, \vert \,  \cF_{t-1})  \\
&+ \Pro(Y_{t} \in B , L_{t}= 0, L_{t-1} = 0 \, \vert \,  \cF_{t-1}).
\end{align*}
Denote the three terms on the right-hand side by I, II, and III respectively. Then by the definition of the transition probability $\Pi_t$, we have
\begin{align*}
\text{III}&= \int_B f_{X_t}(y)  \, \Pro(L_{t}=0, L_{t-1} = 0 \, \vert \,  \cF_{t-1}) \,\mu(dy)\\
&= \int_B f_{X_t}(y) \,  (1-\Pi_{t}) \, \Pro(L_{t-1} = 0 \,  \vert \,  \cF_{t-1}) \,\mu(dy)\\
&= \int_B f_{X_t}(y) \, (1 - \Pi_{t}) \, (1-\widehat{\Gamma}_{t-1}) \,\mu(dy).
\end{align*}
By similar arguments, we have
\begin{align*}
\text{I} &= \int_B g_{X_t}(y) \, \widehat{\Gamma}_{t-1} \,\mu(dy),\quad
\text{II} = \int_B g_{X_t}(y) \Pi_{t} (1-\widehat{\Gamma}_{t-1})\, \mu(dy).
\end{align*}
Combining the three terms and replacing $\widehat{\Gamma}_t$ by $ {\Gamma_t}/{(1+\Gamma_t)}$, we have
$$
\Pro(Y_{t} \in B \vert \cF_{t-1}) = \int_B \frac{1}{1 + \Gamma_{t-1}} \left( (\Gamma_{t-1} + \Pi_t) g_{X_t}(y) + (1-\Pi_t) f_{X_t}(y) \right)\; \mu(dy).
$$
Thus, the conditional density of $Y_t$ given  
$\cF_{t-1}$, relative to $\mu$,
is  $\phi(y;\Gamma_{t-1},\Pi_{t}, X_t)$,
where for each $(y,\gamma,\pi,x) \in \bY \times {\bR}_{+} \times [0,1] \times [K]$,
\begin{align}\label{app:Y_post_density}
    \phi(y; \gamma, \pi, x) =  \frac{1}{1+\gamma} \left(
    (\gamma+\pi)g_x(y)  + (1-\pi) f_x(y)\right).
\end{align}

\subsection{Verification of the Monotone Increase structure} \label{app:mi}
In this subsection, we verify that the Markov decision process (MDP) $\{({\Gamma}_t,\bS_t, \Delta_t), t \geq 0\}$ under actions $\{U_t, t \geq 0\}$ satisfies the ``Monotone Increase'' structure \cite[Assumption I in Chapter 4.3]{bertsekas2022abstract} with $\bar{J}$ being the zero function in $\cJ$. Recall that the transition dynamic is defined in \eqref{dp:transition_S_D} and \eqref{dp:transition_Gamma}, and the cost function 
   $\Tilde{c}(\cdot)$ in Subsection 
\ref{subsec:auxilliary}. Since the cost function  $\tilde{c}$ is non-negative, this is known to be a positive cost MDP model \cite[Chapter 3, Volume II]{bertsekas1995dynamic}.

As in \cite{bertsekas2022abstract}, we define a map $H: \mathcal{S} \times \cU \times \cJ \to [0,\infty]$ as follows: for $s \in \mathcal{S}$, $u \in \cU$, and $J \in \cJ$,
$$
H(s,u,J) \equiv \Tilde{c}(s,u) + \Exp\left[ J(\Gamma_1, \bS_1, \Delta_1)\
\vert \ ({\Gamma}_0, \bS_0 , \Delta_0 ) = s,\ U_0 = u \right].
$$
Thus, for each action $u \in \cU$ and state $s\in \mathcal{S}$, 
\begin{align*}
    \cT_{c,u}(J)(s) = H(s,u,J), \quad 
    \cT_c(J)(s) = \min_{u \in \cU} H(s,u,J).
\end{align*}

Now we verify the ``Monotone Increase'' structure \cite[Assumption I in Chapter 4.3]{bertsekas2022abstract} with $\bar{J}$ being the zero function in $\cJ$. For part (a) of \cite[Assumption I in Chapter 4.3]{bertsekas2022abstract}, since $\tilde{c}(\cdot) \geq 0$ by definition, we have 
$ \bar{J} = 0 \leq H(s,u, \bar{J})$ for any $s \in \mathcal{S}$ and $u \in \cU$. For part (b), for a sequence of non-decreasing functions $\{J_m \in \cJ: m \geq 1\}$ and a function $J \in \cJ$ such that $J_m(s) \geq 0$ and $\lim_{m \to \infty} J_m(s) = J(s)$ for any $s \in \mathcal{S}$, by the monotone convergence theorem, $\lim_{m \to \infty} H(s,u,J_m) = H(s,u,J)$ for any $s \in \mathcal{S}$ and $u \in \cU$. For the last part (c), by definition, $H(s,u, J+ re) = H(s,u,J) + r$ for any scalar $r>0$, where $e \in \cJ$ is the constant function with value $1$; thus, part (c) holds with $\alpha = 1$.

\subsection{Proof of Theorem \ref{dp_structural}} \label{app:proof_structural}
The proof of Theorem \ref{dp_structural} relies on the following Lemma.

\begin{lemma} \label{lemma:concavity_of_opt}
Fix any $c>0$, $\zeta \in \bR^{\kappa}$, and $\delta \in \{0,1\}$. The function that maps $\gamma \in {\bR}_{+}$ to $
(1+\gamma)J_c^*(\gamma, \zeta,\delta) \in \bR$ is concave.
\end{lemma}
\begin{proof}
Recall from Subsection \ref{subsec:auxilliary} that 
\begin{equation*}
\lim_{t \to \infty} \cT_{c}^{\bigotimes t}(0)(\gamma,\zeta,\delta) = J_c^*(\gamma,\zeta,\delta)\;\; \text{ for all }  (\gamma,\zeta,\delta) \in \mathcal{S}.
\end{equation*}
Since the point-wise limit operation preserves concavity,   it suffices to show that for any $J \in \cJ$, 
if $\gamma \mapsto (1+\gamma)J(\gamma, \zeta,\delta)$ is concave for any fixed $(\zeta,\delta) \in \bR^{\kappa} \times  \{0,1\}$, we have that $\gamma \mapsto (1+\gamma)\cT_c(J)(\gamma, \zeta,\delta)$  is concave for any fixed $(\zeta,\delta) \in \bR^{\kappa} \times  \{0,1\}$.   

Fix $c>0$, $\zeta \in \bR^{\kappa}$, $\delta \in \{0,1\}$ and a function $J \in \cJ$ such that $\gamma \mapsto (1+\gamma)J(\gamma, \zeta',0)$ is concave for any $\zeta' \in \bR^{\kappa}$. By definition, $\cT_c(J) = \min_{u \in \cU} \cT_{c,u}(J)$. Since  the point-wise minimum operation preserves  concavity and by definition $\cT_c(J)(\cdot,\zeta,1) = 0$, it suffices to show that for any $u \in \cU$, $\gamma \mapsto (1+\gamma)\cT_{c,u}(J)(\gamma,\zeta,0)$ is concave.

Fix $u = (u^{(1)}, u^{(2)}) \in \cU = [K]\times\{0,1\}$. Recall the definitions of the operator $\cT_{c,u}$ and  the cost function $\Tilde{c}(\cdot)$ in   Subsection \ref{subsec:auxilliary}, and  the $\mu$-density $\phi$ in \eqref{app:density_of_Y}. If $u^{(2)} = 1$, then $\cT_{c,u}(J)(\gamma,\zeta,0) = 1/(1+\gamma)$ for $\gamma \in {\bR}_{+}$, and thus $\gamma \mapsto (1+\gamma)\cT_{c,u}(J)(\gamma,\zeta,0)$ is a concave function in $\gamma$. If 
$u^{(2)} = 0$ and $u^{(1)} = x \in [K]$, due to \eqref{dp:transition_Gamma}, for $\gamma \in {\bR}_{+}$,
$$
 \cT_{c,u}(J)(\gamma,\zeta,0) =   c +   \int J\left(h(y; \gamma, \pi_x, x), \zeta_x, 0 \right) \; {\phi}(y;\gamma,\pi_x,x)\;  \mu(dy),
$$
where $\zeta_x \equiv \Phi(x, \zeta)$, $
\pi_x \equiv \Psi(x, \zeta)$ and 
$$
h(y; \gamma, \pi, x) \equiv  \frac{\gamma + \pi}{1 - \pi} \frac{g_x(y)}{f_x(y)}.
$$
Note that by definition of $\phi$ in \eqref{app:density_of_Y}, 
$$\phi(y; \gamma, \pi, x) = 
 \frac{1-\pi}{1+\gamma} f_x(y) \left( 1+ h(y; \gamma, \pi, x) \right),
$$
which implies that
\begin{align*}
    (1+\gamma)  \,  \cT_{c,u}(J)(\gamma,\zeta,0) 
    &= c(1+\gamma) \\ 
     &+\int \left( 1+ h(y; \gamma, \pi_x, x) \right) J\left(h(y; \gamma, \pi_x, x), \zeta_x, 0 \right) \; (1-\pi_x) f_x(y) \mu(dy).
\end{align*}

Recall that  $x \in [K]$ and $\zeta \in \bR^{\kappa}$ are fixed. 
For any \textit{fixed} $y \in \bY$,  the first mapping below is linear in $\gamma$,  and the second is concave due to the assumption on $J$:
$$
\gamma \mapsto h(y; \gamma, \pi_x, x) \mapsto \left( 1+ h(y; \gamma, \pi_x, x) \right) J\left(h(y; \gamma, \pi_x, x), \zeta_x, 0 \right),
$$
which implies that the integrand is concave in $\gamma$. 
Since the integration operation (with respect to $y$) preserves concavity, we have that
$
\gamma \mapsto (1+\gamma)  \cT_{c,u}(J)(\gamma,\zeta,0) 
$
is concave for any fixed $\zeta \in \bR^{\kappa}$ and action $u \in \cU$. Thus, the proof is complete.
\end{proof}

\begin{proof}[Proof of Theorem \ref{dp_structural}]
From the definition of $T_c^*$ in \eqref{dp_optimal} it follows that  
\begin{align*}
&T_c^* = \inf\{t \geq 0: {\Gamma}_{t} \in B_c(\bS_t)\}, \; \text{ where }\; \\
&B_c(\zeta) \equiv \{\gamma \in {\bR}_{+}: (1+\gamma)J_c^*(\gamma,\zeta,0)  \geq 1\}, \quad \zeta \in \bR^{\kappa}.
\end{align*}

By Lemma \ref{lemma:concavity_of_opt},  $B_c(\zeta)$ is a convex subset of ${\bR}_{+}$ for any $\zeta \in \bR^{\kappa}$, which implies that $B_c(\zeta)$ is a possibly unbounded interval. Due to concavity, $\gamma \mapsto (1+\gamma)J_c^*(\gamma,\zeta,0)$ is continuous for any $\zeta \in \bR^{\kappa}$, and thus $B_c(\zeta)$ is a closed interval. 

Recall the definition of $J_c(\cdot; \cX,T)$ and $J_c^*(\cdot)$ in Subsection \ref{subsec:auxilliary}. For any procedure $(\cX,T)$ and $(\gamma,\zeta) \in \bR_{+}\times \bR^{\kappa}$
$$
J_c(\gamma,\zeta,0; \cX, T) \geq c  \, \mathbbm{1}\{U_0^{(2)} = 0\} + (1+\gamma)^{-1} \, \mathbbm{1}\{U_0^{(2)} = 1\} \geq \min\{c, (1+\gamma)^{-1}\}.
$$
As a result,
$$
(1+\gamma) \,  J_c^*(\gamma,\zeta,0) \equiv (1+\gamma) \inf_{(\cX, T)} J_c(\gamma,\zeta,0; \cX, T) \geq \min\{c(1+\gamma), 1\},
$$
which implies that $[1/c,\infty) \subset B_c(\zeta)$ for any $\zeta \in \bR^{\kappa}$. 
Since $B_c(\zeta)$ is a closed interval containing $[1/c,\infty)$, $B_c(\zeta)$ must be of the form $[b_c(\zeta),\infty)$ for some function $b_c: \bR^{\kappa} \to \bR_{+}$, which completes the proof.
\end{proof}

\section{Proof of  Theorem \ref{thm:upper_bound_T} } \label{app:proofs_regarding_upper_bound}

In this section, we prove  Theorem \ref{thm:upper_bound_T}, which provides a non-asymptotic upper bound on the average sample size of the proposed procedure in Section \ref{sec:mastery}.  
We denote by $\Delta S_n$ the duration of $n^{th}$ stage, 
 i.e.,  $\Delta S_n\equiv S_n-S_{n-1}$, where  $n \in \bN$,   and observe that  
\begin{equation*}
S_{2 N}= \sum_{m=1}^{\infty} ( \Delta S_{2m-1} + \Delta S_{2m}) \; \ind{\{N \geq m\}},
\end{equation*}
 where $N$ is the  number of cycles until stopping. For every $m \in \bN$ we have $\{{N} \geq m\} \in \cF_{{S}_{2m-2}} \subset \cF_{{S}_{2m-1}}$, therefore by the law of iterated expectation we obtain
\begin{align}\label{proposed_T_decomp}
\Exp[\mwd{T}] \; = \;  \sum_{m = 1}^{\infty} \Exp\left[\,\Exp[\Delta S_{2m-1} \, \vert \,  \cF_{S_{2m-2}}] +  \Exp[\Delta S_{2m}  \, \vert \,  \cF_{S_{2m-1}}];\, N \geq m \right].
\end{align}
For each $m \in \mathbb{N}$,  we establish in  Subsection \ref{subsec:upper_each_stage} an upper bound on the conditional expected lengths of the acceleration stage  and the detection stage  in the $m^{th}$ cycle,  $\Exp[\Delta S_{2m-1} \vert \cF_{S_{2m-2}}]$ and  $\Exp[\Delta S_{2m} \vert \cF_{S_{2m-1}}]$,
respectively. These bounds are   deterministic and do not depend on the cycle index $m$, which implies that  the resulting upper bound for $\Exp[\mwd{T}]$
is proportional to the expected number of cycles, $\Exp[N]$.  In Subsection \ref{subsec:upp_cycles}  we establish an upper bound on $\Exp[N]$, and in Subsection \ref{subsec:upper_bound_T} we combine
 these two bounds to complete the proof.


We start by introducing preliminaries  that are used throughout this Appendix.  The first one is a  representation of  $\{L_0, L_{t}, Y_{t},  t \in \bN\}$ in terms of dynamic system equations. Since the response space, $\bY$, is  assumed to be  Polish,  there exists
 \cite[Lemma 3.22]{kallenberg2006foundations}  a measurable function $h: [K] \times \{0,1\} \times (0,1) \to \bY$ such that, 
for each $x \in [K]$,   $h(x,0,V)$ (resp. $h(x,1,V)$) has density $f_x$ (resp. $g_x$) relative to the reference measure $\mu$, where $V\sim \text{Unif}(0,1)$.
 
   Now, if $\{U_t\}$ and $\{V_t\}$ are two independent sequences of i.i.d. $\text{Unif}(0,1)$ random variables on $(\Omega, \cF,\Pro)$,  and for any given assignment rule $\cX$ we set $L_0 \equiv  \mathbbm{1}\{U_0 \leq \pi_0\}$ and 
\begin{align}
\begin{split}
\label{sys_equ}
L_{t} &= \mathbbm{1}{\{L_{t-1}=1\}} + \mathbbm{1}{\{L_{t-1} =0, \; U_t \leq\, \Pi_t \}} 
 \quad \text{ and } \quad 
 Y_t =  h(X_t, L_t, V_t), \quad t \in \bN,
\end{split}
\end{align}
where $\mathbbm{1}\{\cdot\}$ is the indicator function, then 
 $\{L_0, L_{t}, Y_{t},  t \in \bN\}$ has the distribution as described  in Section \ref{sec:ProbFormulation}. Moreover, for any  integer $t \geq 0$ we denote  by $\cH_t$  the $\sigma$-algebra that is generated by all  sources of randomness up to time $t$, i.e., 
\begin{equation}
\label{def:hidden_filtration}
\cH_t \equiv \sigma(U_0, U_s, V_s: 1 \leq s\leq t).
\end{equation}
Clearly,  $\cF_t \subset \cH_t$ for every $t \in \bN$, and $\Theta$ is an $\{\cH_t\}$-stopping time.

%
%
%
%


We continue with two observations regarding the proposed procedure.  First, when $b_1$ is sufficiently large,  with high probability  the change  has already occurred  by the end of an acceleration stage.  Second,  when $d$ is sufficiently large,
 with a high probability, the change is detected in the first detection stage after its occurrence. We state these observations formally in the following lemma. Here, and in what follows, if $n \in \bN$ is an \textit{odd} integer we set 
\begin{equation} \label{B_n}
B_n \equiv \left\{\tau(S_{n},d) < \sigma(S_{n},b_2) \right\}, 
\end{equation}
i.e.,  $B_{n}$ is the event that the detection is negative in the detection stage starting at $S_n+1$.


\begin{lemma}\label{lemma:aux_change_in_training}
If  $n \in \bN$ is \textit{odd}, then 
on the event $\{S_{n} < \widetilde{T}\}$ we have
\begin{align*}
\Pro\left(L_{S_{n}} = 0  \, \vert \,  \cF_{S_{n}} \right) &\leq 1/(1+b_1),\\
\Pro(B_{n}, L_{S_{n}} = 1 \, \vert \,  \cF_{S_{n}}) &\leq (1/d) \, \Pro\left(L_{S_{n}} = 1 \, \vert  \, \cF_{S_{n}} \right).
\end{align*}
\end{lemma}

\begin{proof}
By the definition of the proposed procedure,     the posterior odds process at the end of an acceleration stage is not smaller than  $b_1$, i.e., $\Gamma_{S_{n}} \geq b_1$ for every odd $n \in \bN$. As a result,  the first inequality follows by  Lemma \ref{lemma:err_control}. 
To prove the second inequality, we  recall  representation \eqref{sys_equ} of the responses $\{Y_t\}$ in terms of the ``noise'' sequence  $\{V_{t}\}$, and we  set
\begin{equation*}
Y'_t \equiv h(\Xi_2(t), 1, V_{S_n+t} ), \quad t  \in \bN.
\end{equation*}
That is,  $\{Y'_t\}$ is the sequence of responses that are obtained after $S_n$  if the detection block $\Xi_2$ is assigned \textit{forever} after $S_n$  and the change had already happened by $S_n$.  
These responses coincide with the actual responses of the proposed procedure after $S_{n}$ and until the end of the subsequent detection stage $S_{n+1}$,  on the event $\{L_{S_n} = 1\}$.   Indeed,  for 
$1\leq t \leq S_{n+1} - S_n$,
$$
Y_{S_n+t} = h(X_{S_n+t},1,V_{S_n+t}) = h(\Xi_2(t),1,V_{S_n+t}) =  Y'_t.
$$
Thus, on the event $B_{n} \cap \{L_{S_n} = 1\}$, where the change has already occurred at $S_n$ and the subsequent detection is negative,  
\begin{equation*} 
S_{n+1} - S_n= \tau(S_n,d) =  \inf\left\{t \geq 1: \prod_{j=1}^{t} \frac{f_{\Xi_2(j)}(Y'_j)}{g_{\Xi_2(j)}(Y'_j)} \; \geq d \right\} \equiv \tau'.
\end{equation*}
As a result, 
\begin{align*}
\Pro(B_{n}, L_{S_n} = 1 \, \vert \, \cF_{S_n}) 
&= \Pro( L_{S_n} = 1, \; \tau(S_n, d) < \sigma(S_n, b_2) \, \vert \, \cF_{S_n})\\
&\leq \Pro(L_{S_n} = 1, \tau' < \infty \, \vert \,  \cF_{S_n})\\
&= \Exp[ \Pro(\tau' < \infty \, \vert \,  \cH_{S_n}); \, L_{S_n} = 1 \, \vert \, \cF_{S_n}].
\end{align*}
Since $S_n$ is an $\{\cH_t\}$-stopping time,  by Lemma \ref{lemma:iid_independent} it follows that  $\{V_{S_n+t}, t \geq 1\}$ are i.i.d. $\text{Unif}(0,1)$ and  independent of $\cH_{S_n}$. Thus, $\{Y_t': t\geq 1\}$ is a sequence of independent random variables, independent of $\cH_{S_n}$, and  each $Y_{t}'$ has density $g_{\Xi_2(t)}$ relative to $\mu$. Thus, we can apply  Lemma \ref{lemma:one_sided_SPRT} and obtain
$\Pro(\tau' < \infty \vert \cH_{S_n}) \leq 1/d$, and this  completes the proof.
\end{proof}


\subsection{Upper bound on the average number of cycles} \label{subsec:upp_cycles}

In Lemma \ref{lemma:N_upp_bd} we establish an upper bound on the expected number of cycles, which implies that with high probability there is only \textit{one} cycle when the values of $b_1$ and $d$ are sufficiently large. We note that this lemma does not require any assumption on the change-point model.

\begin{lemma}\label{lemma:N_upp_bd}
Assume \eqref{assumptions:response_KL}  holds.   For any  $b_1,d > 1$ and $m  \in \bN$,
$$
\Pro(N > m) \leq \eta^{m}, \quad \text{ where }\; \eta \equiv \frac{b_1 +d}{d(1+b_1)}.
$$
Consequently, $\Exp[N] \leq 1 + \eta/(1-\eta)$ and  $\Exp[{N}] \rightarrow  1$ as 
 $\min\{b_1, d\} \to \infty$.
\end{lemma}

\begin{proof}
For any  $m  \in \bN$, $B_{2m-1}$, defined in \eqref{B_n}, is the event that the detection in the $m^{th}$ cycle is negative. Thus, 
\begin{align} \label{telescope}
\begin{split}
\Pro({N} > m) &= \Pro({N} > m -1, B_{2m-1})
= \Exp[ \Pro(B_{2m-1} | \cF_{{S}_{2m-1}}); {N} > m-1].
\end{split}
\end{align}
By Lemma \ref{lemma:aux_change_in_training} we have
\begin{align*}
\Pro(B_{2m-1} \vert \cF_{S_{2n-1}})&
\leq \Pro(B_{2m-1},\,  L_{S_{2m-1}}  = 1 \, \vert \, \cF_{S_{2m-1}})
+ \Pro(L_{S_{2m-1}} = 0 \, \vert \, \cF_{S_{2m-1}}) \\
& \leq (1/d) \, \Pro(L_{S_{2m-1}} = 1 \, \vert \,  \cF_{S_{2m-1}}) + \Pro(L_{S_{2m-1}} = 0 \, \vert \, \cF_{S_{2m-1}}) \\
& = 1/d + (1-1/d) \,  \Pro(L_{S_{2m-1}} = 0 \, \vert \, \cF_{S_{2m-1}}) \\
&\leq 1/d +  (1 - 1/d)/ (1+b_1) \equiv \eta.
\end{align*}
The result now follows after applying a  telescoping argument to  \eqref{telescope}.
\end{proof}

%
%

\subsection{Upper bounds on the conditional average duration of each stage} \label{subsec:upper_each_stage}

In this Subsection, we establish an upper bound on the conditional expected duration of each stage, $\Delta S_n$. In order to do so, we couple the system  $\{\Pi_{t}, L_{t}, \mwd{X}_{t}, Y_{t},\Gamma_{t}; \,  t \in \bN \}$ that is associated with the proposed assignment rule with hypothetical systems that \textit{coincide with the original during the stage of interest}, but are easier to analyze. 



Thus, for any $n \in \bN$ we set 
\begin{align*}
\begin{cases}
x^n_t = \Xi_1(t), \quad b^n = b_1, \quad \Xi^n = \Xi_1, \quad \text{ if } n \text{ is odd}, \\
x^n_t = \Xi_2(t), \quad b^n = b_2, \quad \Xi^n = \Xi_2,\quad \text{ if } n \text{ is even}. \\
\end{cases}
\end{align*}
and define  $\{\Pi^{n}_t, L^{n}_t, Y^{n}_t,\Gamma^{n}_t: t\geq 1\}$ to be a system that describes the evolution of transition probabilities, latent states, responses, and posterior odds process of the original system after time $S_{n-1}$ \textit{if we assign block $\Xi^n$  forever after this time}.  Specifically,  we write $S$ for $S_{n-1}$ for simplicity,  set
$L^{n}_0 \equiv  L_{S}$, $\Gamma^{n}_0 \equiv \Gamma_{S}$, and   for each $t \in \bN$
\begin{align}\label{aux_couple_dynamic}
\begin{split}
\Pi^{n}_t &\equiv \pi_{S+t}(X_1,\ldots,X_S, \Xi^n(1:t)),\\
L^{n}_{t} & \equiv \mathbbm{1}{\{L^{n}_{t-1}=1\}} + \mathbbm{1}{\{L^{n}_{t-1} =0, \; U^{n}_t \leq\, \Pi^{n}_{t}\}}, \\
   Y^{n}_{t} &\equiv  h(\Xi^n(t), L^{n}_{t}, V^{n}_{t}),  \\
\Gamma^{n}_{t} &\equiv   (\Gamma^{n}_{t-1} + \Pi^{n}_t)\; \frac{g_{\Xi^n(t)}(Y^{n}_t)}
{(1-\Pi^{n}_t)f_{\Xi^n(t)}(Y^{n}_t)},
  \end{split}
\end{align}
where $(U^{n}_t, V^{n}_t) \equiv (U_{S+t}, V_{S+t})$ is the same ``noise'' that drives the original system after time $S$ (see \eqref{sys_equ}). The evolution of this hypothetical system  coincides with the $n^{th}$ stage of the original system,  in the sense that for any $1 \leq t \leq {S}_n - S_{n-1}$,
\begin{equation}\label{hypothetical_coincide}
(\Pi_{S+t}, L_{S+t},X_{S+t}, Y_{S+t},\Gamma_{S+t}) = (\Pi^{n}_t, L^{n}_t,\Xi^n(t), Y^{n}_t,\Gamma^{n}_t).
\end{equation}
Let $\Theta^n$ denote the change-point of the above  hypothetical system, and let
$\rho^n$ denote the required time after  $\Theta^n$ for the posterior odds process of this hypothetical system to reach the corresponding threshold $b^n$, i.e., 
\begin{align}\label{def:hypothetical_change_point}
\Theta^n \equiv \inf\{t\geq 1: L^{n}_t = 1\}, \quad
\rho^{n} \equiv \inf\{t \geq 0: \Gamma^{n}_{\Theta^{n} +t} \geq b^n\},
\end{align}
where $\rho^{n}$ is well defined only on the event $\{\Theta^{n} < \infty\}$. We can now obtain a pathwise bound for the actual duration of the $n^{th}$ stage, $\Delta S_n$,  of the proposed proposed procedure in terms of the quantities  $\rho^{n}$ and $\Theta^{n}$ of the above hypothetical system. 

\begin{lemma}\label{lemma:orginal_hypothetical_fir}
For each $n \in \bN$  we have   
$
\Delta S_{n} \leq \Theta^{n} + \rho^n \, \mathbbm{1}_{\{\Theta^n < \infty \}}.
$ 
\end{lemma}

\begin{proof}
For each $n \in \bN$ we define $\sigma^{n}$ to be the first time the posterior odds process of the above hypothetical system,  $\Gamma^{n}$, 
exceeds  threshold $b^n$, i.e., 
$$
\sigma^n \equiv \inf\{t\geq 1: \Gamma^{n}_t \geq b^{n}\}.
$$
In view of the definition of $\rho^{n}$ in \eqref{def:hypothetical_change_point} we have $\sigma^n \leq \Theta^{n} + \rho^n \, \mathbbm{1}_{\{\Theta^n < \infty \}}$, thus it suffices to show that $\Delta S_n \leq \sigma^n$.
If  stopping in the $n^{th}$ stage is triggered by the detection rule, i.e. $\Gamma_{S_n} \geq b^{n}$, then $\Delta S_n = \sigma^n$ due to \eqref{hypothetical_coincide}. Otherwise, the posterior odds of the original system do not exceed $b^n$ in the $n^{th}$ stage, and again due to \eqref{hypothetical_coincide} we have $\Delta S_n  < \sigma^n$. In any case, we have $\Delta S_n \leq \sigma^n$, and the proof is complete.
\end{proof}


Lemma \ref{lemma:orginal_hypothetical_fir} is sufficient for bounding the duration of an acceleration stage, as well as that of a detection stage when the change has already occurred at its beginning. However, if the change has not occurred at the beginning of a detection stage,  
we need a better bound on its duration. 
Thus, for  \textit{even} $n \in \bN$ we denote by 
$\{\widehat{Y}^{n}_t, t \in \bN\}$ the responses that would be  obtained after $S_{n-1}$ if block $\Xi_2$ was applied repeatedly after $S_{n-1}$ and \textit{the change never happened}, i.e., 
\begin{equation}\label{aux_Y_widehat}
\widehat{Y}^{n}_t \equiv h(\Xi_2(t), 0, V^{n}_t ), \quad  t  \in \bN,
\end{equation}
where as before $\{V^{n}_t \equiv V_{S_{n-1} + t}: t \geq 1\}$ is the ``noise'' that drives the original system \eqref{sys_equ}  after time $S_{n-1}$.     Indeed,  for $t \leq \min\{\Theta^n-1,\ S_{n} - S_{n-1}\}$,  we have
\begin{equation}
\label{aux_Y_widehat_couple}
\widehat{Y}^{n}_t =  h(\Xi_2(t), 0, V^{n}_t )
=  h(X_{S_{n-1}+t}, L_{S_{n-1}+t}, V_{S_{n-1}+t}) = Y_{S_{n-1}+t},
\end{equation}
and for  $t \leq (\Theta^n-1)$  we have
\begin{equation}
\label{aux_Y_widehat_couple_2}
\widehat{Y}^{n}_t =  h(\Xi_2(t), 0, V^{n}_t ) = Y^{n}_{t}.
\end{equation}
Further, we let  ${\tau}^n$ denote  the number of responses required by  test statistic  associated with $\{\widehat{Y}^{n}_t\}$ to exceed  threshold $\log(d)$, i.e., 
\begin{equation}\label{def:hyp_sprt}
{\tau}^n \equiv \inf\left\{t \in \bN: \sum_{j=1}^{t} \log\left(\frac{f_{\Xi_2(j)}(\widehat{Y}^{n}_j)}{g_{\Xi_2(j)}(\widehat{Y}^{n}_j)}\right)\; \geq \log(d)\right\}.
\end{equation}



\begin{lemma}\label{lemma:orginal_hypothetical_sec}
For any even  $n \in \bN$ we have
$$
\Delta S_{n} \leq \tau^{n} + \rho^{n} \, \mathbbm{1}_{\{\Theta^n \leq \tau^n < \infty\}}.
$$
\end{lemma}

\begin{proof}
For \textit{even} $n \in \bN$,  $\Delta S_n$  represents the duration of a detection stage.  We focus on the event  $\{\tau^n < \infty \}$, since otherwise the result holds trivially. By Lemma \ref{lemma:orginal_hypothetical_fir}, 
\[
\Delta S_n   \leq \Theta^n + \rho^n \mathbbm{1}_{\{\Theta^n < \infty \}},
\]
which is used when the event $\{ \Theta^n \leq \tau^n \}$ happens. 
Further, in view of \eqref{aux_Y_widehat_couple} and \eqref{def:hyp_sprt}, 
$$\Delta S_{n} \leq \tau^n \;  \text{ on } \quad  \{\tau^n < \Theta^n\}.$$
Combining these two bounds we obtain:
\begin{align*}
\Delta S_n &= (\Delta S_n) \mathbbm{1}_{\{{\tau}^n < \Theta^n\} }  + (\Delta S_n)  \, \mathbbm{1}_{\{\Theta^n \leq {\tau}^n \} } \\
&\leq \tau^n \, \mathbbm{1}_{\{{\tau}^n < \Theta^n\} } + (\Theta^n + \rho^n)  \, \mathbbm{1}_{\{\Theta^n \leq \tau^n < \infty \}} \\
&\leq \tau^n  \, \mathbbm{1}_{\{{\tau}^n < \Theta^n\} } + (\tau^n + \rho^n) \, \mathbbm{1}_{\{\Theta^n \leq \tau^n < \infty \} } 
= \tau^n + \rho^n \, \mathbbm{1}_{\{\Theta^n \leq {\tau}^n < \infty\} },
\end{align*}
which completes the proof.
\end{proof}

The next step in the proof is to  upper bound the stopping rule $\rho^n$, defined in \eqref{def:hypothetical_change_point}, associated with the hypothetical system \eqref{aux_couple_dynamic}. Recall the definition \eqref{def:hidden_filtration} of $\{\cH_t: t \geq 0\}$.

\begin{lemma} \label{lemma:sigma_n_upp2}
Suppose \eqref{assumptions:response_KL} and \eqref{assumption_stability} hold. Fix $n \in \bN$ and denote by $\ell$ the length of block $\Xi^n$. Further, fix $0 \leq j < 2\ell$ and $\epsilon > 0$.\\


\noindent (i). On the event $\{\Theta^n < \infty \}$,
\begin{align*}
  \rho^n \leq  \min\left\{\; \inf\{t \geq 0: {Z}^{n}_t  \geq \log(b^n)- \log(\Gamma^{n}_{\Theta^n+j-1} + \Pi^{n}_{\Theta^n+j})\},\;  M\left(\Xi^n, \frac{\epsilon \, \Dro(\Xi^n)}{1+\epsilon} \right) \;\right\}+ j,
\end{align*}
where $\{Z^{n}_t: t \geq 0\}$ is a process after the change-point $\Theta^{n}$:
\begin{align*}
{Z}^{n}_t &\equiv \sum_{s = \Theta^n+j}^{\Theta^n + t+j} \left[ \log \left(\frac{g_{\Xi^n(s)}(Y^{n}_s)}{f_{\Xi^n(s)}(Y^{n}_s)} \right) +d(\Xi^n)-  \frac{\epsilon }{1+ \epsilon} \, \Dro(\Xi^n) \right]\;\;\; \text{ for } \,  t \geq 0.
\end{align*}

\noindent(ii) Set $Z^{n}_{-1} = 0$. On the event  $\{\Theta^n <\infty \}$, $\{\Delta_t \equiv {Z}^{n}_t -Z^{n}_{t-1}: t \geq 0\}$ is  a sequence of independent random variables that is independent of $\cH_{S_{n-1} + \Theta_{n}+j-1}$. Further, for any  $t \geq 0$, $\Delta_t$ and $\Delta_{t+\ell}$ have the same distribution, and
$$
\frac{1}{\ell}\sum_{j=1}^{\ell}\Exp\left[ \Delta_j \right] = \frac{\Dro(\Xi^n)}{1+\epsilon}, \qquad\;\;
\frac{1}{\ell} \sum_{j=1}^{\ell}\Exp\left[  \left( \Delta_j - \Exp[\Delta_j] \right)^2 \right]  = V^{I}(\Xi^n). $$
\end{lemma}

\begin{proof} We assume $j = 0$, since the argument for general $j$ is similar.

(i) Due to assumption \eqref{assumption_stability}, for $M > M\left(\Xi^n, \epsilon \Dro(\Xi^n) / (1+\epsilon) \right)$,
$$
\left\vert \frac{1}{M} \sum_{s = \Theta^n}^{\Theta^n+M-1}\left\vert \log(1-\Pi^n_{s})\right\vert -d(\Xi^n)  \right\vert \; \leq \frac{\epsilon}{1+ \epsilon} \Dro(\Xi^n).
$$
By applying telescoping argument to the recursion \eqref{aux_couple_dynamic} of $\{\Gamma^n_t: t \geq 0\}$, 
and for $t > M\left(\Xi^n, \epsilon \Dro(\Xi^n) / (1+\epsilon) \right)$, we have
\begin{align*}
\log \Gamma^{n}_{\Theta^n + t} & \geq \sum_{s = \Theta^n}^{\Theta^n + t} \left( \log\left(\frac{g_{\Xi^n(s)}(Y^{n}_s)}{f_{\Xi^n(s)}(Y^{n}_s)}\right) + \vert \log(1-\Pi^{n}_s)\vert \right)+ \log(\Gamma^{n}_{\Theta^n-1} + \Pi^{n}_{\Theta^n})  \\
& \geq Z^n_t + \log(\Gamma^{n}_{\Theta^n-1} + \Pi^{n}_{\Theta^n}),
\end{align*}
which completes the proof.

%

\vspace{0.2cm}
\noindent (ii) In view of \eqref{aux_couple_dynamic} and the  definition of $\Theta^n$,  for $t \geq 0$ we have 
 $$
Y^{n}_{\Theta^n + t} = h(\Xi^n(\Theta^n + t), L^{n}_{\Theta^n+t}, V^{n}_{\Theta^n+t}) = 
h(\Xi^n(\Theta^n + t),1,V^{n}_{\Theta^n+t}).
$$
Due to Lemma \ref{lemma:iid_independent}, we have that 
$$\{ V^{n}_{\Theta^n+t}: t \geq 0\} \;
=\;
\{ V_{S_{n-1}+\Theta^n+t}: t \geq 0\} 
$$
are i.i.d. $\text{Unif}(0,1)$ and independent of $\cH_{S_{n-1}+\Theta^n-1}$. 
Then the proof is complete due to the periodic structure of the assignment rule.
\end{proof}

\begin{remark}
In view of (i) in the above lemma, to obtain a further upper
bound on $\rho^n$ we have to obtain a lower bound on
$\log(\Gamma^{n}_{\Theta^n+j-1} + \Pi^{n}_{\Theta^n+j})$,
which is achieved by conditioning on different events.
\end{remark}

We can now establish an upper bound on the conditional expected duration of each stage. Note that since there is typically one cycle, as implied by Lemma \ref{lemma:N_upp_bd},  it is  useful to have a stronger bound for the first acceleration stage than the remaining ones.

\begin{lemma}
\label{lemma:upper_each_stage}
Assume conditions \eqref{assumptions:response_KL} and \eqref{assumption_stability} hold. Let $\epsilon > 0$, $m \in \bN$, $b_2 \geq b_1 \geq 1$,  $d \geq 1$.

\noindent (i) If  $n \in \bN$ is odd, then $\Exp\left[\Delta S_{n} \;\vert\; \cF_{{S}_{n-1}}\right]$ is upper bounded by
\begin{align*}
\mwd{\lambda}(\Xi_1) + \frac{\log(b_1)+|\log(\zeta(\Xi_1))|}{\Dro_{1}(\Xi_1)}(1+\epsilon) + \frac{V^{I}(\Xi_1)}{\Dro^2(\Xi_1)} \, (1+\epsilon)^2 + M\left(\Xi_1, \frac{\epsilon \,  \Dro(\Xi_1)}{1+\epsilon}\right) + 3\ell_1,
\end{align*}
with  $\mwd{\lambda}(\Xi_1)$ replaced by ${\lambda}(\Xi_1)$ when $n=1$.

\noindent (ii) If $n \in \bN$ is even, 
$ \Exp[\Delta S_{n} \vert \cF_{S_{n-1}}]$ is upper bounded by 
\begin{align*}
 &\frac{\log(b_2/b_1)}{\Dro(\Xi_2)} \, (1+\epsilon) + 
 \frac{\log(d)}{1+b_1} \left( \frac{1}{J(\Xi_2)}+\frac{1}{\Dro(\Xi_2)} \right) \, (1+\epsilon) \\
&\; +\;\frac{V^{I}(\Xi_2)}{\Dro^2(\Xi_2)} \, (1+\epsilon)^2+ 
 \frac{V^{J}(\Xi_2)}{2 J^2(\Xi_2)} 
 +M\left(\Xi_2, \frac{\epsilon \, \Dro(\Xi_2)}{1+\epsilon}\right) + 2\ell_2.
\end{align*}
\end{lemma}


\begin{proof}


Fix  $b_2 \geq b_1 \geq 1$ and $d\geq 1$. 

$(i)$  Fix  $n \in \bN$ odd. We only show the first claim, since the second can be proved by the same argument  using the definition  of $\lambda(\Xi_1)$ in \eqref{block_expected_change}.
By assumption \eqref{def:zeta}  there exists some random integer $0 \leq j^* < 2\ell_1$ such that
$$ \Pi^{n}_{\Theta^n + j^*} \geq \zeta(\Xi_1)
,$$
which implies that
$$
\log \left( \Gamma^{n}_{\Theta^n-1+ j^*} + \Pi^{n}_{\Theta^n+ j^*} \right) \geq \log(\Pi^{n}_{\Theta^n+ j^*})\geq \log(\zeta(\Xi_1)).
$$
Thus, by Lemma \ref{lemma:orginal_hypothetical_fir} and Lemma \ref{lemma:sigma_n_upp2}(i) (with $j = j^*$) we have
$\Delta S_n  \leq \Theta^n + \mwd{\rho}^n$, where 
$$
\mwd{\rho}^n   \equiv \min \left\{ \inf\{t \geq 0: {Z}^{n}_t \geq \log(b_1) + |\log(\zeta(\Xi_1))|\},\;
M\left(\Xi_1, \frac{\epsilon \, \Dro(\Xi_1)}{1+\epsilon}\right)
\right\} + 2\ell_1
.
$$
By the definition \eqref{def:hypothetical_change_point} of $\Theta^n$ and \eqref{block_expected_change} of $\mwd{\lambda}(\Xi_1)$, 
\begin{align*}
\Exp[\Theta^n \, | \,  \cF_{S_{n-1}}] \leq \mwd{\lambda}(\Xi_1).
\end{align*}
Further, since $\cF_{S_{n-1}} \subset \cH_{S_{n-1} + j^*+ \Theta^n-1}$, by  Lemma \ref{lemma:sigma_n_upp2}(ii) and \ref{lemma:lorden_renewal} we obtain
\begin{align*}
\Exp[\mwd{\rho}^n \, \vert \,  \cF_{S_{n-1}}] \; \leq \; & \frac{\log(b_1) + |\log(\zeta(\Xi_1))| }{\Dro(\Xi_1)}(1+\epsilon)  +  \frac{V^{I}(\Xi_1)}{\Dro^2(\Xi_1)} \, (1+\epsilon)^2 + \ell_1 \\
&+  M\left(\Xi_1, \frac{\epsilon \, \Dro(\Xi_1)}{1+\epsilon}\right) + 2\ell_1,
\end{align*}
which completes the proof of (i). \\


%
%
%
%
$(ii).$ Fix  $n \in \bN$ \textit{even}.  On the event $\{L_{S_{n-1}} = 1\}$  we have $\Theta^n = 1$. Further, by the definition of the proposed assignment rule and
 definition \eqref{aux_couple_dynamic},  on the event $\{L_{S_{n-1}} = 1\}$ we have 
\begin{equation*}
b_1 \leq \Gamma_{S_{n-1}} = \Gamma^{n}_0 = \Gamma^{n}_{\Theta^n-1} \quad  \Rightarrow \quad 
\log(\Gamma^{n}_{\Theta^n-1} + \Pi^{n}_{\Theta^n}) \geq \log(b_1).
\end{equation*}
Thus, by Lemma \ref{lemma:orginal_hypothetical_fir} and \ref{lemma:sigma_n_upp2}(i) (with $j = 0$), on the event $\{L_{S_{n-1}} = 1\}$, 
\begin{align*}
\Delta S_n \leq 1 + \min\left\{ \inf\{t \geq 0: {Z}^{n}_t  \geq \log(b_{2})- \log(b_1) \},\; \; M\left(\Xi_2, \frac{\epsilon \, \Dro(\Xi_2)}{1+\epsilon}\right)\right\}.
\end{align*}
Then,  due to Lemma \ref{lemma:sigma_n_upp2}(ii) and Lemma \ref{lemma:lorden_renewal}, 
$$
\Exp[\Delta S_n \vert \cH_{S_{n-1}}] \leq 
1 + \frac{\log(b_2/b_1)}{\Dro(\Xi_2)} (1+\epsilon) + \frac{V^{I}(\Xi_2)}{\Dro^2(\Xi_2)}
 (1+\epsilon)^2 + \ell_2 +
M\left(\Xi_2, \frac{\epsilon \, \Dro(\Xi_2)}{1+\epsilon}\right).$$
Since $\{L_{S_{n-1}} = 1\} \in \cH_{S_{n-1}}$
and $\cF_{S_{n-1}} \subset \cH_{S_{n-1}}$, by the  law of iterated expectation we have 
\begin{align} \label{aux_change_happened_in_training}
\begin{split}
 &\Exp[\Delta S_{n}  \, \mathbbm{1}_{\{L_{S_{n-1}} = 1\}} \,\vert \, \cF_{S_{n-1}}] 
 \leq \;
 \Pro(L_{S_{n-1}} = 1 \, \vert \, \cF_{S_{n-1}}) \cdot  \\
 &\; \left(
 \frac{\log(b_2/b_1)}{\Dro(\Xi_2)} \, (1+\epsilon) + \frac{V^{I}(\Xi_2)}{\Dro^2(\Xi_2)} \, 
(1+\epsilon)^2 +  M\left(\Xi_2, \frac{\epsilon \, \Dro(\Xi_2)}{1+\epsilon}\right) + \ell_2 + 1
\right).
\end{split}
\end{align}
\vspace{0.2cm}

Now, we focus on the event $\{L_{S_{n-1}} = 0\}$  and  apply  Lemma \ref{lemma:orginal_hypothetical_sec}. 
By definition \eqref{aux_Y_widehat}, on the event $\{\Theta^n \leq \tau^n\}$
we have  
$$\prod_{j=1}^{\Theta^n-1} \frac{f_{\Xi_2(j)}(\widehat{Y}^{n}_j)}{g_{\Xi_2(j)}(\widehat{Y}^{n}_j)} < d.
$$
Then,  due to \eqref{aux_couple_dynamic} and \eqref{aux_Y_widehat_couple_2}, 
$$
\Gamma^{n}_{\Theta^n-1} 
\geq \Gamma^{n}_{0} \, \prod_{j=1}^{\Theta^n-1}\frac{g_{\Xi_2(j)}({Y}^{n}_j)}{f_{\Xi_2(j)}({Y}^{n}_j)} 
= \Gamma^{n}_{0} \, \prod_{j=1}^{\Theta^n-1}\frac{g_{\Xi_2(j)}(\widehat{Y}^{n}_j)}{f_{\Xi_2(j)}(\widehat{Y}^{n}_j)} 
\geq b_1/d, 
$$ 
which implies that on the event $\{\Theta^n \leq \tau^n < \infty\}$ we have 
$$
\log(\Gamma^{n}_{\Theta^n-1} + \Pi^{n}_{\Theta^n}) \geq \log(b_1/d).
$$
Then, due to Lemma \ref{lemma:orginal_hypothetical_sec} and \ref{lemma:sigma_n_upp2}(i) (with $j = 0$) we have
\begin{align*}
\Delta S_n &\leq \tau^n + \widehat{\rho}^n \mathbbm{1}_{ \{ \Theta^n \leq \tau^n < \infty\}},  \quad \text{where} \\
 \widehat{\rho}^n &\equiv \min\left\{ \inf\{t \geq 0: {Z}^{n}_t  \geq \log(b_{2}) - \log(b_1/d) \}, M\left(\Xi_2, \frac{\epsilon \Dro(\Xi_2)}{1+\epsilon}\right) \right\}.
\end{align*}
Due to Lemma \ref{lemma:iid_independent}, $\{\widehat{Y}^{n}_t: t \geq 1\}$ in \eqref{aux_Y_widehat} is a sequence of independent random variables, that are independent of $\cH_{S_{n-1}}$. Moreover,  for each $t \in \bN$, $\widehat{Y}^{n}_t$ has  density $f_{\Xi_2(t)}$ relative to measure $\mu$. Further, recall the discussion on $\{Z^{n}_t:t\geq 0\}$ in Lemma \ref{lemma:sigma_n_upp2}(ii).
Then, by Lemma \ref{lemma:lorden_renewal} and the law of iterated expectation we have 
\begin{align*}
&\Exp[\tau^n \mathbbm{1}_{\{L_{S_{n-1}} = 0\}} \, \vert \, \cF_{S_{n-1}}]\leq 
\Pro(L_{S_{n-1}} = 0 \, \vert \,  \cF_{S_{n-1}}) \, 
\left( \frac{\log(d)}{J(\Xi_2)} + \frac{V^{J}(\Xi_2)}{J^2(\Xi_2)} + \ell_2
\right),\\
&\Exp[\widehat{\rho}^n \mathbbm{1}_{\{\Theta^n \leq \tau^n,\; L_{S_{n-1}} = 0\}}  \, \vert \,  \cF_{S_{n-1}}] \leq \;
\Pro(L_{S_{n-1}} = 0 \, \vert \,  \cF_{S_{n-1}}) \; \times \; 
\\
&
\left(
 \frac{\log(b_2/b_1) + \log(d)}{\Dro(\Xi_2)}(1+\epsilon) + 
\frac{V^{I}(\Xi_2)}{\Dro^2(\Xi_2)} 
 (1+\epsilon)^2 +  M\left(\Xi_2, \frac{\epsilon \, \Dro(\Xi_2)}{1+\epsilon}\right) + \ell_2
\right).
\end{align*}
which together with Lemma \ref{lemma:aux_change_in_training} implies that
\begin{align*} 
&\Exp[\Delta S_{n}  \mathbbm{1}_{\{L_{S_{n-1}} = 0\}}  \, \vert  \, \cF_{S_{n-1}}] \leq
\frac{\log(d)}{1+b_1} \left( \frac{1}{J(\Xi_2)}+\frac{1}{\Dro(\Xi_2)} \right) (1+\epsilon) 
+\; \frac{V^{J}(\Xi_2)}{2 J^2(\Xi_2)} 
\\
& +
\Pro(L_{S_{n-1}} = 0 \, \vert \,  \cF_{S_{n-1}}) 
\left(
 \frac{\log(b_2/b_1)}{\Dro(\Xi_2)}(1+\epsilon) + 
\frac{V^{I}(\Xi_2)}{\Dro^2(\Xi_2)} 
 (1+\epsilon)^2+ M\left(\Xi_2, \frac{\epsilon \, \Dro(\Xi_2)}{1+\epsilon}\right) + 2 \ell_2
\right).
\end{align*}
Combining this inequality with  \eqref{aux_change_happened_in_training} completes the proof.
\end{proof}

\subsection{Proof of Theorem \ref{thm:upper_bound_T}} \label{subsec:upper_bound_T}
Recall the discussion on the general strategy in the beginning of Appendix \ref{app:proofs_regarding_upper_bound}. Write $\eta$ for $\eta(b_1,d)$ for simplicity.
\begin{proof}[Proof of Theorem \ref{thm:upper_bound_T}]
Fix $\epsilon > 0$. In view of \eqref{proposed_T_decomp} 
 and  Lemma \ref{lemma:upper_each_stage},  $\Exp[\mwd{T}]$ is upper bounded by
 \begin{align*}
& \lambda(\Xi_1) + \frac{\log(b_1)+|\log(\zeta(\Xi_1))|}{\Dro_{1}(\Xi_1)} \, (1+\epsilon) \, \Exp[N] \\
+\, & \frac{\log(b_2/b_1)}{\Dro(\Xi_2)} \, 
(1+\epsilon) \, \Exp[N] + 
 \frac{\log(d)}{1+b_1} \left( \frac{1}{J(\Xi_2)}+\frac{1}{\Dro(\Xi_2)} \right) (1+\epsilon) \, \Exp[N]  \\
+\, &\mwd{\lambda}(\Xi_1) \, (\Exp[N]-1)  + \left( \frac{V^{I}(\Xi_1)}{\Dro^2(\Xi_1)}(1+\epsilon)^2 + M \left(\Xi_1, \frac{\epsilon \Dro(\Xi_1)}{1+\epsilon}\right) + 3\ell_1 \right) \, \Exp[N]  \\
+\, &\left(\frac{V^{I}(\Xi_2)}{\Dro^2(\Xi_2)} (1+\epsilon)^2+ 
 \frac{V^{J}(\Xi_2)}{2 J^2(\Xi_2)} 
 +M\left(\Xi_2, \frac{\epsilon \Dro(\Xi_2)}{1+\epsilon}\right) + 2\ell_2 \right) \, \Exp[N] .
\end{align*}
By Lemma \ref{lemma:N_upp_bd}, $\Exp[N] \leq 1 + \eta/(1-\eta)$, which completes the proof.
\end{proof}

\subsection{Discussion of the thresholds selection}\label{discussion_b1_d_sel}
Note that $b_2 = (1-\alpha)/\alpha$ is fixed. Elementary calculus shows that for any fixed $x,y > 0$, we have
\begin{equation}
\label{aux_opt}
\arg\min_{z} \left\{ \frac{x}{z} + y \log(z) \right\} =  \frac{x}{y}.
\end{equation}
Since $\eta = 1/(1+b_1) + (b_1/(1+b_1)) (1/d)$, if $b_1$ and $d$ are large, $\eta \approx 0$. Thus, if we use the approximation that $(1-\eta) \approx 1$, then 
for given $b_1$, to minimize ${\cU}(b_1,b_2,d)$, we choose 
$$
d =
b_1 \,  \frac{\widetilde{\lambda}(\Xi_{1}) + \log(b_2)/{\Dro(\Xi_{2})}}{1/\Dro(\Xi_{2}) + 1/J(\Xi_2)}.
$$

Plugging the above choice into ${\cU}(b_1,b_2,d)$, using the approximation $(1-\eta) \approx 1$,  and keeping the dominant terms related to $b_1$, we are left with
\begin{align*}
\left(\widetilde{\lambda}(\Xi_{1}) + \frac{\log(b_2)}{\Dro(\Xi_2)}\right) \, \frac{1}{b_1 + 1} +
\log(b_1 + 1) \left( \frac{1}{\Dro(\Xi_1)} -\frac{1}{\Dro(\Xi_2)} \right).
\end{align*}
If $\Dro(\Xi_2) > \Dro(\Xi_1)$, again by \eqref{aux_opt} we select $b_1$ as in \eqref{eq:b1_d_sel}. If $\Dro(\Xi_2) = \Dro(\Xi_1)$, then we choose $b_1 = b_2$.

\section{Proofs about the design for the finite memory model}
In this Appendix, we discuss how to solve the two auxiliary, deterministic dynamic programming problems introduced in Subsection \ref{subsec:example_finite_memory}. Note that in this discussion,  the finite memory change-point model in \eqref{geometric_change_point} is assumed.

\subsection{Selection of acceleration block} \label{app:proof_acc_block_fmm}
Recall that the goal is to solve
\begin{equation}\label{geometric_accleration_opt}
V^*(\bS_0) \equiv \inf_{(x_1,x_2,\ldots) \in [K]^{\infty}} 
 \sum_{t =1}^{\infty} \prod_{s = 1}^{t} \left(1 - \Psi(x_s, \bS_{s-1})\right), 
\end{equation}
where   
$\bS_t = \Phi\left(x_t, \bS_{t-1} \right)$ for $t \in \bN$ according to   \eqref{geometric_change_point}. By definition, 
\begin{align*}
V^*(\bS_0) = & \inf_{x_1 \in [K]} (1- \Psi(x_1,\bS_0)) \inf_{(x_2,x_3,\ldots)\in [K]^{\infty}} \left(1 +\sum_{t=2}^{\infty} \prod_{s=2}^{t} \left(
1 - \Psi(x_s, \bS_{s-1})
\right)
\right) \\
=& \inf_{x_1 \in [K]} (1- \Psi(x_1,\bS_0)) 
(1 + V^*( \Phi(x_1, \bS_{0})).
\end{align*}

Denote by $\cV \equiv \{V: [K]^\kappa \to [0,\infty]\}$ the space of non-negative functions on $[K]^{\kappa}$.  Define an operator $\cT: \cV \to \cV$ as follows:   for each $V \in \cV$ and $\bS \in [K]^{\kappa}$,
\begin{equation*}
\cT(V)(\bS) \equiv \min_{x \in [K]}
\left\{
(1 - \Psi\left(x, \bS) \right)(1 +  V( \Phi(x, \bS)))
\right\}.
\end{equation*}
Then clearly $V^* \in \cV$ is a fixed-point of $\cT$, i.e., $\cT(V^*) = V^*$. Further, we show below that 
\begin{align}\label{app:aux_V_star}
  V^{*} = \lim_{t \to \infty} \cT^{\bigotimes t}(0),  
\end{align}
where $0$ is the zero function  in $\cV$. Thus, we can obtain $V^*(\cdot)$ by repeated applications of $\cT$, starting from the zero function.

Finally, as shown below, once $V^*(\cdot)$ is computed, we can obtain an optimal policy $\upsilon^*:[K]^{\kappa} \to [K]$ as follows: for each $\bS \in [K]^{\kappa}$, 
\begin{align}\label{app:aux_opt_acc_policy}
\upsilon^*(\bS) \equiv \underset{x \in [K]}{\arg\min}
\left\{
(1 - \Psi\left(x, \bS) \right)(1 +  V^*(\Phi(x, \bS)))
\right\}.
\end{align}
It is  optimal  in the sense that 
the sequence $(x^*_1, x^*_2,\ldots)$ solves \eqref{geometric_accleration_opt} where
\begin{align}
    \label{app:opt_x_star_acc}
x^*_t \equiv \upsilon^*(\bS_{t-1}) , \quad \text{and} \quad \bS_t = \Phi\left(x^*_t , \bS_{t-1} \right), \quad t \in \bN.
\end{align}

Next, we justify \eqref{app:aux_V_star} and \eqref{app:aux_opt_acc_policy} using the abstract dynamic programming theory \cite{bertsekas2022abstract}. We first formulate the problem \eqref{geometric_accleration_opt} using  the language of \cite{bertsekas2022abstract}.

Denote by $\Upsilon$ the set of all policies, i.e., $\Upsilon \equiv \{\upsilon:[K]^{\kappa} \to [K]\}$. For each policy $\upsilon \in \Upsilon$, define an associated operator $\cT_{\upsilon}: \cV \to \cV$ as follows:  for each $V \in \cV$ and $\bS \in [K]^{\kappa}$,
\begin{align*}
    \cT_{\upsilon}(V)(\bS) \equiv (1 - \Psi\left(\upsilon(\bS), \bS) \right)(1 +  V( \Phi(\upsilon(\bS), \bS))).
\end{align*}
Then, given a sequence of policies $\{\nu_t \in \Upsilon: t\geq 1\}$, we have
\begin{align*}
 \sum_{t =1}^{\infty} \prod_{s = 1}^{t} \left(1 - \Psi(\nu_{s}(\bS_{s-1}); \bS_{s-1})\right) = \lim_{N \to \infty} 
 T_{\nu_1}(T_{\nu_2}(\cdots T_{\nu_N}(0))(\bS_0),
\end{align*}
where $0$ is the zero function  in $\cV$ and $\bS_{s} = \Phi(\nu_s(\bS_{s-1}), \bS_{s-1})$ for $s \geq 1$. Since the change-point model \eqref{geometric_change_point} involves no randomness, to solve \eqref{geometric_accleration_opt}, it is equivalent to solve
\begin{align*}
  V^*(\bS_0) =  \inf_{\nu_t \in \Upsilon, t\geq 1 } \lim_{N \to \infty} T_{\nu_1}(T_{\nu_2}(\cdots T_{\nu_N}(0))(\bS_0).
\end{align*}
Specifically, if $\nu_1^*,\nu_2^*, \ldots$ is the solution to the above minimization problem, then $x_t = \nu_t^*(\bS_{t-1})$ for $t\geq 1$ is the solution to \eqref{geometric_accleration_opt}. Next, we show that the solution is achieved by a stationary policy $\nu^*$ in \eqref{app:aux_opt_acc_policy}, i.e., $\nu_t = \nu^*$ for $t \geq 1$.

Define a mapping $H: [K]^{\kappa} \times [K] \times \cV \to [0,\infty]$ as follows: for each state $\bS \in [K]^{\kappa}$, action $x \in [K]$ and cost function $V \in \cV$,
$$
H(\bS,x,V) \equiv (1 - \Psi\left(x, \bS) \right)(1 +  V( \Phi(x, \bS))).
$$
Then clearly $\cT(V)(\bS) = \min_{x \in [K]} H(\bS,x,V)$ and $\cT_\upsilon(\bS) = H(\bS, \upsilon(\bS),V)$ for each $\upsilon \in \Upsilon$.

Since $\Psi(\cdot)$ takes value in $[0,1)$, $H$ clearly satisfies the ``Monotone Increase'' structure \cite[Assumption I in Chapter 4.3]{bertsekas2022abstract} with $\bar{J}$ being the zero function in $\cV$ and $\alpha = 1$. Since the action space $[K]$ is finite, by \cite[Proposition 4.3.14]{bertsekas2022abstract}, \eqref{app:aux_V_star} holds and $\upsilon^*$ is the optimal stationary policy.

\subsection{Selection of detection block} \label{app:proof_det_block_fmm}

Recall that we try to solve
\begin{align}
\label{geometric_det_opt}
\begin{split}
V'(\bS_0) \equiv \sup_{(x_1,x_2,\ldots) \in [K]^{\infty}} 
&\liminf_{M \to \infty}
\frac{1}{M}
\sum_{t =1}^{M}\left(
I_{x_t} + \omega(\Psi(x_t, \bS_{t-1}))
\right), \\
&\text{ with } \bS_{t} = \Phi(x_t, \bS_{t-1}) \text{ for } t \geq 1.
\end{split}
\end{align} 
This is an average-cost, deterministic dynamic programming problem \cite{bertsekas1995dynamic}: the state space is $[K]^{\kappa}$, the action space is $[K]$,  the transition dynamic is given by $\bS_{t} = \Phi(x_t, \bS_{t-1})$  for $t \geq 1$, and the one-stage reward function is 
$$
r(x, \bS) \equiv I_x + \omega(\Psi(x, \bS)), \;\text{ for } x \in [K],\; \bS \in [K]^{\kappa}.
$$
Since both the state space and action space are finite, the results in \cite[Chapter 4, II]{bertsekas1995dynamic} apply. 

Without loss of generality, assume $\kappa \geq 1$. Since the transition dynamic is deterministic, for any two states $\bS, \tilde{\bS} \in [K]^{\kappa}$, we show below that there exists a stationary policy $\upsilon: [K]^{\kappa} \to [K]$ and some $1 \leq M \leq \kappa$ such that 
\begin{align}
\label{aux:path_exist}
\bS_0 = \bS, \quad \bS_{t} = \Phi(\upsilon(\bS_{t-1}), \bS_{t-1}) \text{ for }  1 \leq  t \leq M, \quad \bS_{M} = \tilde{\bS}.
\end{align}
Thus, the condition (2) in  \cite[Proposition 2.6 Chapter 4, II]{bertsekas1995dynamic} holds, which implies that
$V'(\cdot)$ is a constant function with some value $\lambda'$, and  there exists a function $h':[K]^{\kappa} \to [K]$ such that
the Bellman equation holds:
$$
\lambda' + h'(\bS) = \max_{x \in [K]} \left[
r(x, \bS) + h'( \Phi(x, \bS) )
\right] \; \text{ for each }\; \bS \in [K]^{\kappa}.
$$
Then, by  \cite[Proposition 2.1 Chapter 4, II]{bertsekas1995dynamic}, the solution to \eqref{geometric_det_opt} is given by $x'_t \equiv \upsilon'(\bS'_{t-1})$ for $t \geq 1$,   where $\bS'_t = \Phi\left(x'_t, \bS'_{t-1} \right)$, $\bS'_0 = \bS_0$, and 
$$
\upsilon'(\bS) = \max_{x \in [K]} \left[
r(x, \bS) + h'( \Phi(x, \bS) )
\right] \; \text{ for each }\; \bS \in [K]^{\kappa}.
$$
Finally,   the numerical algorithms in  \cite[Chapter 4, II]{bertsekas1995dynamic}, such as linear programming, can be used to numerically compute $\upsilon'(\cdot)$.

Now, we justify \eqref{aux:path_exist}. Denote by 
$\tilde{\bS} = (j_\kappa,\ldots,j_1) \in [K]^{\kappa}$. Define
$$
\bS_0 = \bS, \quad \bS_m = \Phi(j_m; \bS_{m-1}) \text{ for } 1 \leq m \leq \kappa.
$$
By definition, $\bS_{\kappa} = \tilde{\bS}$. If $\bS_0,\ldots, \bS_{\kappa-1}$ are distinct, then we may define
$$
\nu(\bS_{m}) = j_{m+1} \text{ for } 0 \leq m < \kappa.
$$
If $\bS_0,\ldots, \bS_{\kappa-1}$ are not distinct, we can collapse the same states and obtain a shorter path $\bS_0,\bS_1,\ldots,\bS_{M}$ with $1\leq M \leq \kappa$, $\bS_0 = \bS$, $\bS_{M}=\tilde{\bS}$, and $\bS_0,\ldots, \bS_{M-1}$ being distinct.

\section{Proofs about the asymptotic optimality}

\subsection{Proof about the asymptotic upper bound}\label{app:aymp_upper_bound}

\begin{proof}[Proof of Corollary \ref{lemma:aymp_upper_bound}]
Fix any $\epsilon > 0$. If $b_1 = b_2$, then by definition, the proposed procedure only has one cycle and  the detection stage has length zero. In this case, we can apply part (i) in Lemma \ref{lemma:upper_each_stage} with $n = 1$, and the conclusion clearly holds. Next, we assume $b_1 < b_2$ and use  the non-asymptotic upper bound in Theorem \ref{thm:upper_bound_T}. 

Due to  assumption \eqref{assumptions:response_KL} and the fact that the response densities and $K$ do not depend on $\alpha$,  we have $ 0 <  \min_{x \in [K]} I_x \leq \Dro(\Xi_i)$ for any  $\alpha \in (0,1)$, $i \in \{1,2\}$. Then, due to \eqref{assumption:asymp_cond_cp1},  
the remainder term in Theorem \ref{thm:upper_bound_T} is negligible, i.e., $\mathcal{R} =o(\lambda(\Xi_1)+|\log(\alpha)|)$ as $\alpha \to 0$. Moreover,  with  thresholds selected according to \eqref{eq:bK_sel} and \eqref{eq:b1_d_sel}, condition \eqref{assumption:asymp_cond_cp1} implies that 
$\mathcal{U}(b_1,b_2,d) \leq (\lambda(\Xi_1) + |\log(\alpha)|/\Dro(\Xi_2))(1+o(1))$.

Thus, for any fixed $\epsilon > 0$, we have
$$
\limsup_{\alpha \to 0} \frac{\Exp[\widetilde{T}]}{\lambda(\Xi_1) + |\log(\alpha)|/\Dro(\Xi_2)} \leq 1+\epsilon.
$$
Since $\epsilon > 0$ is arbitrary and the left-hand side does not depend on $\epsilon$, we complete the proof by letting $\epsilon \to 0$.
\end{proof}

\begin{proof}[Proof of Lemma \ref{lemma:upper_finite_meet_asym_cond}]
Let $\Xi \in \{\Xi_1,\Xi_2\}$ and denote by $\ell \in \{\ell_1,\ell_2\}$ the length of $\Xi$. By assumption,
$$
\kappa \leq \ell = O(1).
$$

Recall the definition of $d(\Xi)$ in \eqref{assumption_stability_pointwise}. 
By assumption \eqref{general_bounded_away_from_one}, for small enough $\alpha > 0$,  we have $\pi_t(\cdot) \leq 1-c$, which implies that $d(\Xi) \leq |\log(c)|$; this verifies the first condition in \eqref{assumption:asymp_cond_cp1}. Further, the third condition in \eqref{assumption:asymp_cond_cp1} is implied by the stronger condition, $\max\{\ell_1, \ell_2, |z_0|\}= O(1)$, which is assumed  in  Lemma \ref{lemma:upper_finite_meet_asym_cond}.

Next, we consider the speed of uniform convergence $M(\Xi,\epsilon)$ defined in \eqref{assumption_stability} for any $\epsilon > 0$. 
Recall from Subsection  \ref{subsec:example_finite_memory} that for the finite memory model,
$$
d(\Xi) = 
\frac{1}{\ell} \sum_{j=1}^{\ell} \omega\left( \Psi(\Xi(j:(j-\kappa))) \right).$$
For any positive integer $M$, define $M_{\ell} \equiv \ell \times \lfloor M/\ell \rfloor$, which is the largest multiple of $\ell$ that does not exceed $M$. By partitioning the sum from $1$ to $\ell$, from $\ell+1$ to $M_{\ell}$, and from $M_{\ell}+1$ to $M$, due to assumption \eqref{general_bounded_away_from_one}, we have that 
for any $t \geq 0, z \in [K]^{t},  j \in [\ell]$, 
\begin{align*}
    &\left\vert \frac{1}{M} \sum_{s = 1}^{M} \omega \left(
\pi_{t+s} \left(z, \Xi(j:
s+j-1) \right) \right) - d(\Xi)  \right\vert \\
\leq &\frac{\ell}{M} |\log(c)| + \left| 
\frac{(M_{\ell}/\ell)-1}{M} \ell d(\Xi) - d(\Xi)
\right| +  \frac{\ell}{M} |\log(c)| \leq \frac{4\ell}{M} |\log(c)|.
\end{align*}
Then condition \eqref{assumption_stability} holds with
  $M(\Xi,\epsilon) \equiv 4\ell |\log(c)|/\epsilon$. Since $\ell= O(1)$, we have that $M(\Xi,\epsilon) = O(1)$ for any $\epsilon>0$, which verifies the second condition in \eqref{assumption:asymp_cond_cp1}.

Finally, we verify the last condition in \eqref{assumption:asymp_cond_cp1}. Recall that $\widetilde{\lambda}(\Xi_1)$ is defined in \eqref{block_upper}, and $\lambda(z_0, \Xi_1)$  in \eqref{modified_lambda}. Due to assumption \eqref{general_bounded_away_from_one} and by definition, 
\begin{align*}
    &\widetilde{\lambda}(\Xi_1) \leq  \ell_1 + \sum_{t =  \ell_1}^{\infty} \prod_{s= \ell_1}^{t} 
    \left(
    1 - \Psi(\Xi[s: (s-\kappa)])
    \right), \\
    & \lambda(z_0, \Xi_1) \geq c^{t_0 +  \ell_1} \sum_{t = \ell_1}^{\infty} \prod_{s=\ell_1}^{t} (1 - \Psi(\Xi[s:(s-\kappa)]).
\end{align*}
Thus $\widetilde{\lambda}(\Xi_1) \leq  \ell_1 +   (1/c)^{t_0 +  \ell_1}  \lambda(z_0, \Xi_1)$. Since $\max\{\ell_1, t_0\}= O(1)$, we have 
$$
\log\left(\widetilde{\lambda}(\Xi_1)\right) = o(\lambda(z_0, \Xi_1) + |\log(\alpha)|).
$$
Further, recall the definition of $\zeta(\Xi)$ in \eqref{def:zeta}.
For the finite-memory model, since $\ell_1 \geq \kappa$, 
$$
\zeta(\Xi_1)  = \max_{1 \leq s \leq \ell_1} \,\Psi(\Xi_1[s:(s-\kappa)]).
$$
Due to the definition of $\widetilde{\lambda}(\Xi_1)$ in \eqref{block_upper}, since $\ell_1 \geq \kappa$ and due to and  assumption \eqref{general_bounded_away_from_one}, 
\begin{align*}
    \tilde{\lambda}(\Xi_1) \geq \sum_{t=\ell_1}^{\infty}  c^{\ell_1}(1-\zeta(\Xi_1) )^{t-\ell_1} = c^{\ell_1}/\zeta(\Xi_1),
\end{align*}
which, together with   the upper bound on $ \tilde{\lambda}(\Xi_1)$, 
verifies the last condition in \eqref{assumption:asymp_cond_cp1}. 
\end{proof}

\subsection{Proof of Theorem \ref{thm:asymptotic_lower_bound}}\label{app:lower_bd}
Recall the definition of $\Rro^{s}_{t}$ in \eqref{LR} and $N_{\epsilon,\alpha}$ following \eqref{detection_delay_optimal}.

\begin{proof}[Proof of Theorem \ref{thm:asymptotic_lower_bound}]
Let $\Xi$ be a block that satisfies \eqref{existence of a block} and fix  $\left(\cX, T \right) \in \cC_{\alpha}$, where $\cX \equiv \{X_t, t \in \bN\}$.  Define a new procedure with  the same stopping rule but a  modified assignment rule, $\cX'  \equiv  \{X'_t, t \in \bN \}$, where 
\begin{align*}
X'_t \equiv  \begin{cases}
\; X_t \qquad \quad\;\; t \leq T \\
\; \Xi(t-T) \quad t > T.
\end{cases}
\end{align*}
Since the new assignment rule coincides with the original one up to time  $T$, we have $\Pro\left(T < \Theta_{\cX'}\right) =
\Pro\left(T < \Theta_{\cX}\right) \leq \alpha$, and consequently   $(\cX', T) \in \cC_{\alpha}$. By this observation and the definition of $\mwd{\lambda}(\Xi)$ in \eqref{block_upper} we have
\begin{align*}
\Exp[\Theta_{\cX'} - T; \;  T < \Theta_{\cX'} ] \;\leq  \; \Pro(T < \Theta_{\cX'}) \;  \mwd{\lambda}(\Xi) \,  \leq\; \alpha \mwd{\lambda}(\Xi),
\end{align*}
and by the  decomposition \eqref{ess_decompose} for $(\cX',T)$ we obtain
$$\Exp\left[T\right] \geq
 \Exp[(T - \Theta_{\cX'})^+] \,+\, \Exp[\Theta_{\cX'}] - \alpha \mwd{\lambda}(\Xi).
 $$
Since  $(\cX, T) \in \cC_{\alpha}$ is arbitrary,  recalling the definition of $\lambda^*$ in \eqref{existence of a block}  we have
\begin{equation*}
\inf_{(\cX, T) \in \cC_{\alpha}} \Exp\left[T\right] \geq
 \inf_{(\cX,T) \in \cC_{\alpha}} \Exp\left[(T - \Theta_{\cX})^+ \right] \,+\, \lambda^* - \alpha \mwd{\lambda}(\Xi),
\end{equation*}
and, due to  \eqref{existence of a block},  it suffices  to  show  that 
as $\alpha \to 0$ \;
\begin{align} \label{conclusion} 
\inf_{(\cX, T) \in \cC_{\alpha}} \;
 \Exp \left[(T-\Theta_\cX)^+\right]\;
\geq \; \frac{|\log(\alpha)|}{\Dro^*}  \, (1+o(1)).
\end{align}
Further, it suffices to show that for any $ \epsilon, \alpha  \in (0,1)$ and any  $(\cX, T) \in \cC_\alpha$  we have  
\begin{align} \label{add_lb_show}
 \Pro (\Theta_\cX \leq T \leq \Theta_{\cX} + N_{\epsilon,\alpha}) \leq \delta_\epsilon(\alpha),
\end{align}
where $N_{\epsilon,\alpha}$ is defined after  \eqref{detection_delay_optimal}  and  $\delta_\epsilon(\alpha)$  is a term that does not depend on $(\cX, T)$ and goes to 0 as $\alpha \to 0$. Indeed, if  \eqref{add_lb_show} holds,  then   by Markov's inequality we obtain 
\begin{align*}
N^{-1}_{\epsilon,\alpha} \; \;  \Exp \left[
(T-\Theta_{\cX})^+\right] 
&\geq \Pro(T \geq  \Theta_{\cX} + N_{\epsilon,\alpha}) \\
&= \Pro(T \geq \Theta_{\cX})  -\Pro (\Theta_\cX \leq T < \Theta_{\cX} + N_{\epsilon,\alpha})  \geq 1- \alpha - \delta_\epsilon(\alpha),
\end{align*}
for any  $\epsilon, \alpha \in (0,1)$ and $(\cX, T) \in \cC_{\alpha}$. Consequently,
\begin{align*} 
 \inf_{(\cX, T) \in \cC_{\alpha}} 
\Exp \left[ (T-\Theta_{\cX})^+\right] 
&\geq {N_{\epsilon,\alpha}}{\left(1- \alpha - \delta_\epsilon(\alpha)\right)},
\end{align*}
and \eqref{conclusion}  follows if we  divide both sides  by $|\log(\alpha)|/\Dro^*$, and let first  $\alpha \to 0$  and then   $\epsilon \to 0$.\\

To prove \eqref{add_lb_show} we fix $(\cX, T) \in \cC_\alpha$ and let  $\Rro^{\Theta_{\cX}}_{T}$ denote the ``likelihood ratio'' statistic at $T$ in favour of the hypothesis that the change occurred at  $\Theta_{\cX}$ against the hypothesis that the change has not happened at time $T$ (this statistic is defined formally in  Appendix \ref{app:err_control}). 
In  Appendix \ref{app:lower_bd} we  show that for any given (small enough)  $\epsilon, \alpha \in (0,1)$,
\begin{align}
\Pro \left(\Theta_{\cX} \leq T < \Theta_{\cX} + N_{\epsilon,\alpha},\;  \Rro^{\Theta_{\cX}}_{T}  < \alpha^{-(1-\epsilon^2)} \right)  \leq \delta'_{\epsilon} (\alpha),
\label{LR_large_enough} \\
\Pro\left(\Theta_{\cX} \leq T < \Theta_{\cX} + N_{\epsilon,\alpha},\;  \Rro^{\Theta_{\cX}}_{T}  \geq \alpha^{-(1-\epsilon^2)} \right) \leq \delta''_{\epsilon} (\alpha), \label{T_large_enough}
\end{align}
where both $\delta'_{\epsilon} (\alpha)$ and $\delta''_{\epsilon} (\alpha)$ do not depend on $(\cX, T)$ and go to 0  as $\alpha \to 0$. 
We verify \eqref{LR_large_enough} in Subsection \ref{app:proof_LR_large_enough} and \eqref{T_large_enough} in Subsection \ref{app:proof_T_large_enough}. 
Since \eqref{LR_large_enough} and \eqref{T_large_enough} clearly imply \eqref{add_lb_show}, the proof is complete.
\end{proof}

\subsubsection{Proof of \eqref{LR_large_enough}} \label{app:proof_LR_large_enough}

\begin{proof}[Proof of \eqref{LR_large_enough}]
Fix  $(\cX, T) \in \cC_{\alpha}$  and write $\Theta$ instead of $\Theta_\cX$ for simplicity. By definition, $\Pro(T < \Theta) \leq \alpha$.  Observe that
\begin{align*}
\Delta & \;\equiv \;\Pro\left(\Theta \leq T < \Theta + N_{\epsilon,\alpha},\;  \Rro^{\Theta}_{T} < 
\alpha^{-(1-\epsilon^2)}
 \right)  \\
&\; = \;\sum_{s = 0}^{\infty} \sum_{t = s }^{s + N_{\epsilon,\alpha}-1}
\Pro\left(T = t,\;  \Rro^{s}_{t} < \alpha^{-(1-\epsilon^2)},\; \Theta = s \right).
\end{align*}
For any $t \geq s$,  $\{T = t\}$ and $\Rro^{s}_{t}$ are both $\cF_t$ measurable.   By Lemma \ref{lemma:change_of_distribution},
\begin{align*}
\Pro\left(T = t,\;  \Rro^{s}_{t} < \alpha^{-(1-\epsilon^2)},\; \Theta = s \right) &= \Exp \left[ \Rro^{s}_{t} \; ; \; T = t, \;  \Rro^{s}_{t} < \alpha^{-(1-\epsilon^2)},\; \Theta>t   \right]\\
&\leq    \alpha^{-(1-\epsilon^2)} \; \Pro(T = t, \Theta > t).
\end{align*}
Putting these together, we obtain
  \begin{align*}
\Delta
& \leq \; \alpha^{-(1-\epsilon^2)}  \sum_{s = 0}^{\infty} \sum_{t = s }^{s + N_{\epsilon,\alpha}-1} \Pro(T = t, \Theta > t)  \\
& \leq \;  \alpha^{-(1-\epsilon^2)} \, N_{\epsilon,\alpha} \, \sum_{t = 0}^{\infty} \Pro(T = t, \Theta > t) = \alpha^{-(1-\epsilon^2)}  \, N_{\epsilon,\alpha} \, \Pro(T < \Theta) \leq  \alpha^{\epsilon^2}  \, N_{\epsilon,\alpha}.
\end{align*}
Clearly, by definition, $\Dro^* \geq \min_{x \in [K]} I_x > 0$. Thus,
$\alpha^{\epsilon^2}  \, N_{\epsilon,\alpha} = o(1)$, which completes the proof.
\end{proof}

\subsubsection{Proof of \eqref{T_large_enough}}\label{app:proof_T_large_enough}
In the remainder of this section, we focus on the proof 
of \eqref{T_large_enough}.  We start with a few observations. 
Given an assignment rule $\cX = \{\cX_t ; t \geq 1\}$, we set 
$\widehat{\Lambda}_t \equiv \log (\Lambda_t)$, where $\Lambda_t$ is defined in Lemma \ref{lemma:recursive},   i.e., 
$$
\widehat{\Lambda}_0 =  0, \quad
\widehat{\Lambda}_t =  \log\left( \frac{g_{X_{t}}(Y_{t})}{f_{X_{t}}(Y_{t})} \right), \quad t \in \bN.
$$

Note that  the treatments and the responses start from time $1$, and $X_0$ is undefined. We further define  $X_0 \equiv 0, \; I_0 \equiv 0$.

\begin{lemma}\label{lemma:sq_int}
Assume \eqref{assumptions:response_KL} holds.
Fix any assignment rule $\cX$, and we write $\Theta$ for $\Theta_\cX$ for simplicity.
For any integer $t \geq 0$, we have
\begin{equation*}
\Exp \left[
\left( \widehat{\Lambda}_{\Theta+t} - I_{X_{\Theta + t}}\right)^2 
\right] \leq V^{*}  \equiv \max_{x \in [K]}{\{V_x^{I}\}} < \infty.
\end{equation*}
\end{lemma}
\begin{proof}
Observe that the quantity of interest is equal to the following
\begin{align*}
&\sum_{s = 0}^{\infty} \sum_{x=0}^{K}
\Exp \left[
\left( \widehat{\Lambda}_{s + t} - I_{X_{s + t}}\right)^2;
\Theta = s, X_{s + t} = x
\right] \\
\leq \;&
\sum_{s = 0}^{\infty} \sum_{x=0}^{K}
\Pro(\Theta = s, X_{s + t} = x) \;
\Exp \left[
\left(\widehat{\Lambda}_{s + t}- I_x\right)^2\vert \Theta = s, X_{s + t} = x
\right]\\
\leq\; & V^{*} \; \sum_{s = 0}^{\infty} \sum_{x=0}^{K}
\Pro(\Theta = s, X_{s + t} = x)
= V^{*},
\end{align*}
where we used the fact that $L_{s+t} = 1$ on the event $\{\Theta = s\}$.
\end{proof}


%
Recall the definition of $\{\cH_{t}: t\geq 0\}$ in \eqref{def:hidden_filtration}.
\begin{lemma}\label{lemma:M_mart}
Assume \eqref{assumptions:response_KL} holds.
Fix any assignment rule $\cX$, and we write $\Theta$ for $\Theta_\cX$ for simplicity. Then the process 
\begin{align*} 
\left\{ 
M_{\Theta+t} \equiv
\sum_{j = \Theta}^{\Theta + t}\left( \widehat{\Lambda}_j - I_{X_j}\right)\; : \; t \geq 0  \right\}
\end{align*}
is a square integrable martingale with respect to $\{\cH_{\Theta + t}: t\geq 0\}$. 
\end{lemma}
\begin{proof}
Adaptivity is obvious and square integrability is established in  Lemma \ref{lemma:sq_int}.
For any $t \geq 1$, in view of \eqref{sys_equ} and since $L_{\Theta + t} =1$, we have
$$
Y_{\Theta + t} = h(X_{\Theta + t}, 1, V_{\Theta + t}).
$$
Since $\Theta + t -1$ is an $\{\cH_t\}$-stopping time,
by Lemma \ref{lemma:iid_independent} it follows that $V_{\Theta + t}$ is independent of $\cH_{\Theta + t-1}$ and has distribution $\text{Unif}(0,1)$.
Since $X_{\Theta + t} \in \cH_{\Theta + t - 1}$, we have
\begin{align*}
\Exp\left[  \widehat{\Lambda}_{\Theta + t} - I_{X_{\Theta + t}} \vert \cH_{\Theta +t-1}
\right]= \int_{\bY} \log\left(\frac{g_{X_{\Theta + t}}}{f_{X_{\Theta + t}}}\right)
g_{X_{\Theta + t}} d\,\mu - I_{X_{\Theta + t}} = 0,
\end{align*}
which completes the proof.
\end{proof}

Next we study the behavior of above martingale. 

\begin{lemma}\label{doob_max}
Fix any assignment rule $\cX$, and we write $\Theta$ for $\Theta_\cX$ for simplicity. 
Consider the process $\{M_{\Theta + t}: t\geq 0\}$ defined in Lemma \ref{lemma:M_mart}. Then, for any  $\epsilon > 0$ we have
$$
\Pro\left(\max_{0 \leq t < m} M_{\Theta + t} \geq \epsilon \, m \right)  \leq 
\frac{V^{*}}{\epsilon^2 m},
$$
where $V^{*} < \infty$ is defined in Lemma \ref{lemma:sq_int}.
\end{lemma}

\begin{proof}
Since $\{ M_{\Theta + t}: t \geq 0\}$ is a square integrable $\{\cH_{\Theta + t}\}$-martingale, by Doob's submartingale inequality we have
\begin{align*}
\Pro\left(  \max_{0 \leq t < m} M_{\Theta + t} \geq \epsilon\, m \right) 
 \leq \frac{1}{\epsilon^2 m^2} \Exp\left[ (M_{\Theta + m-1})^2 \right].
\end{align*}
By the  orthogonality of the increments of the square-integrable martingale and  Lemma \ref{lemma:sq_int}, 
$$
\Exp\left[ (M_{\Theta + m-1})^2 \right] 
= \sum_{s = 0}^{m-1}
\Exp \left[
\left(\widehat{\Lambda}_{\Theta +s} - I_{X_{\Theta +s}}\right)^2 
\right]
\leq m\, V^{*},
$$
which completes the proof.
\end{proof}

We can finally complete the proof of \eqref{T_large_enough}.

\begin{proof}[Proof of \eqref{T_large_enough}]
Fix $(\cX, T) \in \cC_{\alpha}$ and write $\Theta$ for $\Theta_\cX$. 
Observe that 
\begin{align*}
\Pro(\Theta \leq T < \Theta + N_{\epsilon,\alpha},\;  \Rro^{\Theta}_{T} \geq \alpha^{-(1-\epsilon^2)})  
&\leq  \Pro \left(\max_{ 0 \leq t <  N_{\epsilon,\alpha}} \log\Rro^{\Theta}_{\Theta+t} \geq (1-\epsilon^2) \,  |\log (\alpha)| \right) \\
&\leq  \Pro \left(\frac{1}{N_{\epsilon,\alpha}} \max_{0 \leq  t <  N_{\epsilon,\alpha}} \log\Rro^{\Theta}_{\Theta+t} \geq (1+\epsilon) \Dro^* \right).
\end{align*}
Next, by the definition of $\log\Rro^{\Theta}_{\Theta+t}$ in \eqref{LR} it follows that 
\begin{align*}
\log \Rro^{\Theta}_{\Theta + t} &= -|\log \Pi_{{\Theta}}| +
\sum_{j = \Theta}^{\Theta + t} \left( \widehat{\Lambda}_j - I_{X_j}\right) + \sum_{j = \Theta}^{\Theta + t} \left(|\log(1-\Pi_{j})| + I_{X_j}\right) \\
& \leq  M_{\Theta+t}+ \sum_{j = \Theta}^{\Theta + t} \left(|\log(1-\Pi_{j})| + I_{X_j}\right),
\end{align*}
where recall that $\{ M_{\Theta+t}\}$ is defined in Lemma \ref{lemma:M_mart}. 
Due to \eqref{detection_delay_optimal}, for small enough $\alpha$ we have
\begin{align*}
N_{\epsilon,\alpha}^{-1} \max_{0 \leq  t <  N_{\epsilon,\alpha}} \sum_{j = \Theta}^{\Theta + t} \left(|\log(1-\Pi_{j})| + I_{X_j}\right) \leq \Dro^* (1 + 3\epsilon/4).
\end{align*}
As a result, by Lemma \ref{doob_max} there exists a constant $C$ such that  for small enough $\alpha$
\begin{align*}
&\Pro \left(N_{\epsilon,\alpha}^{-1} \max_{0 \leq t <  N_{\epsilon,\alpha}}  \log\Rro^{\Theta}_{\Theta+t} \geq (1+\epsilon) \Dro^*  \right) \\ 
&\qquad \leq\; \Pro\left(N_{\epsilon,\alpha}^{-1} \max_{0 \leq t <  N_{\epsilon,\alpha}} 
M_{\Theta + t} 
 \geq \frac{\epsilon \Dro^* }{4} \right) \leq  \frac{C}{N_{\epsilon,\alpha}},
\end{align*}
which completes the proof.
\end{proof}

\subsection{Proof of Lemma \ref{lemma:finite_lb_condition}} \label{proof:finite_lb_condition}
\begin{proof}
From the  discussion in Subsection \ref{subsec:example_finite_memory} 
    and Appendix \ref{app:proof_acc_block_fmm}, by definition,
 $\lambda^* = \lambda(z_0^*, \Xi_1^*)$   for any $\alpha > 0$.
Further, in the proof of Lemma \ref{lemma:upper_finite_meet_asym_cond}, we showed that
 $$
\widetilde{\lambda}(\Xi_1^*) \leq    \ell_1^* +   (1/c)^{t_0^*   +\ell_1^*}  \lambda(z_0^*, \Xi_1^*),
 $$
 where $t_0^*$ and $\ell_1^*$ are the lengths of $z_0^*$ and $\Xi_1^*$ respectively, and both are upper bounded by $K^{\kappa}$. Since $\kappa$ is fixed, we have that condition \eqref{existence of a block} holds with $\Xi = \Xi_1^*$ and 
$\lambda^* = \lambda(z_0^*, \Xi_1^*)$.

Now we focus on condition \eqref{detection_delay_optimal}. Fix some $\epsilon>0$. By the definition of the finite-memory model \eqref{geometric_change_point}, 
\begin{align*}
&\sup_{t \geq 0, \cX} \; \; 
\frac{1}{  N_{\epsilon,\alpha}}\sum_{s=1}^{N_{\epsilon,\alpha}}\left(
I_{X_{t+s}} + \omega(\Pi_{t+s}(\cX))
\right) \\
= &\max_{\bS_0 \in [K]^{\kappa}, \cX} \frac{1}{  N_{\epsilon,\alpha}}\sum_{s=1}^{N_{\epsilon,\alpha}}\left(
I_{X_{s}} + \omega(\Psi(X_s, X_{s-1},\ldots,X_{s-\kappa}))
\right),
\end{align*}
where $\bS_0 = (x_0,\ldots,x_{1-\kappa})$ is the initial $\kappa$ treatments.

Fix some  $\alpha >0$ and $\bS_0 \in [K]^{\kappa}$. Define the optimal sequence of treatments for the above finite-horizon, deterministic optimization problem as follows:
$$
(x_1^*,\ldots,x_{N_{\epsilon,\alpha}}^*) \equiv  \arg\max_{ \cX} \frac{1}{  N_{\epsilon,\alpha}}\sum_{s=1}^{N_{\epsilon,\alpha}}\left(
I_{x_{s}} + \omega(\Psi(x_s, x_{s-1},\ldots,x_{s-\kappa}))
\right).
$$
Further, define an infinite sequence of treatments by repeating the above finite segment, intervened by the initial $\kappa$ treatments $\bS_0$, that is, 
$$
(\tilde{x}_t: t \geq 1) \equiv (
x_1^*,\ldots,x_{N_{\epsilon,\alpha}}^*, \bS_0, x_1^*,\ldots,x_{N_{\epsilon,\alpha}}^*, \bS_0, \ldots).
$$

Recall the  discussion in Subsection \ref{subsec:example_finite_memory} 
    and Appendix \ref{app:proof_det_block_fmm}. We have that for any $\bS_0 \in [K]^{\kappa}$,
\begin{align*}
    \Dro(\Xi_2^*) =&
 \sup_{(x_1,x_2,\ldots) \in [K]^{\infty}} 
\liminf_{M \to \infty}
\frac{1}{M}
\sum_{t =1}^{M}\left(
I_{x_t} + \omega(\Psi(x_t, x_{t-1},\ldots,x_{t-\kappa}))
\right) \\ 
\geq & 
\lim_{M \to \infty}
\frac{1}{M}
\sum_{t =1}^{M}\left(
I_{\tilde{x}_t} + \omega(\Psi(\tilde{x}_t, \tilde{x}_{t-1},\ldots, \tilde{x}_{t-\kappa}) \right).
\end{align*}
Then by the definition of $(\tilde{x}_t: t \geq 1)$,
\begin{align*}
      \Dro(\Xi_2^*) \geq  \frac{1}{  N_{\epsilon,\alpha} +\kappa}\sum_{s=1}^{N_{\epsilon,\alpha}}\left(
I_{x_{s}^*} + \omega(\Psi(x_s^*, x_{s-1}^*,\ldots,x_{s-\kappa}^*))
\right),
\end{align*}
where $\bS_0 = (x_0^*,\ldots,x_{1-\kappa}^*)$. Further, by the definition of $(x_1^*,\ldots,x_{N_{\epsilon,\alpha}}^*)$,  for any $\bS_0$, 
\begin{align*}
    \Dro(\Xi_2^*) \geq \frac{N_{\epsilon,\alpha} }{  N_{\epsilon,\alpha} + \kappa} \arg\max_{ \cX}\frac{1}{  N_{\epsilon,\alpha}} \sum_{s=1}^{N_{\epsilon,\alpha}}\left(
I_{x_{s}} + \omega(\Psi(x_s, x_{s-1},\ldots,x_{s-\kappa}))
\right),
\end{align*}
which implies that
\begin{align*}
    \sup_{t \geq 0, \cX} \; \; 
\frac{1}{  N_{\epsilon,\alpha} \Dro(\Xi_2^*) }\sum_{s=1}^{N_{\epsilon,\alpha}}\left(
I_{X_{t+s}} + \omega(\Pi_{t+s}(\cX))
\right) \leq 1 + \frac{\kappa}{N_{\epsilon,\alpha}}.
\end{align*}
Since $\kappa$ is fixed and $\Dro(\Xi_2^*) = O(1)$ due to \eqref{general_bounded_away_from_one}, we have that 
condition \eqref{detection_delay_optimal}  holds with $\Dro^* = \Dro(\Xi_2^*)$.  
\end{proof}

\section{Example - exponential decay change-point model} \label{app:ex_exp}
In this section, we study the following example of  Markovian change-point models that is \textit{not} of  finite memory: 
\begin{equation}\label{exponential_decay_CP}
\Pi_t = \widetilde{\Psi}(\xi_t), \quad \xi_t = \sum_{i=1}^{\kappa} r_{i} \,  \xi_{t-i} + q_{X_t},  \quad t \in \bN,
\end{equation}
where $\widetilde{\Psi}: \bR \to [0,1)$,   $r_{i} \in [0,1]$ and $q_x \geq 0$ for $i \in [\kappa]$ and  $x \in [K]$.
This is indeed a Markovian model, in the sense of \eqref{Markovian_change_point},  with  sufficient statistic  $\bS_{t-1} = (\xi_{t-1}, \ldots, \xi_{t-\kappa})$ and the link function
$$
\Phi(X_t, \bS_{t-1}) = \sum_{i=1}^{\kappa} r_{i} \,  \xi_{t-i} + q_{X_t}.
$$

However,  unlike the  finite memory model, any assigned treatment has an impact on  future transition probabilities. When $r_i < 1$ for every $i\in [\kappa]$,  this impact decays exponentially over time, and for this reason we  refer to  $r_{1},\ldots, r_{\kappa}$ as \textit{decay} parameters, to $q_x$  as the \textit{current effect} of treatment $x \in [K]$, and to \eqref{exponential_decay_CP} as \textit{exponential decay} change-point model. We also note that when   $\kappa = 1$,    \eqref{exponential_decay_CP} reduces to 
$$
\Pi_t = \widetilde{\Psi}\left( \sum_{s=1}^{t} q_{X_s} r_1^{t-s} \right), \quad t \in \bN,$$
and setting  $r_1=1$ in the last expression  we recover   the   transition model  in  \cite[equation (7)]{wang2018tracking}. 

\subsection{Design of the proposed procedure and asymptotic optimality}
\label{subsec:exp_cp_design}
As in Subsection \ref{subsec:example_finite_memory} for  the finite memory change-point model, here for the exponential decay model in \eqref{exponential_decay_CP}, we compute
 the quantity $d(\Xi)$ defined in \eqref{assumption_stability_pointwise},    show that  condition \eqref{assumption_stability} holds, and discuss how to obtain the blocks  $\Xi_1$ and $\Xi_2$ that solve  the optimization problems in \eqref{prop_opt_blocks_sel}. We further establish the asymptotic optimality of the proposed procedure. Thus, we  set  
\begin{align*}
\bG \equiv \begin{bmatrix}
r_{1} & r_{2}  & \ldots & r_{\kappa-1} & r_{\kappa} \\ 
1 & 0 & \ldots & 0 & 0 \\ 
0 & 1 & \ldots & 0 & 0\\ 
\vdots &  & \ddots&  & \vdots \\ 
0 & 0 & \ldots & 1 & 0
\end{bmatrix}_{\kappa \times \kappa}, \quad \bQ(x) \equiv 
\begin{bmatrix}
q_{x} \\ 
0 \\ 
0\\ 
\vdots \\ 
0
\end{bmatrix}_{\kappa \times 1} \;\; \text{ for } x \in [K],
\end{align*}
and we denote    by   $\rho(\mathbb{G})$ the spectrum radius of $\mathbb{G}$, and by $\mathbb{I}$ the $\kappa \times \kappa$ identity matrix. 
Moreover, for any block $\Xi \in [K]^\ell$   we  define the following $\kappa \times 1$ vectors:
\begin{align*}
\chi_j \equiv \left(\mathbb{I} - \bG \right)^{-1} \sum_{i=j-\ell+1}^{j} 
\bG^{j - i}\bQ(\Xi[i]) , \;\;\text{ for } 1 \leq j \leq \ell. 
\end{align*}
 
 The following lemma, whose proof can be found in Appendix \ref{app:exp_stability},  verifies  condition \eqref{assumption_stability} and provides the value of 
$d(\Xi)$.

\begin{lemma}\label{lemma:exp_stability}
If the link function $\mwd{\Psi}$ is continuous and 
$\rho(\bG) < 1$,  condition  \eqref{assumption_stability} holds   for any $\Xi \in [K]^{\ell}$ with 
$d(\Xi) = \ell^{-1} \sum_{j=1}^{\ell} \omega(\mwd{\Psi}(\chi_j[1]))$.
\end{lemma}

\noindent \textit{Selection of acceleration block.}  If $\mwd{\Psi} \text{ is non-decreasing}$, then
the solution of the  first optimization problem  in \eqref{prop_opt_blocks_sel}  is 
 $\Xi_1 = (x^*)$, where $x^* \equiv \text{argmax}_{x \in [K]} q_x$. That is, during an acceleration stage we only apply a treatment  with the largest current effect parameter.\\

\noindent \textit{Selection of detection block.} 
For the second optimization problem in  \eqref{prop_opt_blocks_sel}, we provide sufficient conditions for a block of length \textit{one} to be its solution.  We start with a  general result that holds for any change-point model and whose proof can be found in Appendix \ref{app:suff_det}.
\begin{lemma}\label{lemma:suff_det}
If  there exists $x' \in [K]$ such that 
for any $t, M \in \bN$ and  $z \in [K]^t$ we have 
\begin{align*}
I_{x} +  \sum_{s = 1}^{M} 
\omega\left( \pi_{t+s}(z, x,x', \ldots,x') \right)
\leq 
I_{x'} +  \sum_{s = 1}^{M} 
\omega\left( \pi_{t+s}(z, x',x', \ldots,x') \right),
\end{align*}
then  for any $t, M \in \bN$,  $z \in [K]^t$ and $(x_{t+1},x_{t+2},\ldots x_{t+M}) \in [K]^{M}$ we have 
\begin{align*}
\sum_{s =1}^{M}\left(
I_{x_{t+s}} + \omega(\pi_{t+s}(z,x_{t+1},\ldots,x_{t+s}))
\right)\;
\leq \; 
\sum_{s =1}^{M}\left(
I_{x'} + \omega(\pi_{t+s}(z,x',\ldots,x'))
\right),
\end{align*}
 the sequence $\{x',x',\ldots \}$ solves 
for any $t \in \bN$ and $z\in [K]^{t}$ the   optimization problem 
$$
\sup_{(x_{t+1},x_{t+2},\ldots) \in [K]^{\infty}} 
\liminf_{M \to \infty}
\frac{1}{M}
\sum_{s =1}^{M}\left(
I_{x_{t+s}} + \omega(\pi_{t+s}(z, x_{t+1},\ldots,x_{t+s}))
\right),
$$ 
  and $\Xi_2 = (x')$ is the solution to the second optimization problem  in \eqref{prop_opt_blocks_sel}.
\end{lemma}


We illustrate this result for the exponential change-point model when $\kappa = 1$, $0 \leq r_1 < 1$, and $\mwd{\Psi}$ is continuously differentiable. For $t \in \bN$, $z \in [K]^t$, $x,x' \in [K]$, consider two sequences $\{\xi_t'\}$ and $\{\xi_t\}$: if $s \leq t$, $\xi'_s = \xi_s = r_1 \xi_{s-1} + q_{z(s)}$; 
\begin{align*}
&\xi_{t+1} = r_1\xi_{t} + q_{x}, \quad
\xi_{t+s} = r_1\xi_{t+s-1} + q_{x'} \;\;\text{ for } s \geq 2,\\
&\xi'_{t+s} = r_1\xi'_{t+s-1} + q_{x'} \;\;\text{ for } s \geq 1.
\end{align*}
Then, $\xi'_{t+s} - \xi_{t+s} = (q_{x'} - q_{x}) \, r_1^{s}$. Thus,
using the  Mean Value Theorem and the fact that $\omega'(x) = 1/(1-x)$, $x \in (0,1)$, we can see that  the assumption in Lemma \ref{lemma:suff_det} holds if there exists $x' \in [K]$ such that 
\begin{align}\label{exp_suf_cond}
I_{x'} - I_{x} \geq  
\frac{|q_{x'} - q_{x}|}{1-r_1} \; \sup_{\xi \in \bR} \frac{\mwd{\Psi}'(\xi)}{1 - \mwd{\Psi}(\xi)}
\end{align}
holds for every  $x \in [K]$,  where $\mwd{\Psi}'$ is the derivative of $\mwd{\Psi}$.\\

\noindent \textit{Asymptotic optimality.}  The next Corollary establishes the asymptotic optimality of the proposed procedure for the exponential change-point model in \eqref{exponential_decay_CP}  under the previous design. Recall Corollary \ref{cor:main} in the main text, and the definition of $\lambda^*$ in \eqref{existence of a block} 
 and $\Dro^*$ in \eqref{detection_delay_optimal}.

\begin{corollary}\label{cor:exp_cp_asym_opt}
Consider the exponential change-point model in \eqref{exponential_decay_CP}, and assume that the link function $\mwd{\Psi}$ is non-decreasing. If  $\Xi_1 = (x^*)$ where $x^* \equiv \max_{x \in [K]} q_x$,  then  $\lambda(\Xi_1) = \lambda^*$. Further,  if  the condition in Lemma \ref{lemma:suff_det} holds and $\Xi_2 = (x')$, then \eqref{detection_delay_optimal}  holds with $\Dro^* = \Dro(\Xi_2)$. As a result, if the response model satisfies condition 
\eqref{assumptions:response_KL},  the change-point model satisfies condition   \eqref{general_bounded_away_from_one}, and condition  \eqref{assumption:asymp_cond_cp1} holds, then the proposed procedure is asymptotically optimal, given that the thresholds are selected according to \eqref{eq:bK_sel} and \eqref{eq:b1_d_sel}.
\end{corollary}

\subsection{Simulation study for the  exponential decay model}
\label{sec:exp_sims} 

In this study, we  consider the binary response model as in Section \ref{sec:simulation} with $K =3$ treatments and 
$f_1 = 0.4, f_2 = 0.35, f_3 = 0.3$.

Further, we work under an exponential decay change-point model of the form \eqref{exponential_decay_CP} with $\kappa=1$:
\begin{align*}
\xi_t = r \xi_{t-1} + q_x,\;\; \xi_0 = 0;\quad
\Pi_t = 1 -0.15 \, \mathcal{N}(\xi_t - 2), \quad t \in \bN,
\end{align*}
 where $\mathcal{N}$ is the cumulative distribution function for the standard normal distribution. We vary the value of  $r \in [0,1]$ and for each $r$ we set $q_3 = 0$ and select $q_1$ and $q_2$ such that the expected time until the change happens if only treatment $1$ (resp. $2$) is assigned is equal to 15 (resp. 30), i.e.,   $\lambda([1]) = 15$ and $\lambda([2]) = 30$, where $[i]$ denotes the block of length $1$  that consists of treatment $i \in [3]$. The specific values for $q_1, q_2$  are listed in Appendix \ref{app:exp_values}, together with the corresponding values for  $\Dro([i])$,  $i \in [3]$. (Recall the definition of $\Dro(\Xi)$, $\lambda(\Xi)$ in  \eqref{adjust_info_num},\eqref{block_expected_change}).

The optimal procedure is implemented as described in Section \ref{sec:dp}  with state space $[0,1] \times \bR$.   Based on the previous discussion,   we select the acceleration block for the proposed procedure as  $\Xi_1 = [1]$, i.e., we assign only treatment $1$ during an acceleration stage.  Furthermore, we verify numerically that the sufficient condition \eqref{exp_suf_cond} holds with $x' = 3$ if $r \leq 0.92$. Therefore, if we assign only treatment $3$ in detection stages,  i.e if  $\Xi_2 = [3]$, Corollary \ref{cor:main} implies that the proposed procedure is asymptotically optimal at least when $r \leq 0.92$.  The results are presented in Figure \ref{fig:exp_decay}, where we observe that the performance of the proposed method is close to the optimal even when $r > 0.92$.


\begin{figure}[tbp!]
\subfloat{\label{fig:r26}
\includegraphics[width=0.4\linewidth]{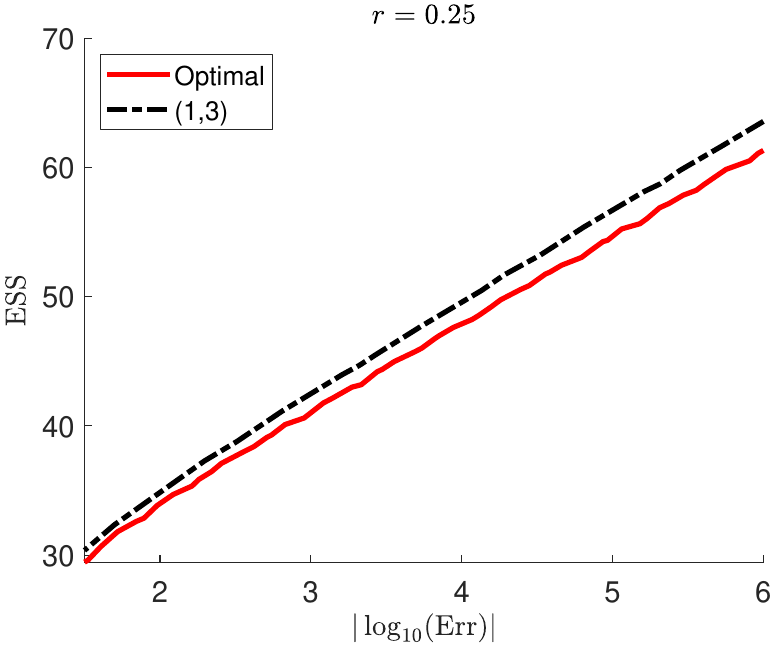}  
}
\hspace{0.1cm}
\subfloat{\label{fig:r91}
\includegraphics[width=0.4\linewidth]{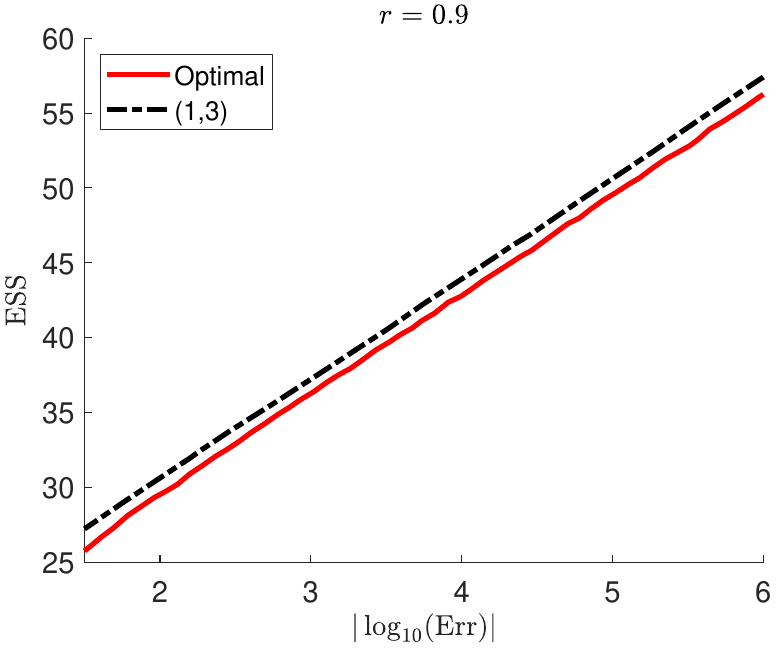}  
}\\
\subfloat{\label{fig:r96}
\includegraphics[width=0.4\linewidth]{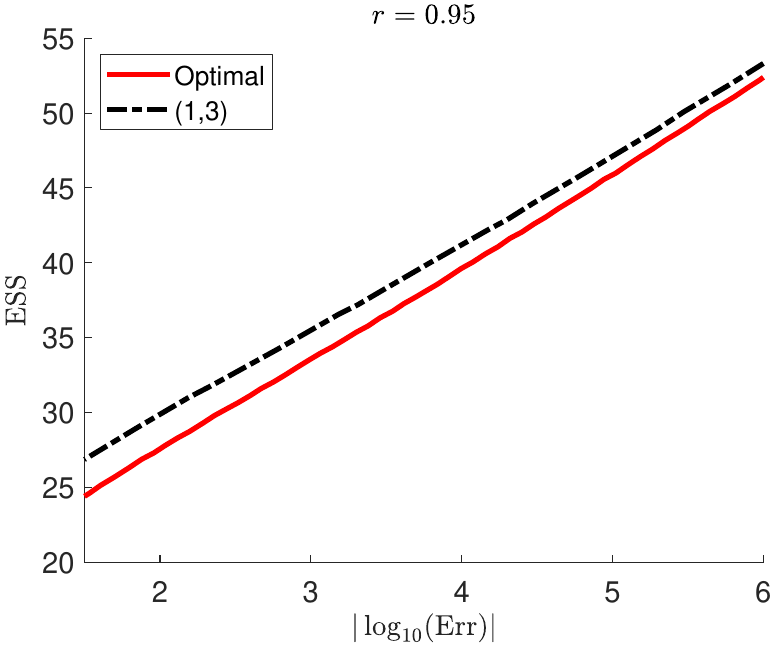}  
}
\hspace{0.1cm}
\subfloat{\label{fig:r101}
\includegraphics[width=0.4\linewidth]{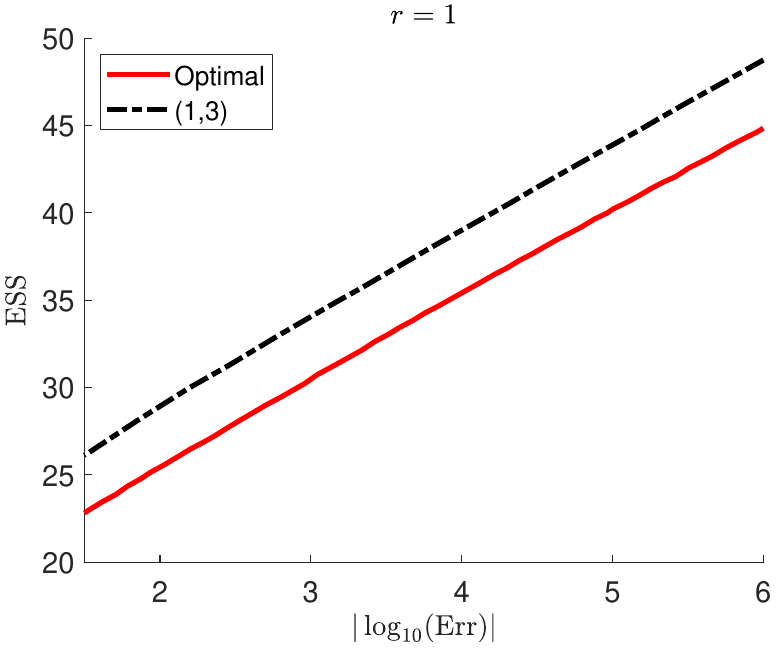}  
}

\caption{ESS versus $|\log_{10}(\text{Err})|$ for different  values of $r$.
}
\label{fig:exp_decay}
\end{figure}

%


\subsection{Proof of Lemma \ref{lemma:exp_stability}} \label{app:exp_stability}

We start with two lemmas. For $t \geq 1$, denote $\bar{{\xi}}_t = [\xi_t, \ldots, \xi_{t-\kappa+1}]^{T}$. Then by definition \eqref{exponential_decay_CP},
$$
\bar{{\xi}}_t = \bG \bar{{\xi}}_{t-1} + \bQ(X_{t}),  t \in \bN,
$$
where $\bar{{\xi}}_0 = (\xi_0, \ldots, \xi_{1-\kappa})$ is the initial state. Denote $\|\cdot\|$ the operator norm of a matrix.

\begin{lemma}\label{lemma:bound_on_xi}
Fix the initial state $\bar{{\xi}}_0$ and assume $\rho(\bG) < 1$. Then there exists a finite positive constant $B$ such that
for any sequence $\{x_t \in [K]: t \in \bN\}$ and $X_t = x_t$,   we have
$\sup_{t \geq 1} \|\bar{{\xi}}_t\| \leq B$.
\end{lemma}
\begin{proof}
Since $\rho(\bG) < 1$, we can pick $\epsilon > 0$ such that $\rho(\bG) < 1 - \epsilon$. By Gelfand's formula, there exists $M>0$ such that  $\|\bG^{t}\| \leq (\rho(\bG) + \epsilon)^t$
for every $t \geq M$. Iterating:
$$
\bar{{\xi}}_t = \sum_{s=1}^{t} \bG^{t-s}\bQ(X_{t}) + \bG^{t} \bar{{\xi}}_0.
$$
Set $B_0 \equiv \max_{x \in [K]} q_x$. Then, for $t \in \bN$,
$$
\|\bar{{\xi}}_t\| \leq B_0 \sum_{s=1}^{M} \|\bG\|^{s} + 
B_0\sum_{s=M}^{\infty} (\rho(\bG) + \epsilon)^{s}  + \max\{1, \|\bG\|^{M}\}\; 
\|\bar{{\xi}}_0\| < \infty,
$$ 
Since the right hand side does not depend on $t$, the proof is complete.
%
\end{proof}

We next establish the stability of the sequence $\{\xi_t: t \in \bN\}$ if some block $\Xi$ is always assigned from the beginning.
 
\begin{lemma}\label{lemma:stability_of_xi_expCP}
Consider the change-point model \eqref{exponential_decay_CP},
assume $\rho(\bG) < 1$ and fix some block $\Xi \in \cB_{\ell}$. Let 
$X_t = \Xi(t)$ and fix $B > 0$. For any $\epsilon > 0$, there exists $s_0$ such that 
for any $\bar{\xi}_0$ with $\|\bar{\xi}_0\| \leq B$,  we have
\begin{align*}
\left\| \bar{\xi}_{s \ell + j} -   \chi_j\right\| < \epsilon \quad   \text{ for } s \geq s_0 \quad  1 \leq j \leq \ell.
\end{align*}
\end{lemma}
\begin{proof}
By telescoping argument, we have
\begin{align*}
\bar{\xi}_{s \ell + j} = &\sum_{i = 1}^{j}   \left( \sum_{\iota = 0}^{s} \bG^{\iota\ell}\right) \bG^{j-i}\bQ(\Xi[i])
+ \sum_{i = j+1-\ell}^{0} \left( \sum_{\iota = 0}^{s-1} \bG^{\iota\ell}\right)  \bG^{j-i}\bQ(\Xi[i])  \\
+ & \bG^{s \ell + j} \bar{\xi}_0.
\end{align*}
Since $\rho(\bG) < 1$, we have
$$
\lim_{t \to \infty} \mathbb{G}^t= 0 \quad \text{ and } \quad
\lim_{t \to \infty} \sum_{\iota=0}^{t} \mathbb{G}^\iota = (\mathbb{I} - \mathbb{G})^{-1},
$$
which completes the proof.
\end{proof}

\begin{proof}[Proof of Lemma \ref{lemma:exp_stability}]
By Lemma \ref{lemma:bound_on_xi}, for fixed initial state $\bar{\xi}_0$, $\|\xi_t\| \leq B$ for any $t \geq 1$. Then the proof is complete due to Lemma \ref{lemma:stability_of_xi_expCP} and the assumption that $\widetilde{\Psi}$ is continuous.
\end{proof}

\subsection{Proof of Lemma \ref{lemma:suff_det}}\label{app:suff_det}
\begin{proof}
We proceed by induction. For $M = 1$, this is just the assumption in Lemma \ref{lemma:suff_det}. We assume the claim holds for some arbitrary $M \in \bN$, and we show that it also holds for $M+1$. For any vector $(x_{t+1},\ldots, x_{t+M+1}) \in [K]^{M+1}$ we have
\begin{align*}
&\sum_{s =1}^{M + 1}\left(
I_{x_{t+s}} + \omega(\pi_{t+s}(z, x_{t+1},\ldots,x_{t+s}))
\right) \\
= \;\;& 
I_{x_{t+1}} + \omega(\pi_{t+1}(z, x_{t+1}))+
\sum_{s =2}^{M+1 }\left(
I_{x_{t+s}} + \omega(\pi_{t+s}(z,x_{t+1},\ldots,x_{t+s}))
\right) \\
\leq &
I_{x_{t+1}} + \omega(\pi_{t+1}(z, x_{t+1}))+
\sum_{s =2}^{M+1 }\left(
I_{x'} + \omega(\pi_{t+s}(z,x_{t+1},x',\ldots,x'))
\right))\\
\leq &\sum_{s =1}^{M + 1}\left(
I_{x'} + \omega(\pi_{t+s}(x',\ldots,x'))
\right),
\end{align*}
where the first inequality is due to the induction hypothesis
 with $t = t+1$ and $z = (z,x_{t+1})$, and  the second is  due to the assumption in Lemma \ref{lemma:suff_det}. Thus, the proof is complete.
\end{proof}

\section{Example - polynomial decay change-point model} \label{app:ex_poly}
In this section, we consider
the following example of \textit{non-Markovian}  change-point models:
 \begin{equation}\label{polynomial_decay_change_point} 
\Pi_t = \widetilde{\Psi}\left(\sum_{s=1}^{t} {q_{X_s}}{(t-s+1)^{-r}}\right), \quad t \in \bN,
\end{equation}
where  $\widetilde{\Psi} : [0,\infty] \, \to\, [0,1)$,  $r > 0$, and $q_{x} \geq 0$, $x \in [K]$. We refer to  \eqref{polynomial_decay_change_point}  as \textit{polynomial decay} change-point model.  Note that this is \textit{not} Markovian in the sense of \eqref{Markovian_change_point}, since there is no sufficient statistics of a fixed dimension. 

\subsection{Design and asymptotic optimality}
For the polynomial decay model in \eqref{polynomial_decay_change_point}, 
we follow the outline as in Subsection \ref{subsec:example_finite_memory} and  Appendix \ref{subsec:exp_cp_design}: compute
 the quantity $d(\Xi)$ defined in \eqref{assumption_stability_pointwise},    verify condition \eqref{assumption_stability}, discuss how to obtain the blocks  $\Xi_1$ and $\Xi_2$ for the optimization problems in \eqref{prop_opt_blocks_sel}, and establish the asymptotic optimality.

For a block  $\Xi \in [K]^{\ell}$ we   set 
$$
\chi_j \equiv \sum_{i=j-\ell + 1}^{j} q_{\Xi(i)} \left(
\sum_{\iota = 0}^{\infty} (j + 1 - i + \iota \ell)^{-r}
\right), \quad 1 \leq j \leq \ell.
$$
 The following lemma, whose proof can be found in Appendix \ref{app:poly_stability},  verifies  condition \eqref{assumption_stability} and provides the value of 
$d(\Xi)$ for each block. 

\begin{lemma}\label{lemma:poly_stability}
If  the link function $\mwd{\Psi}$ in \eqref{polynomial_decay_change_point} is continuous and $r > 1$,  condition \eqref{assumption_stability} holds 
for any $\Xi \in \cB_{\ell}$  with
$d(\Xi) = {\ell}^{-1} \sum_{j=1}^{\ell} \omega(\mwd{\Psi}(\chi_j))$.
\end{lemma}

\noindent \textit{Selection of acceleration block.} If $\mwd{\Psi} \text{ is non-decreasing}$, then $\Xi_1 = (x^*)$ is the solution to  the first problem in \eqref{prop_opt_blocks_sel}
where
$x^* \equiv \max_{x \in [K]} q_x$.\\

\noindent \textit{Selection of detection block.}  By  Lemma \ref{lemma:suff_det} again it follows that there is some $x' \in [K]$ so that    $\Xi_2 = (x')$ is the solution to the second optimization problem  in \eqref{prop_opt_blocks_sel} if  the  condition  in Lemma \ref{lemma:suff_det}  holds, and   with a similar argument as in the case of the exponential decay model  we can show that this is indeed the case  if  
\begin{align*}
I_{x'} - I_{x} \geq  
|q_{x'} - q_{x}| \left(\sum_{i=1}^{\infty} i^{-r} \right) \sup_{\xi \in \bR} \frac{\mwd{\Psi}'(\xi)}{1 - \mwd{\Psi}(\xi)}
\end{align*}
holds for every  $x \in [K]$,  where  again $\mwd{\Psi}'$ is the derivative of $\mwd{\Psi}$.\\

\noindent \textit{Asymptotic optimality.} We note that the statements in Corollary  \ref{cor:exp_cp_asym_opt} continue to hold, if we replace the exponential decay model in \eqref{exponential_decay_CP} by the polynomial decay model in \eqref{polynomial_decay_change_point}.

\subsection{Proof of Lemma \ref{lemma:poly_stability}} \label{app:poly_stability}

\begin{proof}
Fix $\Xi \in \cB_{\ell}$. Let us define
$$
q_{*} \equiv \max_{i \in [K]} q_{i}  \quad \text{ and } \quad 
\xi_t \equiv \sum_{s = 1}^{t} q_{X_s} (t-s+1)^{-r}   \; \; \text{for} \quad t \geq 1.
$$
For any $t_0 \in \bN$ and $(x_1,\ldots, x_{t_0}) \in [K]^{t_0}$,  we set $X_i = x_i$ for $1 \leq i \leq t_0$ and $X_{t_0 +s } = \Xi(s) $ for $s \geq 1$.
Then for $s \geq 0$ and $1 \leq j \leq \ell$,  we have 
\begin{align*}
\xi_{t_0+ s \ell + j} = & \sum_{i=1}^{t_0} q_{x_i} (t_0+ s \ell + j - i + 1)^{-r}  \\
+ & 
\sum_{ i = 1}^{j}  \sum_{\iota = 0}^{s} q_{\Xi(i)} ( s \ell + j -i - \iota \ell +1 )^{-r}
+ 
\sum_{ i = j - \ell +1}^{0}  \sum_{\iota = 1}^{s} q_{\Xi(i)} ( s \ell + j -i - \iota \ell +1 )^{-r}.
\end{align*}
We denote by $\text{I, II, III}$ the three terms on the right hand side. For the first term, since $r > 1$,
\begin{align*}
\left\vert \text{I}
\right\vert & \leq q_{*} \sum_{i=1}^{\infty} (s\ell + j + i)^{-r}\; \to \; 0, \text{ as } s \to \infty.
\end{align*}
Further,
for any $\epsilon > 0$ there exists $C_{\epsilon}$, that does not depend on $t_0$ and 
$(x_1,\ldots,x_{t_0})$, such that
$$
\left\vert \text{ II} + \text{III}  - \chi_j \right\vert < \epsilon \; \; \text{ for any } \, \;
s \geq C_{\epsilon}.
$$
Then, the proof is complete, since $\omega$ and $\mwd{\Psi}$ are continuous.
\end{proof}

\section{Posterior odds process under parameter  uncertainty}\label{proof:post_odds_para}

Fix an arbitrary assignment rule $\cX$ and an integer $t \geq 0$. If the parameter $\eta$ was known,  we can apply
 Lemma \ref{lemma:recursive}. Specifically, by the definition of $\Gamma_{t,\eta}$ in Section \ref{sec:distr_uncertainty}, almost  surely, 
\begin{align*}
\Pro\left(L_t = 1 \, \vert  \, \cF_t, \eta \right)  = {\Gamma_{t,\eta}}/{(1 + \Gamma_{t,\eta})}.
\end{align*}
 Lemma \ref{app:lemma_post_eta} below establishes the $\cF_t$-conditional distribution of $\eta$, which implies that 
\begin{align*}
 \Pro\left(L_t = 1 \, \vert \, \cF_t \right)  & = \frac{\int_{\cH} \Pro\left(L_t = 1 \,  \vert  \, \cF_t, \tilde{\eta} \right) D_{t,\tilde{\eta}} \, (1+\Gamma_{t,\tilde{\eta}}) \, d \tilde{\mu}(\tilde{\eta})}{\int_{\cH}  D_{t,\tilde{\eta}} \, (1+\Gamma_{t,\tilde{\eta}}) \, d \tilde{\mu}(\tilde{\eta})} \\
 & = \frac{\int_{\cH}  D_{t,\tilde{\eta}} \, \Gamma_{t,\tilde{\eta}} \, \ d \tilde{\mu}(\tilde{\eta})}{\int_{\cH}  D_{t,\tilde{\eta}} \, (1+\Gamma_{t,\tilde{\eta}})\ d \tilde{\mu}(\tilde{\eta})}.
\end{align*}
As a result, the posterior odds of $L_t=1$ given $\cF_t$ is
\begin{align*}
\Gamma_t =  \frac{\Pro\left(L_t = 1 \, \vert \,  \cF_t \right)}{1 - \Pro\left(L_t = 1 \, \vert  \, \cF_t \right)} =
\frac{\int_{\cH}  D_{t,\tilde{\eta}} \, \Gamma_{t,\tilde{\eta}}\ d \tilde{\mu}(\tilde{\eta})}{\int_{\cH}  D_{t,\tilde{\eta}} \ d \tilde{\mu}(\tilde{\eta})},
\end{align*}
which completes the proof of Lemma \ref{distr:lemma_postodds}.

Next, we derive the $\cF_t$-conditional distribution of $\eta$.

\begin{lemma}\label{app:lemma_post_eta} 
    Fix an assignment rule $\cX$ and an integer $t \geq 0$. For any measurable subset $C \subset \cH$, almost surely, 
    \begin{align*}
        \Pro(\eta \in C \vert \cF_t) = \frac{\int  D_{t,\tilde{\eta}}(1+\Gamma_{t,\tilde{\eta}})\mathbbm{1}\{\tilde{\eta} \in C\}\ d \tilde{\mu}(\tilde{\eta})}{\int  D_{t,\tilde{\eta}}(1+\Gamma_{t,\tilde{\eta}})\ d \tilde{\mu}(\tilde{\eta})}.
    \end{align*}
\end{lemma}
\begin{proof} 
Fix any measurable subset $C \subset \cH$. As in Appendix \ref{app:poster_odds}, we set $y_{1:t} \equiv  (y_1,\ldots, y_t)$. Since $\cX$ is an assignment rule, there exists a sequence of measurable functions $\{x_j: j \geq 1\}$, such that $X_j = x_j(Y_{1:(j-1)})$. Recall the definitions of $D_{t,\eta}$ and $\Gamma_{t,\eta}$ in Section \ref{sec:distr_uncertainty}.
For each $\eta \in \cH$,
\begin{align*}
D_{t,\eta} = d_{t,\eta}(Y_{1:t}), \text{ with }\;\;
d_{t,\eta}(y_{1:t}) \equiv (1-\pi_{0,\eta}) \prod_{s=1}^{t} \left(1 - \pi_{s,\eta}(x_{1:s})\right) f_{x_s,\eta}(y_s), 
 \end{align*}
 where, for simplicity, we write $x_j$ for $x_j(y_{1:(j-1)})$, with the understanding that they are actually  functions. Similarly, for each $\eta \in \cH$,
 \begin{align*}
     \Gamma_{t,\eta} = \sum_{s=0}^{t} r_t^s(Y_{1:t}; \eta), \text{ with }\;\;
r_t^s(y_{1:t}; \eta) \equiv \pi_{s,\eta}(x_{1:s}) \prod_{j=s}^{t} \frac{g_{x_j,\eta}(y_j)/f_{x_j,\eta}(y_j) }{1 - \pi_{j,\eta}(x_{1:j})},
 \end{align*}
where we use the convention that $g_{x_0,\eta}(y_0)/f_{x_0,\eta}(y_0) = 1$.

For any non-negative, measurable function $u: \bY^{t} \to [0,\infty]$, by definition, 
\begin{align}\label{aux:eta_integral}
\begin{split}
        &\Exp\left[u(Y_{1:t});  \eta \in C\right] 
    =\Exp\left[u(Y_{1:t}); \Theta > t, \eta \in C\right] 
    + \sum_{s=0}^{t} \Exp\left[u(Y_{1:t}); \Theta = s, \eta \in C\right]
    \\
    &=  \int u(y_{1:t}) \,   d_{t,\eta}(y_{1:t}) \left(1 + \sum_{s=0}^{t}  r_t^s(y_{1:t}; \eta) \right) \mathbbm{1}\{\eta \in C\} \; d\mu^{t}(y_{1:t}) \tilde{\mu}(d\eta).
    \end{split}
\end{align}

On the other hand,
\begin{align*}
    I_C \equiv &\Exp\left[u(Y_{1:t}) \frac{\int    D_{t,\tilde{\eta}}(1+\Gamma_{t,\tilde{\eta}})\mathbbm{1}\{\tilde{\eta} \in C\}\ d \tilde{\mu}(\tilde{\eta})}{\int   D_{t,\tilde{\eta}}(1+\Gamma_{t,\tilde{\eta}})\ d \tilde{\mu}(\tilde{\eta})}\right] 
    \\
    =& \Exp\left[u(Y_{1:t}) \frac{\int    d_{t,\tilde{\eta}}(Y_{1:t}) \, (1+\sum_{s=0}^{t}  r_t^s(Y_{1:t}; \tilde{\eta}) ) \, \mathbbm{1}\{\tilde{\eta} \in C\}\ d \tilde{\mu}(\tilde{\eta})}{\int    d_{t,\tilde{\eta}}(Y_{1:t})(1+\sum_{s=0}^{t}  r_t^s(Y_{1:t}; \tilde{\eta}) ) \, d \tilde{\mu}(\tilde{\eta})}\right].
\end{align*}
The random variable inside the expectation is a non-negative function of $Y_{1:t}$. Thus, by the calculation in \eqref{aux:eta_integral}, which is for an arbitrary non-negative function of $Y_{1:t}$,  we have
\begin{align*}
    I_C =& \int \left[
    u(y_{1:t}) \frac{\int    d_{t,\tilde{\eta}}(y_{1:t}) \, (1+\sum_{s=0}^{t}  r_t^s(y_{1:t}; \tilde{\eta}) )\mathbbm{1}\{\tilde{\eta} \in C\}\ d \tilde{\mu}(\tilde{\eta})}{\int    d_{t,\tilde{\eta}}(y_{1:t})(1+\sum_{s=0}^{t}  r_t^s(y_{1:t}; \tilde{\eta}) )\ d \tilde{\mu}(\tilde{\eta})}\right]  \\
    & \qquad d_{t,\eta}(y_{1:t}) \, \left(1 + \sum_{s=0}^{t}  r_t^s(y_{1:t}; \eta) \right) \; d\mu^{t}(y_{1:t}) \; \tilde{\mu}(d\eta).
\end{align*}
Since the term inside the bracket above does not depend on $\eta$, by integrating first with respect to $\eta$ and cancelling the term in the denominator above, we have
\begin{align*}
       I_C = \int u(y_{1:t}) \,   d_{t,\tilde{\eta}}(y_{1:t}) \left(1 + \sum_{s=0}^{t}  r_t^s(y_{1:t}; \tilde{\eta}) \right) \mathbbm{1}\{\tilde{\eta} \in C\} \; d\mu^{t}(y_{1:t}) \, d \tilde{\mu}(\tilde{\eta}).
\end{align*}
Comparing it with  \eqref{aux:eta_integral}, we have
\begin{align*}
    \Exp\left[u(Y_{1:t});  \eta \in C\right]  = I_C.
\end{align*}
Since $u(\cdot)$ is arbitrary, the proof is complete.
\end{proof}

\section{Simulation parameters in the main text}
\subsection{Parameters for the finite memory model} \label{app:geo_values}
We continue the discussion in Subsection \ref{subsec:simulation_geo} and provide the values for $\{f_x\}$ and $\Psi(\cdot)$ for the two cases: (i) $(\kappa = 2, K =4)$ and (ii) $(\kappa = 3, K =6)$. We used $(0.4708, 0.4175, 0.3695, 0.3695)$ for  $f_1$-$f_4$ in (i), and
$$(0.4708, 0.4175, 0.3755, 0.3695, 0.3380, 0.3380)$$ for $f_1$-$f_6$ in (ii). Note that these numbers are randomly selected from $(0.3,0.5)$. The reason that the two cases share some numbers is that we used the same seed to generate them.
For $(i)$, the values ($64$ in total) for $\{\Psi(x,y,z): x,y,z \in [4]\}$ are listed below, where the index for $\Psi(x,y,z)$ is $16(x-1)+4(y-1)+z$:
\begin{align*}
&0.0580, 0.0336, 0.0388, 0.0010, 0.0390, 0.0753, 0.0681, 0.0584, 0.0087, 0.0715, \\&0.0686, 0.0132, 0.0506, 0.0016, 0.0093, 0.0253, 0.0126, 0.0607, 0.0655, 0.0276, \\&0.0255, 0.0089, 0.0067, 0.0570, 0.0480, 0.0045, 0.0384, 0.0321, 0.0678, 0.0574,\\& 0.0482, 0.0442, 0.0759, 0.0789, 0.0270, 0.0192, 0.0637, 0.0051, 0.0292, 0.0056,\\& 0.0255, 0.0056, 0.0232, 0.0632, 0.0724, 0.0634, 0.0449, 0.0493, 0.0289, 0.0135, \\ &0.0349, 0.0586, 0.0050, 0.0017, 0.0616, 0.0240, 0.0561, 0.0588, 0.0746, 0.0320,\\ &0.0287, 0.0645, 0.0612, 0.0522.
\end{align*}
These numbers are randomly picked from $(0,0.08)$. For $(ii)$, since it involves $6^4$ numbers, interested readers may contact the authors for the data.

\subsection{Parameters for the exponential decay model}  \label{app:exp_values} 
We continue the discussion in Subsection \ref{sec:exp_sims}. 
In Table \ref{tab:exponential_parameters}, we provide the values for $q_1, q_2$, as well as $\Dro([i])$ for $i \in [3]$ under different $r$.

\begin{table}[tbp!]
\centering
\caption{Parameter values for different $r$.}
\begin{tabular}{cccccc}
\hline
$r$      & $q_1$     & $q_2$     & $\Dro([1])$ & $\Dro([2])$ & $\Dro([3])$ \\ \hline
0.25 & 1.43 & 0.94 & 0.1530 & 0.2205 & 0.3423 \\ 
0.90 & 0.32 & 0.16 & 0.2235 & 0.2388 & 0.3423 \\ 
0.95 & 0.25 & 0.10  &  0.2434&    0.2637&    0.3423\\ 
1.00 & 0.20 & 0.06 & 0.2436 & 0.3482 & 0.3423 \\ \hline
\end{tabular}
\label{tab:exponential_parameters}
\end{table}

\section{Additional lemmas}
The following lemma is widely known and its proof can be found, e.g., in Theorem 4.1.3 of \cite{durrett2010probability}.

\begin{lemma}\label{lemma:iid_independent}
Let $\{W_{t}, t \geq 0 \}$ be a sequence of independent and identically distributed 
$d$-dimensional random vectors and denote by
$ \cG_{t} = \sigma(W_{s}: 0 \leq s \leq t), t \geq 0 $ its natural filtration.
If $S$ is a finite  $\{\cG_t\}$-stopping time,  then $\{W_{S+t}, t \geq 1\}$
is independent of $\cG_S$ and has the same distribution as $\{W_{t}, t \geq 0\}$.
\end{lemma}

The following result is non-asymptotic, and is essentially due to \cite{lorden1970}.
\begin{lemma}\label{lemma:lorden_renewal}
Fix $\ell \in \bN$  and  let $\{Z_t, t \in \bN\}$ be independent random variables such that $Z_t$ and $Z_{t+\ell}$ have the same distribution for every $t \in \bN$. For $1 \leq i \leq \ell$, assume $\Exp[Z_i^2] < \infty$ and $\Exp[Z_{i}] > 0$, and set
$$
\bar{\mu} \equiv  \ell^{-1} \sum_{i=1}^{\ell} \Exp[Z_i], \quad
\bar{V} \equiv  \ell^{-1} \sum_{i=1}^{\ell} \Exp[\left(Z_i - \Exp[Z_i]\right)^2].
$$
Let $\{S_t \equiv \sum_{s=1}^{t} Z_s, t \geq 1\}$ be the associated random walk and denote by $T(b)$ the first time that $\{S_t\}$ exceeds  some threshold $b$, i.e.
$T(b) = \inf\{t \geq 1: S_t > b\}$.
Then, for any $b > 0$ we have
$$
\Exp[T(b)] \;\leq\; {b}/{\bar{\mu}} + \ell + {\bar{V}}/{\bar{\mu}^2}.
$$ 
\end{lemma}
\begin{proof}
Let $Z'_t \equiv \sum_{s = (t-1)\ell + 1}^{t \ell} Z_s$ and 
$T'(b) \equiv \inf\{t \geq 1: \sum_{s=1}^{t} Z'_s > b\}$. Clearly, by definition, 
$$
T(b) \leq \ell \, T'(b),\quad
\Exp[Z'_1] = \ell \bar{\mu}, \quad
\Exp[\left(Z'_1 -\Exp[Z'_1] \right)^2] =
\ell \, \bar{V} .
$$
Since $\{Z'_t, t\geq 1\}$ are i.i.d., by \cite{lorden1970} we have
$$
\Exp[T(b)] \leq 
\ell \, \Exp[T'(b)]
\leq \ell \left(
\frac{b}{\ell \bar{\mu}} + 1 + \frac{\ell \bar{V}}{(\ell \bar{\mu})^2}
\right),
$$
which completes the proof.
\end{proof}

The following lemma 
follows directly from Wald's likelihood ratio identity  \citep{wald1945sequential}. 
\begin{lemma}\label{lemma:one_sided_SPRT}
Let $\{f_t: t \in \bN\}$ and $\{g_t: t \in \bN \}$ be two sequences of densities on the measurable space $(\bY,\cB(\bY))$ with respect  to some $\sigma$-finite measure $\mu$, and let
 $\{Y_t, t \in \bN\}$ be a sequence of independent, $\bY$-valued  random elements such that  $Y_t$ has a density $g_t$. Further, define for any $d > 0$,
$$
\tau(d) \equiv \inf\left\{t \in \bN: \prod_{s=1}^{t} \frac{f_s(Y_s)}{g_s(Y_s)} \geq d \right\}.
$$
Then, $\Pro(\tau(d) < \infty) \leq 1/d$.
\end{lemma}

\end{appendix}

\end{document}